\documentclass[11pt,reqno]{amsart}

\usepackage{amsmath, amsfonts, amsthm, amssymb,graphicx, color, enumitem}
\newtheorem{theorem}{Theorem}[section]
\newtheorem{lemma}{Lemma}[section]

\theoremstyle{definition}
\newtheorem{Definition}{Definition}[section]

\theoremstyle{remark}
\newtheorem{remark}{Remark}[section]

\numberwithin{equation}{section}
\allowdisplaybreaks

\newcommand{\R}{{\mathbb R}}

\def\f{\frac}

\def\hf1{^\f{1}{1-\xi^2}}

\def\be{\begin{equation}}
\def\en{\end{equation}}
\def\bs{\begin{split}}
\def\es{\end{split}}
\def\ba{\begin{align}}
\def\ea{\end{align}}

\author[Alessandro Morando]{Alessandro Morando}
\address{DICATAM, University of Brescia, Via Valotti 9, 25133 Brescia, Italy}
\email{alessandro.morando@unibs.it}

\author[Paolo Secchi]{Paolo Secchi}
\address{DICATAM, University of Brescia, Via Valotti 9, 25133 Brescia, Italy}
\email{paolo.secchi@unibs.it}

\author[Paola Trebeschi]{Paola Trebeschi}
\address{DICATAM, University of Brescia, Via Valotti 9, 25133 Brescia, Italy}
\email{paola.trebeschi@unibs.it}

\author[Difan Yuan]{Difan Yuan}
\address{School of Mathematical Sciences,  Beijing Normal University and Laboratory of Mathematics and Complex Systems, Ministry of Education, Beijing 100875, China.}
\email{yuandf@amss.ac.cn}

\title[Compressible Current-Vortex Sheets in two-dimensional MHD]
{Nonlinear stability and existence of two-dimensional compressible current-vortex sheets}

\date{}

\begin{document}
\begin{abstract}
We are concerned with nonlinear stability and existence of two-dimensional current-vortex sheets in ideal compressible magnetohydrodynamics. This is a nonlinear hyperbolic initial-boundary value problem with characteristic free boundary. It is well-known that current-vortex sheets may be at most weakly (neutrally) stable due to the existence of surface waves solutions that yield a loss of derivatives in the energy estimate of the solution with respect to the source terms. We first identify a sufficient condition ensuring the weak stability of the linearized current-vortex sheets problem. Under this stability condition for the background state, we show that the linearized problem obeys an energy estimate in anisotropic weighted Sobolev spaces with a loss of derivatives. Based on the weakly linear stability results, we then establish the local-in-time existence and nonlinear stability of current-vortex sheets by a suitable Nash-Moser iteration, provided the stability condition is satisfied at each point of the initial discontinuity. This result gives a new confirmation of the stabilizing effect of sufficiently strong magnetic fields on Kelvin-Helmholtz instabilities.

\end{abstract}
\maketitle

\section{Introduction}\label{intro}

In this paper, we are focusing on two-dimensional ideal compressible magnetohydrodynamics (MHD)(See \cite{Dafermos,Goedbloed,Laudau}):
\begin{equation}\label{MHD}
\begin{cases}
\partial_t\rho+\mathrm{div}(\rho \mathbf{u})=0,\\
\partial_t(\rho \mathbf{u})+\mathrm{div}(\rho \mathbf{u}\otimes \mathbf{u}-\mathbf{H}\otimes \mathbf{H})+\nabla (p+\frac{1}{2}|\mathbf{H}|^2)=0,\\
\partial_t\mathbf{H}-\nabla\times(\mathbf{u}\times \mathbf{H})=0,\\
\partial_t(\rho e+\frac{1}{2}|\mathbf{H}|^2)+ \mathrm{div}((\rho e+p)\mathbf{u}+\mathbf{H}\times(\mathbf{u}\times \mathbf{H}))=0,\\
\end{cases}
\end{equation}
supplemented with
\begin{equation}\label{H}
\mathrm{div}\mathbf{H}=0
\end{equation}
on the initial data for the Cauchy problem.
Here density $\rho,$ velocity $\mathbf{u}=(u_1,u_2)$, magnetic field $\mathbf{H}=(H_1,H_2)$ and pressure $p$ are unknown functions of time $t$ and spacial variables $\mathbf{x}=(x_1,x_2)$. $p=p(\rho,S)$ stands for the pressure, $S$ is the entropy, $e=E+\frac{1}{2}|\mathbf{u}|^2$, where $E=E(\rho,S)$ stands for the internal energy. By using the state equation of gas, $\rho=\rho(p,S),$ and the first principle of thermodynamics, we have that \eqref{MHD} is a closed system.
We write ${\bf{U}}={\bf{U}}(t,\mathbf{x})=(p,\mathbf{u},\mathbf{H},S)^T,$ with initial data  ${\bf{U}}(0,\mathbf{x})={\bf{U}}_0(\mathbf{x}).$
By \eqref{H}, we can rewrite \eqref{MHD} into the following form
\begin{equation}\label{MHD2}
\begin{cases}
(\partial_t+\mathbf{u}\cdot\nabla)p+\rho c^2\mathrm{div}\mathbf{u}=0,\\
\rho(\partial_t+\mathbf{u}\cdot\nabla)\mathbf{u}-(\mathbf{H}\cdot\nabla)\mathbf{H}+\nabla q=0,\\
(\partial_t+\mathbf{u}\cdot \nabla) \mathbf{H}-(\mathbf{H}\cdot\nabla)\mathbf{u}+\mathbf{H} \mathrm{div}\mathbf{u}=0,\\
(\partial_t+\mathbf{u}\cdot\nabla)S=0,\\
\end{cases}
\end{equation}
where $q=p+\frac{1}{2}|\mathbf{H}|^2$ represents the total pressure, $c$ denotes the speed of sound and $\frac{1}{c^2}:=\frac{\partial\rho}{\partial p}(\rho,S)=\rho_p(p,S)$.
\eqref{MHD2} can be written in the matrix form as
\begin{equation}\label{quasi}
A_0({\bf{U}})\partial_t{\bf{U}}+A_1({\bf{U}})\partial_1{\bf{U}}+A_2({\bf{U}})\partial_2{\bf{U}}=0,
\end{equation}
where
\begin{equation}\nonumber
A_0({\bf{U}}):=\mathrm{diag}\Big\{\frac{1}{\rho c^2},\rho,\rho,1,1,1\Big\},
\end{equation}
\begin{equation}\label{Am}
\begin{split}  %Â¿ÂªÃÅÃÃœÃÂ§Â»Â·ÅžÂ³
 &A_1({\bf{U}}):=\begin{bmatrix}%\begin{smallmatrix}             %ÃÃ³ÃÅ¡ÂºÃ
    \frac{u_1}{\rho c^2} & 1 & 0 & 0 & 0 & 0  \\
   1 & \rho u_1 & 0 & 0 & H_2  & 0 \\
    0 & 0 & \rho u_1 & 0 & -H_1 & 0 \\
    0 & 0 & 0 & u_1 & 0 & 0   \\
    0 & H_2 & -H_1 & 0 & u_1 & 0  \\
    0 & 0 & 0 & 0 & 0  & u_1  \\
%\end{smallmatrix}
\end{bmatrix},\\
& A_2({\bf{U}}):=\begin{bmatrix}%\begin{smallmatrix}             %ÃÃ³ÃÅ¡ÂºÃ
    \frac{u_2}{\rho c^2} & 0 & 1 & 0 & 0 & 0  \\
   0 & \rho u_2 & 0 & -H_2 & 0  & 0 \\
    1 & 0 & \rho u_2 & H_1 & 0 & 0 \\
    0 & -H_2 & H_1 & u_2 & 0 & 0   \\
    0 & 0&0 & 0 & u_2 & 0  \\
    0 & 0 & 0 & 0 & 0  & u_2  \\
%\end{smallmatrix}
\end{bmatrix}
.\\
\end{split}
\end{equation}
The quasilinear system \eqref{quasi} is symmetric hyperbolic if the state equation $\rho=\rho(p,S)$ satisfies
the hyperbolicity condition $A_0>0:$
\begin{equation}\label{hyper}
\rho(p,S)>0,\quad \rho_p(\rho,S)>0.
\end{equation}
We suppose that $\Gamma:=\bigcup\limits_{t\in[0,T]}\Gamma(t)$, where $\Gamma(t):=\{{\bf x}\in\mathbb R^2\,:\,\,x_1-\varphi(t,x_2)=0\}$, is a smooth hypersurface in $[0,T]\times \R^2.$
The weak solutions of \eqref{MHD} satisfy the Rankine-Hugoniot jump conditions on $\Gamma(t)$:
\begin{equation}\label{RH}
\begin{cases}
&[j]=0,\quad [H_N]=0,\\
& j[u_N]+\vert N\vert^2[q]=0,\quad j[u_{\tau}]=H_N[H_{\tau}],\\
&H_N[u_{\tau}]=j[\frac{H_{\tau}}{\rho}],\\ &j[e+\frac{1}{2}\frac{|\mathbf{H}|^2}{\rho}]+[qu_N-H_N(\mathbf{H}\cdot\mathbf{u})]=0.
\end{cases}
\end{equation}
Here we denote $[\upsilon]=(\upsilon^+-\upsilon^-)|_{\Gamma}$ the jump of $\upsilon$,  with $\upsilon^{\pm}:=\upsilon$ in $\Omega^{\pm}(t)=\{\pm(x_1-\varphi(t,x_2))>0\},$ and $j=\rho(u_N-\partial_t\varphi)$ is the mass transfer flux across the discontinuity surface. We also denote the tangential and normal components of velocity and magnetic fields   $u_{\tau}=\mathbf{u}\cdot\tau,H_{\tau}=\mathbf{H}\cdot\tau,$ $u_N=\mathbf{u}\cdot N, H_N=\mathbf{H}\cdot N,$ where $N=(1,-\partial_2\varphi),\tau=(\partial_2\varphi,1).$ \\

We are focusing on the current-vortex sheets solutions, which obey the following additional conditions along the interface $\Gamma$:
\begin{equation}\nonumber
j^{\pm}|_{\Gamma}=0,\quad H^{\pm}_N|_{\Gamma}=0.
\end{equation}
Then, the Rankine-Hugoniot conditions reduce to the following boundary conditions
\begin{equation}\label{boundary}
\partial_t\varphi=u^{\pm}_{N},\quad H^{\pm}_{N}=0,\quad[q]=0
\end{equation}
on $\Gamma(t).$

The  tangential components of the velocity and the magnetic field may undergo any jumps: $[u_{\tau}]\neq0,\quad [H_{\tau}]\neq0.$
The initial data are given as follows:
\begin{equation}\label{initial}
\mathbf{U}^{\pm}(0,\mathbf{x})=\mathbf{U}^{\pm}_0(\mathbf{x}),\quad \mathbf{x}\in \Omega^{\pm}(0),\quad \varphi(0,x_2)=\varphi_0(x_2),\quad x_2\in \R.
\end{equation}

It is obvious that there exist trivial vortex sheets (contact discontinuity) solutions consisting of two constant states separated by a flat surface as follows:
\begin{equation}\label{constant}
\bar{\mathbf{U}}(\mathbf{x}):=
\begin{cases}
\bar{\mathbf{U}}^+:=(\bar{p}^+,0,\bar{u}^+_2,0,\bar{H}^+_2,\bar{S}^+), & \text{ if } x_1>0,\\
\bar{\mathbf{U}}^-:=(\bar{p}^-,0,\bar{u}^-_2,0,\bar{H}^-_2,\bar{S}^-), & \text{ if } x_1<0,\\
\end{cases}
\end{equation}
where on account of \eqref{boundary}, we require that $$\bar{p}^+\neq\bar{p}^-,\quad\bar{u}^+_2\neq\bar{u}^-_2,\quad\bar{H}^+_2\neq\bar{H}^-_2,\quad\bar{S}^+_2\neq\bar{S}^-_2,\quad \bar{p}^++\frac{1}{2}|\bar{H}^+_2|^2=\bar{p}^-+\frac{1}{2}|\bar{H}^-_2|^2.$$

Compressible vortex sheets are fundamental waves in the study of entropy solutions to multidimensional hyperbolic conservation laws, arising in many important physical phenomena. For the two-dimensional compressible flows governed by the Euler equations, Fejer and Miles \cite{Fejer1963,Miles1957,Miles1958}, see also \cite{Coulombel2004-2,Serre}, proved that vortex sheets are unstable when $M<\sqrt{2}$, while the vortex sheets in three-dimensional
Euler flows are always violently unstable (the violent instability is the analogue of the Kelvin-Helmholtz instability for incompressible fluids).

In their pioneer works \cite{Coulombel2004,Coulombel2008} Coulombel and Secchi  proved the nonlinear stability and existence of two-dimensional vortex sheets when the Mach number $M>\sqrt{2}$.
The linear and nonlinear stability of vortex sheets has also been established in \cite{Morando2008,MTW2018} for two-dimensional nonisentropic Euler flows,  in \cite{ChenG2017} for two-dimensional relativistic fluids,  in \cite{WangYJDE2013,WangYSIAM2015} for three-dimensional steady Euler flows.

For the three-dimensional compressible flows, various stabilizing effects on vortex sheets have been considered.
Taking into account the effect of magnetic fields, Chen-Wang \cite{ChenG2008,ChenG2012} and Trakhinin \cite{Trakhinin2005,Trakhinin2009} proved that large non-parallel magnetic fields stabilize the motion of three dimensional current-vortex sheets. Elasticity can also provide stabilization of vortex sheets: Chen-Hu-Wang \cite{RChen2017,RChen2018} and Chen-Hu-Wang-Wang-Yuan \cite{RChen2020} successfully proved the linear stability and nonlinear stability, respectively,  of two-dimensional compressible vortex sheets in elastodynamics by introducing the upper-triangularization method. More recently, Chen-Huang-Wang-Yuan \cite{RChen2021} confirmed the stabilizing effect of elasticity also on three dimensional compressible vortex sheets.
Another stabilizing effect on vortex sheets is provided by surface tension: for the three-dimensional compressible Euler flows, the local existence and structural stability  were proved in Stevens \cite{Stevens}.

The analysis of three-dimensional current-vortex sheets in \cite{ChenG2008,ChenG2012,Trakhinin2005,Trakhinin2009} does not cover the two dimensional case.
In fact, such a case can be considered as the case when the third component of magnetic field $\mathbf{H}=(H_1,H_2,H_3)$ is zero, i.e. $H_3=0,$  and therefore the non-collinear stability conditions in \cite{ChenG2008,Trakhinin2005,Trakhinin2009} fail. As was shown in Trakhinin \cite{Trakhinin2005}, the case when either tangential magnetic fields are collinear or one of them is zero corresponds to the transition to violent instability.

For the two-dimensional compressible flows, Wang and Yu \cite{WangYARMA2013} proved the linear stability of rectilinear current-vortex sheets under suitable stability conditions by the spectral analysis technique, through the computation of the roots of the Lopatinski determinant and construction of a Kreiss symmetrizer.

In the present paper we investigate the nonlinear stability and existence of two dimensional current-vortex sheets. From the mathematical point of view, this is a nonlinear hyperbolic free boundary problem. Since the Kreiss--Lopatinski condition does not hold uniformly, there is a loss of tangential derivatives in the estimates of the solution. The free boundary is characteristic, which yields a possible loss of regularity in the normal direction to the boundary, and the loss of control of the traces of the characteristic part of the solution.

Differently from \cite{WangYARMA2013}, we consider the nonisentropic flows and for the analysis of linear stability, instead of spectral analysis, we use a direct energy estimate argument, adapting the dissipative symmetrizer technique introduced by Chen-Wang \cite{ChenG2008,ChenG2012} and Trakhinin \cite{Trakhinin2005,Trakhinin2009} for the three-dimensional problem. Moreover, in our linear stability result we study a general case of $2D$ current-vortex sheets, while in \cite{WangYARMA2013} both states of the background magnetic field are assumed to have the same strength, see the subsequent Remark \ref{rhs-stability cond}.

First, we introduce a secondary symmetrization of the system of equations by multiplying by a suitable ``secondary generalized Friedrichs symmetrizer'' and impose the hyperbolicity of the new system of equations. Then,  we identify a sufficient stability condition that makes the boundary conditions dissipative for the new symmetrized system.
The new stability condition on the boundary takes the form
\begin{equation}\label{intro-1.11}
\frac{c_A^+}{\sqrt{1+\left(\frac{c^{+}_A}{c^{+}}\right)^2}}  +\frac{c_A^-}{\sqrt{1+\left(\frac{c^{-}_A}{c^{-}}\right)^2}}-|[ \mathbf{u}\cdot\tau]|>0\,,
\end{equation}
where
%\begin{equation*}
%a^\pm=\sqrt{1+c^{\pm}_A)^2(c^{\pm})^2}
%\end{equation*}
$c^{\pm}_A={\vert\mathbf{H}^{\pm}\vert}/{\sqrt{\rho^{\pm}}}$ stands for the Alfv\'en speed.
This condition indicates that larger magnetic fields than the jump of tangential velocity play a stabilization effect; in some sense this corresponds to the ``subsonic'' bubble in linear stability result of \cite{WangYARMA2013}, see Remark \ref{rhs-stability cond}. Condition \eqref{intro-1.11} is in agreement with the stability result found for the three-dimensional current-vortex sheets of \cite{ChenG2008,Trakhinin2005,Trakhinin2009} that only holds in the subsonic regime.

We observe that condition \eqref{intro-1.11} has a strong similarity with the Syrovatskij stability condition \cite{axford,michael,syrovatskii} that is necessary and sufficient for the linear stability of the two-dimensional {\em incompressible} current-vortex sheets. For our problem we show that condition \eqref{intro-1.11} is sufficient for the linear stability and optimal with respect to the specific dissipative symmetrizer technique that we use in our proof. On the other hand, it is likely that \eqref{intro-1.11} is not necessary for the linear and nonlinear stability. Indeed, by taking the incompressible limit as $c^\pm\to +\infty$ in \eqref{intro-1.11}, we formally get the inequality
$$
\vert{\mathbf H}^+\vert/\sqrt{\bar{\rho}}+\vert{\mathbf H}^-\vert/\sqrt{\bar{\rho}}-\vert[{\mathbf u}\cdot\tau]\vert>0\,,
$$
where, because the incompressible flow has uniform density, $\rho^\pm$ have been replaced by a constant $\bar{\rho}>0$. This inequality describes somehow the ``half'' of the whole $2D$ neutral stability domain from \cite{michael}. Moreover, in Wang-Yu \cite{WangYARMA2013} even a ``supersonic'' linear stability domain is found for the studied case of particular piece-wise constant background states; the same kind of ``supersonic'' region is therefore expected to appear for a general case of $2D$ current-vortex sheet, see again Remark \ref{rhs-stability cond}.

Under condition \eqref{intro-1.11} for the background state, we show that the linearized problem obeys an energy estimate in anisotropic weighted Sobolev spaces (see \cite{ChenS,Ohno,secchi96, Secchi1996}) with a loss of derivatives.
The energy estimate for the linearized problem takes the form of a $\it{tame}$ estimate, since it exhibits a fixed loss of derivatives from the basic state to the solutions.
In order to compensate the loss of derivatives, for the proof of the existence of the solution to the nonlinear problem, we apply a modified Nash-Moser iteration scheme. For an introduction to the Nash-Moser technique, refer to \cite{Alinhac2,Secchi2016}.

Compared to \cite{ChenG2008, Trakhinin2009} for the three-dimensional problem, our existence result shows a lower loss of regularity from the initial data to the solution. This is mainly due to the use of finer Moser-type  and imbedding estimates in anisotropic Sobolev spaces.

The rest of the paper is organized as follows. In Section \ref{straighten} we formulate the nonlinear problem for the current-vortex sheets in a fixed domain. In Section \ref{functional spaces} we introduce the definition of anisotropic Sobolev spaces and then state our main Theorem \ref{maintheorem}. In Section \ref{linearized} we linearize the nonlinear problem with respect to the basic state. Then, we introduce dissipative Friedrichs symmetrizer for two dimensional MHD equations. In Section \ref{wellposedness} we study the well-posedness of the linearized problems \eqref{IVP2}, \eqref{ho2} and determine the stability condition. In Sections \ref{tameestimate} and \ref{tameestimate2} we prove the tame estimate in anisotropic Sobolev spaces for the linearized problems \eqref{ho2} and \eqref{IVP2}, respectively. In Section \ref{compa} we formulate the compatibility conditions for the initial data and construct the approximate solution. In Section \ref{nash} we introduce the Nash-Moser iteration scheme and in Section \ref{proof} we  prove the existence of the solution to the nonlinear problem.
In this section one important step is the new construction of the modified state, notably the modified magnetic field, and the delicate derivation of its  estimates.
In Appendix \ref{tracetheorem} we recall the trace theorem in anisotropic Sobolev space $H^m_{\ast}$ and in Appendix \ref{proof-Thm-5-1} we give a detailed proof of the well-posedness of the homogeneous linearized problem \eqref{ho2} stated in Theorem \ref{th-wp-hom}.

\bigskip

\section{Reformulate Current-Vortex Sheets Problem in a Fixed Domain}\label{straighten}

Let us reformulate the current-vortex sheets problem into an equivalent one posed in a fixed domain.
Motivated by M\'etivier \cite{Metivier}, we introduce the functions
$$\Phi^{\pm}(t,\mathbf{x}):=\Phi(t,\pm x_1,x_2)=\pm x_1+\Psi^{\pm}(t,\mathbf{x}),\quad \Psi^{\pm}(t,\mathbf{x}):=\chi(\pm x_1)\varphi(t,x_2),$$
where $\chi\in C^{\infty}_0(\R),\quad \chi\equiv1 \text{ on } [-1,1],$ and $||\chi'||_{L^{\infty}(\R)}\leq\frac{1}{2}$. Here, as in \cite{Metivier}, we use the cut-off function $\chi$ to avoid assumptions about compact support of the initial data in our subsequent nonlinear existence Theorem \ref{maintheorem}. The unknowns ${\mathbf U}^{\pm}$ are smooth in $\Omega^{\pm}(t)$ and can be replaced by the following
$${\mathbf U}^{\pm}_{\natural}(t,\mathbf{x}):={\mathbf U}^{\pm}(t,\Phi(t,\pm x_1,x_2),x_2)\,,$$
after changing the variables, which are smooth in the fixed domain $\Omega=\R^2_{+}=\{x_1>0, x_2\in \R\}$.
Dropping $\natural$ in ${\mathbf U}^{\pm}_{\natural}$ for convenience,
we reduce \eqref{MHD}, \eqref{boundary} and \eqref{initial} into the following initial boundary value problem (IBVP)
\begin{equation}\label{IVP}
\begin{cases}
\mathbb{L}({\mathbf U}^{\pm},\Psi^{\pm})=0 & \text{in } [0,T]\times \R^2_+,\\
\mathbb{B}({\mathbf U}^{+},{\mathbf U}^{-},\varphi)=0 & \text{on }[0,T]\times \Gamma,\\
{\mathbf U}^{\pm}(0,\mathbf{x})={\mathbf U}^{\pm}_0 & \text{ in } \R^2_+, \\
\varphi(0,x_2)=\varphi_0 & \text{ in }\R,
\end{cases}
\end{equation}
where $\Gamma:=\{x_1=0\}\times \R,$ $\mathbb{L}({\mathbf U}^\pm,\Psi^\pm)=L({\mathbf U}^\pm,\Psi^\pm){\mathbf U}^\pm$,
\begin{equation}\label{opL}
L({\mathbf U}^\pm,\Psi^\pm)=A_0({\mathbf U}^\pm)\partial_t+\tilde{A}_1({\mathbf U}^\pm,\Psi^\pm)\partial_1+A_2({\mathbf U}^\pm)\partial_2\,,
\end{equation}
with
\begin{equation}\label{A11}
\tilde{A}_1({\mathbf U}^{\pm},\Psi^{\pm})=\frac{1}{\partial_1\Phi^{\pm}}\Big(A_1({\mathbf U}^{\pm})-A_0({\mathbf U}^{\pm})\partial_t\Psi^{\pm}-A_2({\mathbf U}^{\pm})\partial_2\Psi^{\pm}\Big),
\end{equation}
where $\partial_1\Phi^{\pm}=\pm 1+\partial_1\Psi^{\pm}.$
The boundary condition in \eqref{IVP} takes the following form
$$\partial_t\varphi=u^{\pm}_{N},\quad [q]=q^+-q^-=0 \text{ on } [0,T]\times \{x_1=0\}\times \R,$$
where $u^{\pm}_{N}=u^{\pm}_1-u^{\pm}_2\partial_2\varphi$.
\newline
The following Lemma \ref{i3} yields that the divergence constraints \eqref{H} and the boundary conditions $H^{\pm}_N|_{x_1=0}=0$ on $\Gamma$ (that is not included in \eqref{IVP}) can be regarded as the conditions on the initial data. The proof follows by similar calculations as in \cite[Proposition 1]{Trakhinin2009}.
\begin{lemma}\label{i3}
Let the initial data in \eqref{IVP} satisfy the following
\begin{equation}\label{div}
\mathrm{div}\mathbf{h}^{\pm}=0,\quad\mbox{in}\,\,\,\mathbb R^2_+
\end{equation}
and the boundary conditions
\begin{equation}\label{Hnn}
H^{\pm}_{N}|_{x_1=0}=0,\quad\mbox{on}\,\,\,\{x_1=0\}\times\mathbb R\,,
\end{equation}
where
\begin{equation}\nonumber
\mathbf{h}^{\pm}=(H^{\pm}_n, H^{\pm}_2\partial_1\Phi^{\pm}),\quad H^{\pm}_n=H^{\pm}_1-H^{\pm}_2\partial_2\Psi^{\pm},\quad H^{\pm}_N|_{x_1=0}=H^{\pm}_n|_{x_1=0}.
\end{equation}
If the IBVP \eqref{IVP} admits a solution $({\mathbf U}^{\pm},\varphi),$ then this solution satisfies \eqref{div} and \eqref{Hnn} for all $t\in[0,T].$
\end{lemma}

\section{Properties of Function Spaces and Main Theorem}\label{functional spaces}

In this section, we first introduce some notations, then anisotropic Sobolev spaces are defined. At the end, we are ready to state the main result of this paper.
\subsection{Notations}\label{notations}

Let us denote $\Omega_T:=(-\infty,T)\times\Omega$ and $\Gamma_T:=(-\infty,T)\times\Gamma$ for $T>0.$ We write $\partial_t=\frac{\partial}{\partial t}, \partial_i=\frac{\partial}{\partial x_i},i=1,2, \nabla_{t,\mathbf{x}}=(\partial_t,\nabla).$ $D^{\alpha}_{\ast}:=\partial^{\alpha_0}_t(\sigma\partial_1)^{\alpha_1}\partial^{\alpha_2}_2\partial^{\alpha_3}_1, \alpha:=(\alpha_0,\alpha_1,\alpha_2,\alpha_3),\quad |\alpha|=\alpha_0+\alpha_1+\alpha_2+\alpha_3.$ Here $\sigma$ is an increasing smooth function, which satisfies  $\sigma(x_1)=x_1$ for $0\leq x_1\leq\frac{1}{2}$ and $\sigma(x_1)=1$ for $x_1\geq1.$ The symbol $A\lesssim B$ represents that $A\leq CB$ holds uniformly for some universal positive constant $C.$

\subsection{Anisotropic Sobolev spaces}

 For any integer $m\in \mathbb{N}$ and the interval $I\subseteq \R,$ function spaces $H^m_{\ast}(\Omega)$ and $H^{m}_{\ast}(I\times\Omega)$ are defined by
$$H^m_{\ast}(\Omega):=\{u\in L^2(\Omega): D^{\alpha}_{\ast}u\in L^2(\Omega) \text{ for } \langle\alpha\rangle:=|\alpha|+\alpha_3\leq m, \text{ with } \alpha_0=0\},$$
$$H^m_{\ast}(I\times\Omega):=\{u\in L^2(I\times\Omega): D^\alpha_{\ast}u\in L^2(I\times\Omega) \text{ for }\langle\alpha\rangle:=|\alpha|+\alpha_3\leq m\},$$
and equipped with the norm $||\cdot||_{H^m_{\ast}(\Omega)}$ and $||\cdot||_{H^m_{\ast}(I\times\Omega)}$ respectively, where
\begin{equation}\label{norm1}
||u||^2_{H^m_{\ast}(\Omega)}:=\sum_{\langle\alpha\rangle\leq m,\alpha_0=0}||D^{\alpha}_{\ast}u||^2_{L^2(\Omega)},
\end{equation}
\begin{equation}\label{norm2}
||u||^2_{H^m_{\ast}(I\times\Omega)}:=\sum_{\langle\alpha\rangle\leq m}||D^{\alpha}_{\ast}u||^2_{L^2(I\times\Omega)}.
\end{equation}
Define the norm
\begin{equation}\label{norma-equiv}
|||u(t)|||^2_{m,\ast}:=\sum^m_{j=0}||\partial^j_tu(t)||^2_{H^{m-j}_{\ast}(\Omega)}\,.
\end{equation}
We also write $||\cdot||_{m,\ast,t}:=||u||_{H^m_{\ast}(\Omega_t)}$ for convenience. Then, from \eqref{norm2}, we have
\begin{equation}\label{norma-int}
||u||^2_{m,\ast,t}=\int^t_{-\infty}|||u(s)|||^2_{m,\ast}ds\,.
\end{equation}

\subsection{Moser-type calculus inequalities}

Now, we introduce two lemmata, which will be useful in the proof of tame estimates in $H^m_{\ast}(\Omega_T)$ for the problem \eqref{IVP2} when $m$ is large enough. We first introduce the Moser-type calculus inequalities in $H^m,$ see \cite[Propositions 2.1.2 and 2.2]{Alinhac2}.

\begin{lemma}\label{moser1}
Let $m\in \mathbb{N}_+.$ Let $\mathcal{O}$ be an open subset of $\R^n$ with Lipschitz boundary. Assume that $F\in C^{\infty}$ in a neighbourhood of the origin with $F(0)=0$ and that $u,v\in H^m(\mathcal{O})\cap L^\infty(\mathcal{O})$. Then,
$$||\partial^{\alpha}u\,\partial^{\beta}v||_{L^2}\lesssim ||u||_{H^m}||v||_{L^{\infty}}+||u||_{L^{\infty}}||v||_{H^m},$$
$$||F(u)||_{H^m}\leq C||u||_{H^m},$$
for all $\alpha,\beta\in \mathbb{N}^n$ with $|\alpha|+|\beta|\leq m$ and where $C$ depends only on $F$ and $\Vert u\Vert_{L^\infty}$.
\end{lemma}

Next, we introduce the Moser-type calculus inequalities for $H^m_{\ast}$.
\newline
Let us define the space
$$
W^{1,\infty}_\ast(\Omega_T):=\{u\in L^\infty(\Omega_T)\,:\,\,\,D^\alpha_\ast u\in L^\infty(\Omega_T)\,,\,\,\langle\alpha\rangle\le 1\}\,,
$$
equipped with the natural norm
\begin{equation}\label{normaW1inf}
||u||_{W^{1,\infty}_{\ast}(\Omega_T)}:=\sum_{\langle\alpha\rangle\leq1}||D^{\alpha}_{\ast}u||_{L^{\infty}(\Omega_T)}\,,
\end{equation}
where $D^{\alpha}_{\ast}$ and $\langle\alpha\rangle$ are defined in Section \ref{notations}.
\begin{lemma}[\cite{MTP2009,Trakhinin2009}]\label{moser2}
Let $m\in \mathbb{N}_+.$ Assume that $F$ is a $C^{\infty}-$function and $u,v\in H^m_{\ast}(\Omega_T)\cap W^{1,\infty}(\Omega_T)$. Then, there hold
\begin{equation}\label{moser3}
||D^{\alpha}_{\ast}uD^{\beta}_{\ast}v||_{L^2(\Omega_T)}\lesssim ||u||_{m,\ast,T}||v||_{W^{1,\infty}_{\ast}(\Omega_T)}+||v||_{m,\ast,T}||u||_{W^{1,\infty}_{\ast}(\Omega_T)},
\end{equation}
\begin{equation}\label{moser4}
||uv||_{m,\ast,T}\lesssim ||u||_{m,\ast,T}||v||_{W^{1,\infty}_{\ast}(\Omega_T)}+||v||_{m,\ast,T}||u||_{W^{1,\infty}_{\ast}(\Omega_T)},
\end{equation}
%\begin{equation}\label{moser5}
%||F(u)||_{m,\ast,T}\leq C(M_{\ast})(1+||u||_{m,\ast,T}),
%\end{equation}
for any multi-index $\alpha,\beta\in \mathbb{N}^4$ with $\langle\alpha\rangle+\langle\beta\rangle\leq m.$ Let $M_\ast$ be a positive constant such that
\[
||u||_{W^{1,\infty}_{\ast}(\Omega_T)}\le M_\ast\,.
\]
Moreover, if we assume that $F(0)=0$, then there holds
\begin{equation}\label{moser6}
||F(u)||_{m,\ast,T}\leq C(M_{\ast})||u||_{m,\ast,T}\,.
\end{equation}
\end{lemma}
For the proof of \eqref{moser3} and \eqref{moser4}, one can check \cite{MTP2009}. For  \eqref{moser6} one can check \cite{Trakhinin2009}.

\subsection{Embedding and trace theorem}

Now, we introduce the Sobolev embedding theorem for $H^m_{\ast}(\Omega_T)$.
\begin{lemma}[\cite{MTP2009}]
The following inequalities hold:
\begin{equation}\label{sobolev1}
||u||_{L^{\infty}(\Omega_T)}\lesssim||u||_{3,\ast,T},\quad ||u||_{W^{1,\infty}_{\ast}(\Omega_T)}\lesssim||u||_{4,\ast,T},
\end{equation}
\begin{equation}\label{sobolev2}
||u||_{W^{1,\infty}(\Omega_T)}\lesssim||u||_{5,\ast,T},\quad ||u||_{W^{2,\infty}_{\ast}(\Omega_T)}\lesssim||u||_{6,\ast,T},
\end{equation}
where $||u||_{W^{1,\infty}_{\ast}(\Omega_T)}$ is defined by \eqref{normaW1inf} and
$$||u||_{W^{2,\infty}_{\ast}(\Omega_T)}:=\sum_{\langle\alpha\rangle\leq1}||D^{\alpha}_{\ast}u||_{W^{1,\infty}(\Omega_T)}.$$
\end{lemma}
\begin{proof}
Using \cite[Theorem B.4]{MTP2009}, we obtain the first inequality in \eqref{sobolev1} in $\Omega_T\subseteq \R^3.$  Then, the second one in \eqref{sobolev1} can be obtained by definition. Observing that $$||u||_{W^{1,\infty}(\Omega_T)}\lesssim\sum_{\langle\alpha\rangle\leq2}||D^{\alpha}_{\ast}u||_{L^{\infty}(\Omega_T)},$$
we can obtain the first inequality in \eqref{sobolev2} from the first inequality in \eqref{sobolev1}. Similarly the second one in \eqref{sobolev2} can be obtained by definition.
\end{proof}

For higher order energy estimate, we also need to use the following trace theorem by Ohno, Shizuta, Yanagisawa \cite{Ohno-Shizuta1994} for the anisotropic Sobolev spaces $H^m_{\ast},$ see also Appendix \ref{tracetheorem}.
\begin{lemma}[\cite{Ohno-Shizuta1994}]\label{tra}
Let $m\geq1$ be an integer. Then, the following arguments hold:\\
(i) If $u\in H^{m+1}_{\ast}(\Omega_T),$ then its trace $u|_{x_1=0}$ belongs to $H^m(\Gamma_T),$ and it satisfies
$$||u|_{x_1=0}||_{H^m(\Gamma_T)}\lesssim ||u||_{m+1,\ast,T}.$$
(ii) There exists a continuous lifting operator $\mathcal{R}_T:$
$$H^m(\Gamma_T)\rightarrow H^{m+1}_{\ast}(\Omega_T),$$ such that
$(\mathcal{R}_Tu)|_{x_1=0}=u \text{ and } ||\mathcal{R}_Tu||_{m+1,\ast,T}\lesssim ||u||_{H^m({\Gamma_T})}.$

\end{lemma}
\subsection{Main Theorem}\label{mainresult}
\begin{theorem}\label{maintheorem}
Let $m\in \mathbb{N}$ and $m\geq15,$ and let $\bar{\mathbf{U}}^{\pm}$ be defined in \eqref{constant}. Suppose that the initial data in \eqref{IVP} satisfy
$$({\mathbf U}^{\pm}_0-\bar{\mathbf{U}}^{\pm},\varphi_0)\in H^{m+11.5}(\R^2_+)\times H^{m+11.5}(\R),$$
and also satisfy the hyperbolicity condition \eqref{hyper}, the divergence-free constraint \eqref{div} for all $\mathbf{x}\in \R^2_+.$ Let the initial data at $x_1=0$ satisfy the stability condition \eqref{intro-1.11}. The hyperbolicity condition \eqref{hyper} and the stability condition  \eqref{intro-1.11} have to be satisfied uniformly in the sense of \eqref{hy} and \eqref{stability1}, for suitable $k>0$. Assume that the initial data are compatible up to order m+10 in the sense of Definition \ref{def} and satisfy the boundary constraints \eqref{Hnn}. Assume also that
$$
\| \varphi_0\|_{L^\infty(\mathbb R)}< \frac{1}{2}.
$$
Then, there exists a sufficiently short time $T>0$ such that problem \eqref{IVP} has a unique solution on the time interval $[0,T]$ satisfying
$$({\mathbf U}^{\pm}-\bar{\mathbf{U}}^{\pm},\varphi)\in H^{m}_{\ast}([0,T]\times\R^2_+)\times H^m([0,T]\times\R).$$
\end{theorem}
\begin{remark}\label{dato_iniziale}
Since the initial data ${\bf U}_0^\pm$ have the form ${\bf U}_0^\pm=\overline{{\bf U}}^\pm+\tilde{{\bf U}}_0^\pm$, with $\tilde{{\bf U}}_0^\pm\in H^m(\mathbb R^2_+)$ vanishing at infinity (as $\vert{\bf x}\vert\to +\infty$), the hyperbolicity and stability conditions satisfied in the sense of \eqref{hy}, \eqref{stability1} (see Remark \ref{remark-hyp}) yield that the same conditions hold for the constant states \eqref{constant}.
\end{remark}
\begin{remark}
We note that Theorem \ref{maintheorem} implies corresponding results for the original free boundary problem \eqref{MHD}, \eqref{boundary} and \eqref{initial}, posed in the moving domain, because Lemma \ref{i3} and the relation $\partial_1\Phi^+\geq\frac{1}{2}$ and $\partial_1\Phi^-\leq-\frac{1}{2}$ hold in $[0,T]\times \R^2_+$ for sufficiently small $T>0.$
\end{remark}
\begin{remark}
	Compared to \cite{ChenG2008, Trakhinin2009}, in Theorem \ref{maintheorem} there is less loss of regularity from the initial data to the solution. This is mainly due to the use of finer Moser-type  and imbedding estimates in anisotropic Sobolev spaces.
\end{remark}
\begin{remark}
The analysis in this paper could also be applied to prove the nonlinear stability and existence of two dimensional relativistic current-vortex sheets, see \cite{ChenG2017, Freistuhler}.
\end{remark}
\bigskip

\section{Linearized Problem}\label{linearized}
\subsection{The basic state}
Let the basic state
\begin{equation}\label{basic}
(\hat{{\mathbf U}}^{\pm}(t,\mathbf{x}),\hat{\varphi}(t,x_2))
\end{equation}
be a given vector-valued and sufficiently smooth function, where $\hat{{\mathbf U}}^{\pm}=(\hat{p}^{\pm},\hat{\mathbf{u}}^{\pm},\hat{\mathbf{H}}^{\pm},\hat{S}^{\pm})^T$ are defined in
$\Omega_T.$
We assume that we shall linearize the problem \eqref{IVP} around the basic state \eqref{basic}, which satisfy the {\it hyperbolicity condition} \eqref{hyper} in $\Omega_T:$
\begin{equation}\label{hy}
\rho(\hat{p}^{\pm},\hat{S}^{\pm})\ge k>0,\quad \rho_{p}(\hat{p}^{\pm},\hat{S}^{\pm})\ge k>0\,,
\end{equation}
the Rankine-Hugoniot jump condition
\begin{equation}\label{jump}
\partial_t\hat{\varphi}=\hat{u}^{\pm}_N|_{x_1=0}\,,
\end{equation}
where $\hat{u}^{\pm}_{N}=\hat{u}^{\pm}_1-\hat{u}^{\pm}_2\partial_2\hat{\varphi}$ and the {\it stability condition} on the boundary:
\begin{equation}\label{stability1}
\hat a^+|\hat H^+_2| + \hat a^-|\hat H^-_2| -|[\hat u_2]|\ge k>0\,,
\end{equation}
for a suitable constant $k$, where $\hat{c}^{\pm}_A=\frac{\vert\hat{\mathbf{H}}^{\pm}\vert}{\sqrt{\hat{\rho}^{\pm}}}$ denotes the Alfv\'en speed and
\begin{equation}\label{apm}
\hat a^\pm:=\sqrt{\frac{1}{\hat{\rho}^{\pm}(1+\frac{(\hat{c}^{\pm}_A)^2}{(\hat{c}^{\pm})^2})}}\,.
\end{equation}
%
%Observe that the left-hand side of \eqref{stability1} corresponds to \eqref{intro-1.11} with $\mathbf{H}^\pm=(0,\hat{H}_2^\pm), \tau=(0,1)$.
%
\begin{remark}\label{remark-hyp}
	The stability condition \eqref{intro-1.11} can be written in uniform form as \eqref{stability1}
by making use of \eqref{jump} and the following boundary constraint in \eqref{c2}.
\end{remark}
\begin{remark}\label{Rem_vincoli_st_base}
Let us observe that, differently from \cite{Trakhinin2009}, in the Rankine-Hugoniot conditions we dot not require that the basic state satisfies $[\hat q]=0$. It seems that this condition could not be helpful to simplify the boundary quadratic form which appear from the application of the energy method to the linearized problem.
\end{remark}
\begin{remark}\label{H_2nonzero}
Estimate \eqref{stability1} implies in particular that at least one among $\hat H^+_2$ and $\hat H^-_2$ must be nonzero along the boundary.
\end{remark}
\begin{remark}
The presence of the positive constant $k$ in the hyperbolicity and stability assumptions \eqref{hy}, \eqref{stability1} is needed in order to ensure the boundedness of all coefficients, nonlinearly depending on the basic state $\hat{\mathbf U}^\pm$, $\hat\varphi$, appearing in the arguments of the energy method developed in the sequel.
\end{remark}
Let us assume
\begin{equation}\label{space}
\begin{split}
 &\hat{{\mathbf U}}^{\pm}\in W^{3,\infty}(\Omega_T),\quad\hat{\varphi}\in W^{4,\infty}(\Gamma_T),\\
 &||\hat{{\mathbf U}}^{\pm}||_{W^{3,\infty}(\Omega_T)}+||\hat{\varphi}||_{W^{4,\infty}(\Gamma_T)}\leq K,
\end{split}
\end{equation}
where $K>0$ is a constant.
Moreover, without loss of generality, we assume
\begin{equation}\label{phi-piccola}
||\hat{\varphi}||_{L^{\infty}(\Gamma_T)}<\frac{1}{2}.
\end{equation}
This implies
$$\partial_1\hat{\Phi}^+\geq\frac{1}{2},\quad \partial_1\hat{\Phi}^-\leq-\frac{1}{2},$$
with $$\hat{\Phi}^{\pm}(t,\mathbf{x}):=\pm x_1+\hat{\Psi}^{\pm}(t,\mathbf{x}),\quad \hat{\Psi}^{\pm}(t,\mathbf{x}):=\chi(\pm x_1)\hat{\varphi}(t,x_2).$$
%By using \eqref{space},
%$\hat{\mathbf W}:=(\hat{{\mathbf U}}^{\pm},\nabla_{t,\mathbf{x}}\hat{\Psi}^{\pm})$ satisfies $||\hat{\mathbf W}||_{W^{2,\infty}(\Omega_T)}\leq C(K)$ for some positive constant $C(K)$ depending on $K.$
We also assume the following nonlinear constraints on the background states:
\begin{equation}\label{c}
\partial_t\hat{\mathbf{H}}^{\pm}+\frac{1}{\partial_1\hat{\Phi}^{\pm}}\Big((\hat{\mathbf{w}}^{\pm}\cdot\nabla)\hat{\mathbf{H}}^{\pm}-(\hat{\mathbf{h}}\cdot\nabla)\hat{\mathbf{u}}^{\pm}+\hat{\mathbf{H}}^{\pm}\mathrm{div}{\hat{\mathrm{\mathbf{v}}}^{\pm}}\Big)=0,
\end{equation}
where
\begin{equation}\label{uvw}
\begin{split}
&\hat{\mathbf{v}}^{\pm}=(\hat{u}^{\pm}_n,\hat{u}^{\pm}_2\partial_1\hat{\Phi}^{\pm}),\quad\hat{u}^{\pm}_n=\hat{u}^{\pm}_1-\hat{u}^{\pm}_2\partial_2\hat{\Psi}^{\pm},\\
&\hat{\mathbf{h}}^{\pm}=(\hat{H}^{\pm}_n,\hat{H}^{\pm}_2\partial_1\hat{\Phi}^{\pm}),\quad \hat{H}^{\pm}_n=\hat{H}^{\pm}_1-\hat{H}^{\pm}_2\partial_2\hat{\Psi}^{\pm},\\
&\hat{\mathbf{w}}^{\pm}=\hat{\mathbf{v}}^{\pm}-(\partial_t\hat{\Psi}^{\pm},0)\,.
\end{split}
\end{equation}
Then, we can obtain that
the constraints \eqref{div} and \eqref{Hnn}, that is
\begin{equation}\label{c2}
\mathrm{div}\hat{\mathbf{h}}^{\pm}=0,\quad \hat{H}^{\pm}_N|_{x_1=0}=0,
\end{equation}
hold for all times $t>0$ if they hold initially (see Appendix A in \cite{Trakhinin2009}), where $\hat{H}^{\pm}_N=\hat{H}^{\pm}_1-\hat{H}^{\pm}_2\partial_2\hat{\varphi}.$
%%%%%%%%%%%%%%%%%%%%%%%%%%%%%%%%%%%%%%%%%%%%%%%%%%%%%%%%%%%%%%%%%%%%%%%%%
\subsection{The linearized equations}
The linearized equations for \eqref{IVP} around the basic state \eqref{basic} can be defined as follows:
$$\mathbb{L'}(\hat{{\mathbf U}}^{\pm},\hat{\Psi}^{\pm})(\delta {\mathbf U}^{\pm},\delta \Psi^{\pm}):=\frac{d}{d\varepsilon}\mathbb{L}({\mathbf U}^{\pm}_{\varepsilon},\Psi^{\pm}_{\varepsilon})|_{\varepsilon=0}=\mathbf{f}^{\pm} \text{ in } \Omega_T,$$
$$\mathbb{B'}(\hat{{\mathbf U}}^{+}, \hat{{\mathbf U}}^{-}, \hat{\varphi})(\delta {\mathbf U}^{+}, \delta{\mathbf U}^{-}, \delta\varphi)=\frac{d}{d\varepsilon}\mathbb{B}({\mathbf U}^{+}_{\varepsilon},{\mathbf U}^{-}_{\varepsilon},\varphi_{\varepsilon})|_{\varepsilon=0}=\mathbf{g} \text{ on }\Gamma_T,$$
where ${\mathbf U}^{\pm}_{\varepsilon}=\hat{{\mathbf U}}^{\pm}+\varepsilon\delta {\mathbf U}^{\pm},\quad \varphi_{\varepsilon}=\hat{\varphi}+\varepsilon\delta\varphi$
and
$$\Psi^{\pm}_{\varepsilon}(t,\mathbf{x}):=\chi(\pm x_1)\varphi_{\varepsilon}(t,x_2),\quad \Phi^{\pm}_{\varepsilon}(t,\mathbf{x}):=\pm x_1+\Psi^{\pm}_{\varepsilon}(t,\mathbf{x}),\quad
\delta\Psi^{\pm}(t,\mathbf{x}):=\chi(\pm x_1)\delta\varphi(t,x_2).$$
In the following argument, we shall drop $\delta$ for simplicity.
The linearized operators have the following form:
\begin{equation}\label{Lop}
\begin{split}
&\mathbb{L}'(\hat{{\mathbf U}}^{\pm},\hat{\Psi}^{\pm})({\mathbf U}^{\pm},\Psi^{\pm})\\
&=L(\hat{{\mathbf U}}^{\pm},\hat{\Psi}^{\pm}){\mathbf U}^{\pm}+C(\hat{{\mathbf U}}^{\pm},\hat{\Psi}^{\pm}){\mathbf U}^{\pm}-\{L(\hat{{\mathbf U}}^{\pm},\hat{\Psi}^{\pm})\Psi^{\pm}\}\frac{\partial_1\hat{{\mathbf U}}^{\pm}}{\partial_1\hat{\Phi}^{\pm}},
\end{split}
\end{equation}
\begin{equation}\label{Bop}
\mathbb{B}'(\hat{{\mathbf U}}^{+},\hat{{\mathbf U}}^{-},\hat{\varphi})({\mathbf U}^{+},{\mathbf U}^{-},\varphi)=\left[
\begin{array}{c}
\partial_t\varphi+\hat{u}^+_2\partial_2\varphi-u^+_N\\
\partial_t\varphi+\hat{u}^-_2\partial_2\varphi-u^-_N\\
q^+-q^-
\end{array}\right],
\end{equation}
where the operator $L(\hat{{\mathbf U}}^{\pm},\hat{\Psi}^{\pm})$ is defined in \eqref{opL}, $q^{\pm}=p^{\pm}+\hat{\mathbf{H}}^{\pm}\cdot \mathbf{H}^{\pm},u^{\pm}_N=u^{\pm}_1-u^{\pm}_2\partial_2\hat{\varphi},$
and the matrix $C(\hat{{\mathbf U}}^{\pm},\hat{\Psi}^{\pm})$ is defined as follows:
\begin{equation}\nonumber
\begin{split}
C(\hat{{\mathbf U}}^{\pm},\hat{\Psi}^{\pm})Y=(Y,\nabla_yA_0(\hat{{\mathbf U}}^{\pm}))\partial_t\hat{{\mathbf U}}^{\pm}+(Y,\nabla_y\tilde{A}_1(\hat{{\mathbf U}}^{\pm},\hat{\Psi}^{\pm}))\partial_1\hat{{\mathbf U}}^{\pm}+(Y,\nabla_yA_2(\hat{{\mathbf U}}^{\pm}))\partial_2\hat{{\mathbf U}}^{\pm}\,,
\end{split}
\end{equation}
\begin{equation}\nonumber
(Y,\nabla_yA(\hat{{\mathbf U}}^{\pm}))=\sum_{i=1}^6y_i(\frac{\partial A(Y)}{\partial y_i}\Big|_{Y=\hat{{\mathbf U}}^{\pm}}),\quad Y=(y_1,\cdots,y_6).
\end{equation}
We introduce the Alinhac's good unknown \cite{Alinhac}
\begin{equation}\label{alinhac}
\dot{{\mathbf U}}^{\pm}:={\mathbf U}^{\pm}-\frac{\partial_1\hat{{\mathbf U}}^{\pm}}{\partial_1\hat{\Phi}^{\pm}}\Psi^{\pm}.
\end{equation}
In terms of \eqref{alinhac}, the linearized interior equations have the following form
\begin{equation}\label{linearized22}
L(\hat{{\mathbf U}}^{\pm},\hat{\Psi}^{\pm})\dot{{\mathbf U}}^{\pm}+C(\hat{{\mathbf U}}^{\pm},\hat{\Psi}^{\pm})\dot{{\mathbf U}}^{\pm}-\frac{\Psi^{\pm}}{\partial_1\hat{\Phi}^{\pm}}\partial_1\{\mathbb{L}(\hat{{\mathbf U}}^{\pm},\hat{\Psi}^{\pm})\}=\mathbf{f^{\pm}}.
\end{equation}
Since the zero order terms in $\Psi^{\pm}$ can be regarded as the error terms in the Nash-Moser iteration,  we drop these terms and  consider the effective linear operators:
\begin{equation}\label{Le1}
\begin{split}
\mathbb{L}'_e(\hat{{\mathbf U}}^{\pm},\hat{\Psi}^{\pm})\dot{{\mathbf U}}^{\pm}&:=L(\hat{{\mathbf U}}^{\pm},\hat{\Psi}^{\pm})\dot{{\mathbf U}}^{\pm}+C(\hat{{\mathbf U}}^{\pm},\hat{\Psi}^{\pm})\dot{{\mathbf U}}^{\pm}\\
&=A_0(\hat{{\mathbf U}}^{\pm})\partial_t\dot{{\mathbf U}}^{\pm}+\tilde{A}_1(\hat{{\mathbf U}}^{\pm},\hat{\Psi}^{\pm})\partial_1\dot{{\mathbf U}}^{\pm}+A_2(\hat{{\mathbf U}}^{\pm})\partial_2\dot{{\mathbf U}}^{\pm}+C(\hat{{\mathbf U}}^{\pm},\hat{\Psi}^{\pm})\dot{{\mathbf U}}^{\pm}.
\end{split}
\end{equation}
Concerning the boundary differential operator $\mathbb{B}'$, we rewrite \eqref{Bop} in terms of the Alinhac's good unknowns \eqref{alinhac} to get
\begin{equation}\label{Be1}
\mathbb{B}'_e(\hat{{\mathbf U}},\hat{\varphi})(\dot{{\mathbf U}},\varphi)=\left[
\begin{array}{c}
\partial_t\varphi+\hat{u}^+_2\partial_2\varphi-\dot{u}^+_N-\varphi\partial_1\hat{u}^+_N\\
\partial_t\varphi+\hat{u}^-_2\partial_2\varphi-\dot{u}^-_N+\varphi\partial_1\hat{u}^-_N\\
\dot{q}^+-\dot{q}^-+\varphi(\partial_1\hat{q}^++\partial_1\hat{q}^-)
\end{array}\right],
\end{equation}
where $\dot{u}^{\pm}_N=\dot{u}^{\pm}_1-\dot{u}^{\pm}_2\partial_2\hat{\varphi}$ (recall also that $\partial_1\hat\Phi^\pm\vert_{x_1=0}=\pm 1$).
\begin{remark}
Notice that, due to the fact that we have transformed the domains $\Omega^\pm (t)$ into the same half-space $\mathbb{R}^2_+$, the jump on the boundary of a normal derivative of a function $a=a(t,{\bf x})$ is defined as follows
\begin{equation}
[\partial_1a]:=\partial_1a^+_{|x_1=0}+\partial_1a^-_{|x_1=0}.
\label{norm_jump}
\end{equation}
This is why the jump of the total pressure $q$ in the last row of \eqref{Be1} reduces under Alinhac's change of unknowns to
\[
\dot{q}^+-\dot{q}^-+\varphi(\partial_1\hat{q}^++\partial_1\hat{q}^-)\,.
\]
In the following, according to \eqref{norm_jump}, we will set
\begin{equation}\label{jump_q}
[\partial_1\hat q]:=\partial_1\hat{q}^+\vert_{x_1=0}+\partial_1\hat{q}^-\vert_{x_1=0}\,.
\end{equation}
\end{remark}
Denote the operator:
\begin{equation}
\mathbb{L}'_e(\hat{{\mathbf U}},\hat{\Psi})\dot{{\mathbf U}}:=\left[
\begin{array}{c}
\mathbb{L}'_e(\hat{{\mathbf U}}^+,\hat{\Psi}^+)\dot{{\mathbf U}}^+\\
\mathbb{L}'_e(\hat{{\mathbf U}}^-,\hat{\Psi}^-)\dot{{\mathbf U}}^-
\end{array}\right]\,.
\end{equation}
Now, we are focusing on the following linear problem for $(\dot{{\mathbf U}}^{\pm},\varphi):$
\begin{equation}\label{IVP2}
\begin{cases}
\mathbb{L}'_e(\hat{{\mathbf U}},\hat{\Psi})\dot{{\mathbf U}}=\mathbf{f} & \text{in } \Omega_T,\\
\mathbb{B}'_e(\hat{{\mathbf U}},\hat{\varphi})(\dot{{\mathbf U}},\varphi)=\mathbf{g} & \text{on } \Gamma_T,\\
(\dot{{\mathbf U}},\varphi)=0 & \text{for } t<0, \\
\end{cases}
\end{equation}
where $\mathbf{f}=(\mathbf{f}^+,\mathbf{f}^-)=(f^+_1,\cdots,f^+_6,f^-_1,\cdots,f^-_6),$
and $\mathbf{g}=(g^+_1,g^-_1,g_2)$ vanish in the past.
In order to prove the well-posedness of the linearized problem \eqref{IVP2}, we can state the following Lemma \ref{hoo}, that can be proved as in \cite[Proposition 2]{Trakhinin2009}.
\begin{lemma}\label{hoo}
Let the basic state \eqref{basic} satisfy the assumption \eqref{c} and \eqref{c2}. Then, the solutions of the problem \eqref{IVP2} satisfy
\begin{equation}
\mathrm{div} \dot{\mathbf h}^{\pm}=r^{\pm} \text{ in } \Omega_T,
\end{equation}
\begin{equation}\label{g3}
\hat{H}_2^{\pm}\partial_2\varphi-\dot{H}^{\pm}_N\mp\varphi\partial_1\hat{H}^{\pm}_N=g^{\pm}_3\\
\text{ on } \Gamma_{T}.
\end{equation}
Here $$\dot{\mathbf h}^{\pm}=(\dot{H}^{\pm}_n,\dot{H}^{\pm}_2\partial_1\hat{\Phi}^{\pm}),\quad \dot{H}^{\pm}_n=\dot{H}^{\pm}_1-\dot{H}^{\pm}_2\partial_2\hat{\Psi}^{\pm}, \quad  \dot{H}^{\pm}_N |_{x_1=0}=\dot{H}^{\pm}_n |_{x_1=0}, $$
where $r^{\pm}=r^{\pm}(t,\mathbf{x}),\quad g^{\pm}_3=g^{\pm}_3(t,x_2),$ which vanish in the past, are determined by the source terms and the basic state as solutions to the linear equations
\begin{equation}\label{F}
\partial_tR^{\pm}+\frac{1}{\partial_1\hat{\Phi}^{\pm}}\Big\{(\hat{\mathbf{w}}^{\pm}\cdot\nabla R^{\pm})+R^{\pm}\mathrm{div}\hat{\mathbf{v}}^{\pm}\Big\}=\mathcal{F}^{\pm} \text{ in } \Omega_T,
\end{equation}
\begin{equation}\label{G}
\partial_tg^{\pm}_3+\hat{u}^{\pm}_2\partial_2g^{\pm}_3+\partial_2\hat{u}^{\pm}_2g^{\pm}_3=\mathcal{G}^{\pm} \text{ on } \Gamma_T,
\end{equation}
where $R^{\pm}=\frac{r^{\pm}}{\partial_1\hat{\Phi}^{\pm}},\quad\mathcal{F}^{\pm}=\frac{\mathrm{div} \mathbf{f}^{\pm}_h}{\partial_1\hat{\Phi}^{\pm}}, \quad \mathbf{f}^{\pm}_h=(f^{\pm}_n,f^{\pm}_5),\quad f^{\pm}_n=f^{\pm}_4-f^{\pm}_5\partial_2\hat{\Psi}^{\pm}$,
\newline
$\mathcal{G}^{\pm}=\{\partial_2(\hat{H}^{\pm}_2g^{\pm}_1)-f^{\pm}_n\}|_{x_1=0}$, with  $\hat{\mathbf w}^\pm$ and $\hat{\mathbf v}^\pm$  defined in \eqref{uvw}.
\end{lemma}
\subsection{Reduction to homogeneous boundary conditions}\label{hom_Sect}
In this section, we reduce the inhomogeneous  boundary condition in \eqref{IVP2} to the homogeneous one. We follow the same ideas in Trakhinin \cite{Trakhinin2009}, that for reader's convenience, we recall here.  Suppose there exists a solution $(\dot{{\mathbf U}},\varphi)\in H^s_{\ast}(\Omega_T)\times H^s({\Gamma_T})$ to  problem \eqref{IVP2}, with a given $s\in \mathbb{N}.$  We define a vector-valued function
$$\tilde{{\mathbf U}}=(\tilde{p}^+,\tilde{\mathbf{u}}^+,\tilde{\mathbf H}^+,\tilde{S}^+,\tilde{p}^-,\tilde{\mathbf{u}}^-,\tilde{\mathbf H}^-,\tilde{S}^-)\in H^{s+2}_{\ast}(\Omega_T)$$
that vanishes in the past and such that it is a ``suitable lifting" of the boundary data $({\mathbf g}, g^+_3, g^-_3)\in H^{s+1}(\Gamma_T)$. We choose $\tilde{{\mathbf U}}$ such that on the boundary $\Gamma_T$, it satisfies the boundary conditions in \eqref{IVP2} with $\varphi=0$, i.e.
$$\mathbb{B}'_e(\hat{{\mathbf U}},\hat{\varphi})(\tilde{{\mathbf U}},0)=\mathbf{g},$$
and both conditions \eqref{g3} and \eqref{G} with $\varphi=0$ (where $\tilde q$, $\tilde u^\pm_N$, $\tilde H_N^\pm$ are defined similarly to $\dot q$, $\dot u_N^\pm$ and $\dot H_N^\pm$). More explicitly we require
 $$\tilde{u}^{\pm}_N |_{x_1=0}=-g^{\pm}_{1},$$
 $$
 \tilde q^+ |_{x_1=0} - \tilde q^- |_{x_1=0}= g_2,
 $$
 $$
 \tilde{H}^{\pm}_N |_{x_1=0}=-g^{\pm}_{3}
 $$
 with $g^\pm_3$ solution of \eqref{G}.
Then, we define $\tilde q^\pm$, $\tilde u_n^\pm$ and $\tilde H_n^\pm$ in the interior domain $\Omega_T$ by using the lifting operator (that exists thanks to the trace theorem in anisotropic Sobolev spaces $H^s_*$, see \cite{Ohno-Shizuta1994} and Appendix \ref{tracetheorem})
$$
\mathcal{R}_T: H^{s+1}(\Gamma_T) \rightarrow H^{s+2}_*(\Omega_T)
$$
which gives
$$
\tilde q^\pm=\mathcal{R}_T(\tilde q^\pm |_{x_1=0}), \quad \tilde u^\pm_n=\mathcal{R}_T(\tilde u^\pm_N |_{x_1=0}), \quad \tilde H^\pm_n=\mathcal{R}_T(\tilde H^\pm_N |_{x_1=0}).
$$
Let also $\tilde u_1^\pm$ and $\tilde H^\pm_1$ be such that
$$
\tilde u_n^\pm= \tilde u_1^\pm - \tilde u_2^\pm \partial_2 \hat\Psi^\pm, \quad \tilde H_n^\pm= \tilde H_1^\pm - \tilde H_2^\pm \partial_2 \hat\Psi^\pm,
$$
where $\tilde u_2^\pm$ is arbitrary and can be taken for instance as zero, and where we can define $\tilde H^\pm_2$ in such a way that it satisfies equation \eqref{F} for $R^\pm=\frac{\text{div} \tilde{\mathbf h}^\pm}{\partial_1\hat\Phi^\pm}$, where $\tilde{\mathbf h}^\pm=(\tilde H_n^\pm, \tilde H_2^\pm\partial_1\hat\Phi^\pm)$ (this is possible since we  have a freedom in the choice of ``characteristic unknown" $\tilde H^\pm_2$ ). The last components $\tilde S^\pm$ of $\tilde{\mathbf U}$  can again be taken as zero. To sum up, the vector $\tilde{\mathbf U}$ is defined as
$$
\tilde{{\mathbf U}}=(\tilde{p}^+,\tilde{u}^+_n, 0,\tilde{H}^+_n+\tilde{H}^+_2\partial_2\hat{\Psi}^+,\tilde{H}^+_2,0,\tilde{p}^-,\tilde{u}^-_n, 0,\tilde{H}^-_n+\tilde{H}^-_2\partial_2\hat{\Psi}^-,\tilde{H}^-_2,0)\,,
$$
where $\tilde H^\pm_2$ satisfies equation \eqref{F} for $R^\pm=\frac{\text{div} \tilde{\mathbf h}^\pm}{\partial_1\hat\Phi^\pm}$ and $\tilde{\mathbf h}^\pm=(\tilde H_n^\pm, \tilde H_2^\pm\partial_1\hat\Phi^\pm)$.
\newline
We define $\dot{{\mathbf U}}^{\natural}=\dot{{\mathbf U}}-\tilde{{\mathbf U}},$ then $\dot{{\mathbf U}}^{\natural}$ satisfies
\begin{equation}\label{ho}
\begin{cases}
\mathbb{L}'_e(\hat{{\mathbf U}},\hat{\Psi})\dot{{\mathbf U}}^{\natural}=\mathbf{F}=\mathbf{f}-\mathbb{L}'_e(\hat{{\mathbf U}},\hat{\Psi})\tilde{{\mathbf U}} & \text{ in } \Omega_T,\\
\mathbb{B}'_e(\hat{{\mathbf U}},\hat{\Psi})(\dot{{\mathbf U}}^{\natural},\varphi)=0 & \text{ on } \Gamma_T.
\end{cases}
\end{equation}
Here $\mathbf{F}=(F^+_1,\cdots,F^+_6,F^-_1,\cdots,F^-_6).$
Moreover, in view of equations \eqref{F} for $R^{\pm}=\frac{\mathrm{div}\tilde{\mathbf h}^{\pm}}{\partial_1\hat{\Phi}^{\pm}},$ condition \eqref{g3} for $\tilde{\mathbf U}$ with $\varphi=0,$ and \eqref{G}, we have, from \eqref{ho}, that \eqref{F} and \eqref{G} are satisfied for $R^{\pm}=\frac{\mathrm{div}\dot{\mathbf h}^{\natural\pm}}{\partial_1\hat{\Phi}^{\pm}},$ $g^{\pm}_3=(\hat{H}^{\pm}_2\partial_2\varphi-\dot{H}^{\natural\pm}_N\mp\varphi\partial_1\hat{H}^{\pm}_N)|_{x_1=0}$
with right-hand sides $\mathcal{F}^{\pm}=\mathcal{G}^{\pm}=0$. Here $\dot{\mathbf h}^{\natural\pm}$ and $\dot{H}^{\natural\pm}_N$ are defined similarly to $\dot{\mathbf h}^{\pm}$ and $\dot{H}^{\pm}_N.$
Notice that $\dot{{\mathbf U}}^{\natural}=0, \text{ for } t<0.$  Hence, the conditions
\begin{equation}\label{diesis-h-hN}
\mathrm{div}\dot{\mathbf h}^{\natural\pm}=0,\quad (\hat{H}^{\pm}_2\partial_2\varphi-\dot{H}^{\natural\pm}_N{\mp}\varphi\partial_1\hat{H}^{\pm}_N)|_{x_1=0}=0
\end{equation}
hold for $t<0$. Then, by standard method of characteristic curves, we get that equations \eqref{diesis-h-hN} are satisfied for all  $t\in(-\infty,T].$
Notice that from
$$||\partial_1(\cdot)||_{s,\ast,T}\leq||\cdot||_{s+2,\ast,T},
$$
we have that
\begin{equation}\label{2}
||\mathbb{L}'_e(\hat{{\mathbf U}},\hat{\Psi})\tilde{{\mathbf U}}||_{s,\ast,T}\leq C||\tilde{{\mathbf U}}||_{s+2,\ast,T}.
\end{equation}
By the definition of ${\mathbf F}$, using \eqref{2} and the definition of $\tilde {\mathbf U}$ as a lifting of the boundary data $({\mathbf g},g_3^+, g_3^-)$, we obtain that
\begin{equation}\label{stima_F}
||\mathbf{F}||_{s,\ast,T}\leq C\Big(||\mathbf{f}||_{s,\ast,T}+||\mathbf{g}||_{H^{s+1}(\Gamma_T)} + ||(g_3^+,g_3^-)||_{H^{s+1}(\Gamma_T)}\Big).
\end{equation}
Using \eqref{G} and the trace Theorem in the anisotropic Sobolev spaces to estimate
$$
\Vert f_n\vert_{x_1=0}\Vert_{H^{s+1}(\Gamma_T)}\le C\Vert\mathbf{f}\Vert_{s+2,\ast,T}\,,
$$
we get
\begin{equation}\label{stima_g3}
\begin{split}
||(g^{+}_3,g^{-}_3) ||_{H^{s+1}(\Gamma_T)}\leq C||(\mathcal{G}^{+},\mathcal{G}^{-}) ||_{H^{s+1}(\Gamma_T)}\leq C\Big(||\mathbf{f}||_{s+2,\ast,T}+||(g^{+}_1, g^{-}_1) ||_{H^{s+2}(\Gamma_T)}\Big).
\end{split}
\end{equation}
Then, from \eqref{stima_F} and \eqref{stima_g3}, we derive
\begin{equation}\label{estimate}
||\mathbf{F}||_{s,\ast,T}\leq C\Big(||\mathbf{f}||_{s+2,\ast,T}+||\mathbf{g}||_{H^{s+2}(\Gamma_T)}\Big).
\end{equation}
 We obtain that $\dot{{\mathbf U}}^{\natural}$ solves the following problem:
\begin{equation}\label{ho2}
\begin{cases}
\mathbb{L}'_e(\hat{{\mathbf U}},\hat{\Psi})\dot{{\mathbf U}^\natural}=\mathbf{F} & \text{ in } \Omega_T,\\
\mathbb{B}'_e(\hat{{\mathbf U}},\hat{\varphi})(\dot{{\mathbf U}^\natural},\varphi)=0 & \text{ on } \Gamma_T,\\
(\dot{{\mathbf U}^\natural},\varphi)=0 & \text{ for } t<0,
\end{cases}
\end{equation}
where the source term $\mathbf{F}$ satisfies the estimate \eqref{estimate} and solutions of \eqref{ho2} satisfy the constraints
\begin{equation}\label{div2}
\mathrm{div}\dot{\mathbf h^\natural
}^{\pm}=0 \text{ in }\Omega_T,
\end{equation}
\begin{equation}\label{Hn2}
\hat{H}^{\pm}_2\partial_2\varphi-\dot{H}^{\natural\,\pm}_{N}{\mp}\varphi\partial_1\hat{H}^{\pm}_{N}=0 \text{ on }\Gamma_T.
\end{equation}
We have thus proved the following result.
\begin{lemma}\label{lemma4.2}
Let problem \eqref{ho2} be well-posed and its unique solution $(\dot{{\mathbf U}^\natural},\varphi)$ belongs to $H^s_{\ast}(\Omega_T)\times H^s({\Gamma_T})$ for $\mathbf{F}\in H^s_{\ast}(\Omega_T),$ where $s\in \mathbb{N}$ is a given number. Then  problem \eqref{IVP2} is well-posed, namely it admits a unique solution $(\dot{\mathbf U},\varphi) \in H^s_{\ast}(\Omega_T)\times H^s({\Gamma_T})$ for data $(\mathbf{f},\mathbf{g})\in H^{s+2}_{\ast}(\Omega_T)\times H^{s+2}(\Gamma_T).$
\end{lemma}
\begin{remark}
{\rm Let us observe that loss of regularity from the data to the solution in the inhomogeneous problem \eqref{IVP2} is due to the introduction of lifting function $\tilde{{\mathbf U}}$.}
\end{remark}
\subsection{New Friedrichs Symmetrizer for 2D MHD equations}

Motivated by the idea of Trakhinin \cite{Trakhinin2009}, in the following we will make use of a new symmetric form of the MHD system. This symmetric form is the result of the application of a ``secondary generalized Friedrichs symmetrizer"  $\mathbb S=(S({\mathbf U}), \mathcal T({\mathbf U}))$ to system \eqref{quasi}:
\begin{equation}\label{sec}
\begin{split}
&\mathcal{S}({\mathbf U})A_0({\mathbf U})\partial_t{\mathbf U}+\mathcal{S}({\mathbf U})A_1({\mathbf U})\partial_1{\mathbf U}+\mathcal{S}({\mathbf U})A_2({\mathbf U})\partial_2{\mathbf U}+\mathcal{T}({\mathbf U})\mathrm{div} \mathbf{H}\\
&\quad:=B_0({\mathbf U})\partial_t{\mathbf U}+B_1({\mathbf U})\partial_1{\mathbf U}+B_2({\mathbf U})\partial_2{\mathbf U}=0.
\end{split}
\end{equation}
The Friedrichs symmetrizer can be written as (see \cite{Trakhinin2005})
\begin{equation}\label{new}
\mathcal{S}({\mathbf U})=\left[\begin{array}{cccccc}
    1 & \frac{\lambda H_1}{\rho c^2} & \frac{\lambda H_2}{\rho c^2} & 0 & 0 & 0 \\
    \lambda H_1\rho & 1 & 0 & -\rho \lambda & 0 & 0\\
    \lambda H_2\rho & 0 & 1 & 0 & -\rho \lambda & 0\\
    0 & -\lambda & 0 & 1 & 0 & 0 \\
    0 & 0 & -\lambda & 0 & 1 & 0\\
    0 & 0 & 0 & 0 & 0 & 1
  \end{array}\right],\quad
 \mathcal{ T}({\mathbf U})=-\lambda\left[\begin{array}{c}
    1 \\
    0\\
    0\\
  H_1 \\
    H_2\\
   0
  \end{array}\right],
\end{equation}
\begin{equation}\label{B0}
B_0=\mathcal{S}A_0=\left[\begin{array}{cccccc}
    \frac{1}{\rho c^2} & \frac{\lambda H_1}{ c^2} & \frac{\lambda H_2}{ c^2} & 0 & 0 & 0 \\
    \frac{\lambda H_1}{c^2} & \rho & 0 & -\rho \lambda & 0 & 0\\
    \frac{\lambda H_2}{c^2} & 0 & \rho & 0 & -\rho \lambda & 0\\
    0 & -\rho\lambda & 0 & 1 & 0 & 0 \\
    0 & 0 & -\rho\lambda & 0 & 1 & 0\\
    0 & 0 & 0 & 0 & 0 & 1
  \end{array}\right],
\end{equation}
where $\lambda=\lambda(\mathbf U)$ is an arbitrary function.
\newline
In order to make system \eqref{sec} symmetric hyperbolic, we need $B_0>0,$ i.e.,
\begin{equation}\label{B01}
\rho\lambda^2<\frac{1}{1+\frac{c^2_A}{c^2}}.
\end{equation}
Condition \eqref{B01} ensures the equivalence of system \eqref{quasi} and \eqref{sec} on smooth solutions if $\lambda({\mathbf U})$ is a smooth function of ${\mathbf U}$ (see Trakhinin \cite{Trakhinin2005} and \cite{Trakhinin2009}).
\medskip
\newline
Now let us apply the {\it{new symmetrization}} to the {\it homogeneous} linearized problem \eqref{ho2}. From now on {\it we drop the index $\natural$} from the unknown $\dot{\mathbf U}^\natural$ of system \eqref{ho2}.
\newline
Multiplying \eqref{ho2} on the left  by $\mathcal{S}(\hat{{\mathbf U}})$ and adding to the result the vector
\begin{equation}
\left[\begin{array}{c}
    \frac{\mathrm{div} \dot{\mathbf h}^{+}}{\partial_1\hat{\Phi}^{+}}\mathcal{T}(\hat{{\mathbf U}}^+) \\
     \frac{\mathrm{div} \dot{\mathbf h}^{-}}{\partial_1\hat{\Phi}^{-}}\mathcal{T}(\hat{{\mathbf U}}^-) \\
  \end{array}\right],
\end{equation}
we obtain
\begin{equation}\label{B1}
B_0(\hat{{\mathbf U}})\partial_t\dot{{\mathbf U}}+\tilde{B}_1(\hat{{\mathbf U}},\hat{\Psi})\partial_1\dot{{\mathbf U}}+B_2(\hat{{\mathbf U}})\partial_2\dot{{\mathbf U}}+\tilde{C}(\hat{{\mathbf U}},\hat{\Psi})\dot{{\mathbf U}}=\tilde{\mathbf F}(\hat{{\mathbf U}}),
\end{equation}
where $\tilde{C}(\hat{{\mathbf U}},\hat{\Psi})=\mathcal{S}(\hat{{\mathbf U}})C(\hat{{\mathbf U}},\hat{\Psi}),\quad \tilde{\mathbf F}(\hat{{\mathbf U}})=\mathcal{S}(\hat{{\mathbf U}}){\mathbf F}$,
$$S(\hat{{\mathbf U}})=\mathrm{diag}(\mathcal{S}(\hat{{\mathbf U}}^+),S(\hat{{\mathbf U}}^-)),\quad B_{\alpha}(\hat{{\mathbf U}})=\mathrm{diag}( B_{\alpha}(\hat{{\mathbf U}}^+), B_{\alpha}(\hat{{\mathbf U}}^-)),\,\,\,\alpha=0,2\,,$$
$$C(\hat{{\mathbf U}},\hat{\Psi})=\mathrm{diag}(C(\hat{{\mathbf U}}^+,\hat{\Psi}^+),C(\hat{{\mathbf U}}^-,\hat{\Psi}^-)),\quad \tilde{B}_1(\hat{{\mathbf U}},\hat{\Psi})=\mathrm{diag}(\tilde{ B}_{1}(\hat{{\mathbf U}}^+,\hat{\Psi}^+), \tilde{B}_{1}(\hat{{\mathbf U}}^-,\hat{\Psi}^-)), $$
$$\tilde{B}_1({\mathbf U}^{\pm},\Psi^{\pm})=\frac{1}{\partial_1\Phi^{\pm}}\Big(B_1({\mathbf U}^{\pm})-B_0({\mathbf U}^{\pm})\partial_t\Psi^{\pm}-B_2({\mathbf U}^{\pm})\partial_2\Psi^{\pm}\Big).$$

\section{Stability condition and well-posedness of the linearized problem}\label{wellposedness}
In this section, let us introduce the new unknown $\mathbf V=(\mathbf V^+,\mathbf V^-),$ where
$$\mathbf V^{\pm}=(\dot{q}^{\pm},\dot{u}^{\pm}_n,\dot{u}^{\pm}_2,\dot{H}^{\pm}_n,\dot{H}^{\pm}_2,\dot{S}^{\pm}).$$
Rewriting system \eqref{ho2} in terms of $\mathbf V$ shows in clear way that the boundary matrix of the resulting system  has constant rank at the boundary (see \eqref{A1-rank4}), i.e. the system is symmetric hyperbolic with characteristic boundary of constant multiplicity (in the sense of Rauch \cite{Rauch1985}). Indeed, we obtain that
\begin{equation}\label{V}
\dot{{\mathbf U}}=J{\mathbf V}\,,
\end{equation}
where $J=\mathrm{diag}(J^+,J^-),$
\begin{equation}
J^{\pm}=\left[\begin{array}{cccccc}
   1 & 0 & 0 & -\hat{H}^{\pm}_1 & -\hat{H}^{\pm}_{\tau}  & 0 \\
   0 & 1 & \partial_2\hat{\Psi}^{\pm} & 0  & 0 & 0\\
    0 & 0 & 1 & 0 & 0 & 0\\
    0 & 0 & 0 & 1 & \partial_2\hat{\Psi}^{\pm} & 0 \\
    0 & 0 & 0 & 0 & 1 & 0\\
    0 & 0 & 0 & 0 & 0 & 1
  \end{array}\right],
\end{equation}
with $\hat{H}^{\pm}_{\tau}=\hat{H}\cdot\tau^{\pm},\quad \tau^{\pm}=(\partial_2\hat{\Psi}^{\pm},1).$
In terms of $\mathbf V$, systems \eqref{ho2} and \eqref{B1} can be equivalently rewritten as:
\begin{equation}\label{A}
\mathcal{A}_0(\hat{{\mathbf U}},\hat{\Psi})\partial_t{\mathbf V}+\mathcal{A}_1(\hat{{\mathbf U}},\hat{\Psi})\partial_1{\mathbf V}+\mathcal{A}_2(\hat{{\mathbf U}},\hat{\Psi})\partial_2{\mathbf V}+\mathcal{A}_3(\hat{{\mathbf U}},\hat{\Psi}){\mathbf V}=\mathcal{F}(\hat{{\mathbf U}},\hat{\Psi}),
\end{equation}
\begin{equation}\label{BBB}
\mathcal{B}_0(\hat{{\mathbf U}},\hat{\Psi})\partial_t{\mathbf V}+\mathcal{B}_1(\hat{{\mathbf U}},\hat{\Psi})\partial_1{\mathbf V}+\mathcal{B}_2(\hat{{\mathbf U}},\hat{\Psi})\partial_2{\mathbf V}+\mathcal{B}_3(\hat{{\mathbf U}},\hat{\Psi}){\mathbf V}=\tilde{\mathcal{F}}(\hat{{\mathbf U}},\hat{\Psi}),
\end{equation}
where $\mathcal{A}_{\alpha}=J^TA_{\alpha}J,\quad \mathcal{B}_{\alpha}=J^TB_{\alpha}J,\quad \alpha=0,2,$
$$\mathcal{A}_1=J^T\tilde{A}_1J,\quad \mathcal{B}_1=J^T\tilde{B}_1J,\quad \mathcal{F}=J^T\mathbf{F},\quad\tilde{\mathcal{F}}=J^T\tilde{\mathbf{F}},$$
$$\mathcal{A}_3=J^T\Big(CJ+A_0\partial_tJ+\tilde{A}_1\partial_1J+A_2\partial_2J\Big),$$
$$\mathcal{B}_3=J^T\Big(CJ+B_0\partial_tJ+\tilde{B}_1\partial_1J+B_2\partial_2J\Big).$$
In view of constraints \eqref{jump} and \eqref{c2} on the basic state $\hat{\mathbf U}$, the boundary matrix in \eqref{A} has the following form
\begin{equation}\label{AA}
\mathcal{A}_1=\mathcal{A}+\mathcal{A}_{(0)},\quad \mathcal{A}=\mathrm{diag}\Big(\frac{1}{\partial_1\hat{\Phi}^+}\mathcal{E}_{12},\frac{1}{\partial_1\hat{\Phi}^-}\mathcal{E}_{12}\Big),\quad \mathcal{A}_{(0)}|_{x_1=0}=0.
\end{equation}
Here, $\mathcal{E}_{ij}$ is a $6\times6$ symmetric matrix, in which $(ij)$th and $(ji)$th elements are 1 and the remaining elements are 0. The explicit form of $\mathcal{A}_{(0)}$ is of no interest and it is only important that the all non zero elements of $\mathcal{A}_{(0)}$ are multiplied either by the function $\hat{u}_n -\partial_t \hat{\Psi}$ or by the function $\hat{H}_n$.
%In view of constraints \eqref{jump} and \eqref{c2} on the basic state $\hat{\mathbf U}$, we have
%$$
%\mathcal{A}_{(0)}|_{x_1=0}=0.
%$$
Therefore, the boundary matrix
\begin{equation}\label{A1-rank4}
\mathcal{A}_1|_{x_1=0}=\mathrm{diag}(\mathcal{E}_{12},-\mathcal{E}_{12}).
\end{equation}
 This is a matrix of constant rank 4 and has two positive and two negative eigenvalues.
 \newline
Concerning system \eqref{BBB}, note that $\mathcal{B}_0>0$ because of the hyperbolicity condition \eqref{B01} satisfied for the basic state,  hence the symmetric system \eqref{BBB} is hyperbolic.  Moreover, considering the boundary matrix $\mathcal{B}_1$ in \eqref{BBB}, we will only need its explicit form on the boundary:
$$\mathcal{B}_1|_{x_1=0}=\mathrm{diag}\Big(\mathcal{B}(\hat{{\mathbf U}}^+|_{x_1=0}),-\mathcal{B}(\hat{{\mathbf U}}^-|_{x_1=0})\Big),\quad \mathcal{B}(\hat{\mathbf U}^\pm)=\mathcal{E}_{12}-\lambda(\hat{\mathbf U}^\pm)\mathcal{E}_{14},$$
which gives
\begin{equation}\label{boundary-q-form-lambda}
(\mathcal{B}_1{\mathbf V},{\mathbf V})|_{x_1=0}=2[\dot{q}(\dot{u}_N-\hat{\lambda}\dot{H}_N)],
\end{equation}
with $\hat\lambda:= \lambda(\hat{\mathbf U})$.
\newline
Using  \eqref{ho2} and \eqref{Hn2} (see also \eqref{Be1}) for $\dot{u}^{\pm}_{N}$ and $\dot{H}^{\pm}_N,$ (recall that we have dropped the index $\natural$ from the unknown) and the boundary constraint $[\dot{q}]=-\varphi[\partial_{1}\hat{q}]$ (recall the definition of $[\partial_1\hat q]$ in \eqref{jump_q}) we obtain that
%\begin{equation}\label{diss}
%\begin{split}
%[\dot{q}(\dot{u}_N-\hat{\lambda}\dot{H}_N)]=\dot{q}^+[\hat{u}_2-\hat\lambda\hat{H}_2]\partial_{2}\varphi + \{\text{lower order terms}\}
%\end{split}
%\end{equation}
%where
%\begin{equation*}
%\begin{split}
%\{\text{lower order terms}\}:=-\dot{q}^+\varphi [\partial_{1}\hat{u}_N-\hat{\lambda} \partial_{1}\hat{H}_N]-\varphi [\partial_{1}\hat{q}]\partial_t\varphi\\
%-\varphi [\partial_{1}\hat{q}]\left((\hat{u}_2^--\hat{\lambda}^-\hat{H}_2^-)\partial_{2}\varphi  + \varphi(\partial_{1}\hat{u}_N^--\hat{\lambda}^- \partial_{1}\hat{H}_N^-) \right)\,.
%\end{split}
%\end{equation*}
\begin{equation}\label{boundary-q-form_lambda_1}
(\mathcal B_1{\mathbf V}\cdot{\mathbf V})=2[\hat u_2-\hat\lambda\hat H_2]\dot q^+\partial_2\varphi+\mbox{l.o.t}\,,\quad\mbox{on}\,\,\,\{x_1=0\}\,,
\end{equation}
with
\begin{equation}\label{lot_1}
\begin{split}
\mbox{l.o.t.}:=&-2[\partial_1\hat u_N-\hat\lambda\partial_1\hat H_N]
\dot q^+\varphi-2[\partial_1\hat q]\varphi\partial_t\varphi-2[\partial_1\hat q](\hat u^-_2-\hat\lambda\hat H^-_2)\varphi\partial_2\varphi\\
&-2[\partial_1\hat q](\partial_1\hat u^-_N-\hat\lambda^-\partial_1\hat H^-_N)\varphi^2\,.
\end{split}
\end{equation}
We use ``l.o.t." to mean boundary terms that can be manipulated in the energy estimate by passing to volume integral and using integration by parts, so that they do not give any trouble in the derivation of the energy estimate, see Appendix \ref{proof-Thm-5-1}.
%%%%%%%%%%%%%%%%%%%%%%%%%scelta di $\hat\lambda$  %%%%%%%%%%%%%%%%%%%%%

\medskip
\noindent
Let us now make a suitable choice of the function $\hat{\lambda}=\lambda(\hat{\mathbf U})$. We first take $\hat{\lambda}$ as follows, see \cite{Trakhinin2009},
\begin{equation}\label{lambda}
\hat{\lambda}^\pm:=\lambda(\hat{{\mathbf U}}^{\pm}):=\eta(x_1)\lambda^{\pm}(t,x_2),
\end{equation}
with $\eta(x_1)$  a smooth monotone decreasing function satisfying $\eta(0)=1$ and $\eta(x_1)=0,$ for $x_1>\varepsilon$, with $\varepsilon>0$ sufficiently small. The functions $\lambda^{\pm}$ will be chosen below.
\begin{remark}
The motivation of the definition of $\hat{\lambda}$ by using the cut-off function $\eta=\eta(x_1)$ is that it guarantees that the hyperbolicity condition \eqref{B01}, that trivially holds for $\hat\lambda=0$ remains true for the basic state in a small neighborhood $\{0<x_1<\varepsilon\}$ of the boundary, thanks to the continuity of the basic state. Hence, by definition of $\eta$, the hyperbolicity condition still holds in the whole domain $\Omega_T$, see \cite{Trakhinin2009} for more details.
\end{remark}
The functions $\lambda^{\pm}(t,x_2)$ are chosen in this way: if the jump $[\hat{u}_2](t,x_2)=0,$ we define $\lambda^{\pm}(t,x_2)=0$; otherwise, if the jump $[\hat{u}_2](t,x_2)\neq0,$ we choose $\lambda^{+}(t,x_2)$ and $\lambda^{-}(t,x_2)$ satisfying the following relation
\begin{equation*}
[\hat{u}_2-\lambda\hat{H}_2]=0.
\end{equation*}
The following Lemma ensures the existence of such functions $\lambda^{+}(t,x_2)$, $\lambda^{-}(t,x_2)$ satisfying the above equation and the hyperbolicity assumption \eqref{B01} written for the basic state, see \eqref{hyp-constraints}.

%%%%%%%%%%%%%%%%%%%%%%%%NOSTRO LEMMA%%%%%%%%%%%%%
\begin{lemma}\label{lemma-Paolo-Ale-Paola}
For all $(t,x_2)\in \Gamma_T$, there exist $\lambda^{\pm}(t,x_2)$ satisfying
\begin{equation}\label{0-equality}
[\hat{u}_2-\lambda\hat{H}_2]=0,
\end{equation}
 and the hyperbolicity conditions
\begin{equation}\label{hyp-constraints}
|{\lambda}^{\pm}|< \hat a^\pm
\end{equation}
if and only if the basic states $\hat{\mathbf U}^\pm$ obey the stability estimate
\begin{equation}\label{new-stability}
|[\hat u_2]|<|\hat H^+_2| \hat a^++ |\hat H^-_2|\hat a^-\,,
\end{equation}
where $\hat a^\pm$ are defined in \eqref{apm}. In particular, under \eqref{new-stability} we can set
\begin{equation}\label{def-lambda}
\begin{split}
\lambda^+= sgn(\hat H^+_2) \frac{\hat a^+ [\hat u_2]}{\hat a^+ |\hat H^+_2| + \hat a^- |\hat H^-_2|},\\
\lambda^-=- sgn(\hat H^-_2) \frac{\hat a^- [\hat u_2]}{\hat a^+ |\hat H^+_2| + \hat a^- |\hat H^-_2|}.
\end{split}
\end{equation}
\end{lemma}
%\section{Proof of Lemma \ref{lemma-Paolo-Ale-Paola}}\label{proof-lemma}
%In this appendix we prove Lemma \ref{lemma-Paolo-Ale-Paola}.
\begin{proof}
	For shortness in the proof we drop the hat  and the subscript $2$ on the variables, moreover we do not write the considered variables are restricted along the boundary $\{x_1=0\}$ that is we write
	$$
	u=\hat u^\pm_2\vert_{x_1=0}, \quad H^\pm= \hat H^\pm_2\vert_{x_1=0},\quad\text{etc.}\,.
	$$
	Let us first note that equation \eqref{0-equality} can be restated as
	\begin{equation}\label{1-equality}
		[u]= \lambda^+ H^+ -\lambda^- H^-.
	\end{equation}
	Let us consider the simple case $H^+\neq 0$ and $H^-=0$. Then, from equation \eqref{1-equality} we get at once
	$$
	\lambda^+= \frac{[u]}{H^+}
	$$
	and plugging the above into the constraint in \eqref{hyp-constraints} for $\lambda^+$ gives
	\begin{equation}\label{3-Ale}
		|[u]| < a^+ |H^+|\,,
	\end{equation}
	that is \eqref{new-stability}. In this case $\lambda^-$ can be whatever function satisfying the corresponding constraint \eqref{hyp-constraints}
	$$
	|\lambda^-|<a^-
	$$
	(for instance $\lambda^-\equiv 0$); it does not yield any condition on the background state in addition to \eqref{3-Ale}. Of course the symmetric case $H^+=0$ and $H^-\neq 0$ is similar with
	$$
	\lambda^- =- \frac{[u]}{H^-}, \quad \lambda^+\equiv 0
	$$
	and the condition
	\begin{equation}\label{4-Ale}
		|[u]| < a^- |H^-|
	\end{equation}
	on the background state, which is again \eqref{new-stability}.
	\newline
	Let us consider now the case $H^+\neq 0$ and $H^-\neq 0$. From \eqref{1-equality}, we write $\lambda^+$ as a known function of $\lambda^-$ as
	\begin{equation}\label{lp}
		\lambda^+= \frac{[u] + \lambda^- H^-}{H^+}
	\end{equation}
	and plug the above into $|\lambda^+| < a^+$ to get, after simple manipulations,
	\begin{equation}\label{A-Ale}
		-a^+ - \frac{[u]}{H^+} < \frac{H^-}{H^+}\lambda^- < a^+ - \frac{[u]}{H^+}.
	\end{equation}
	We have that $\lambda^-$ must obey simultaneously the two constraints \eqref{A-Ale} and
	\begin{equation}\label{B-Ale}
		-a^- <\lambda^- < a^-.
	\end{equation}
	Let assume now that $\frac{H^-}{H^+}>0$. Then equation \eqref{A-Ale} becomes
	\begin{equation}\label{Aprime-Ale}
		b_2< \lambda^- < b_1,
	\end{equation}
	where
	\begin{equation}\label{def-b1,b2}
		b_1= \frac{H^+}{H^-} \left( a^+ - \frac{[u]}{H^+} \right)\quad\mbox{and}\quad  b_2= \frac{H^+}{H^-} \left(- a^+ - \frac{[u]}{H^+} \right).
	\end{equation}
	%If we assume
	%\begin{equation}
	%b_1 \leq - a^- \quad \text{or} \quad a^-\leq b_2
	%\end{equation}
	%then no $\lambda^-$ exists satisfying both \eqref{Aprime-Ale} and \eqref{B-Ale}. Therefore, if there exists $\lambda^-$ satisfying \eqref{Aprime-Ale} and \eqref{B-Ale}, it must hold
	A necessary condition for \eqref{B-Ale} and \eqref{Aprime-Ale} to hold simultaneously
	is that
	\begin{equation}\label{C-Ale}
		b_1 > - a^- \quad \text{and} \quad  a^- > b_2.
	\end{equation}
	After some calculations, in view of the definition \eqref{def-b1,b2}, the above becomes equivalent to
	\begin{equation*}
		- \left(\frac{H^+}{H^-} a^+ + a^-\right) <\frac{[u]}{H_-}< \frac{H^+}{H^-} a^+ + a^- .
	\end{equation*}
	Assuming $H^- <0$ (hence $H^+<0$)  the above is equivalent to
	%\begin{equation*}
	%\left( H^+ a^+ + H^- a^-\right) < [u] < - \left( H^+ a^+ + H^- a^-\right),
	%\end{equation*}
	%that is equivalent to
	\begin{equation*}
		|[u]| < - \left( H^+ a^+ + H^- a^-\right)= |H^+| a^+ + |H^-| a^-,
	\end{equation*}
	which is the estimate \eqref{new-stability}.
	Similar calculation in the case $H^- >0$ and $H^+ >0$  still yield to \eqref{new-stability}.
	
	To sum up if $\frac{H^-}{H^+} >0$ we get that condition \eqref{new-stability} is at least necessary in order to find $\lambda^-$ satisfying \eqref{B-Ale} and \eqref{Aprime-Ale}, that is
	\begin{equation*}
		\max\{ -a^-, b_2\} < \lambda^- < \min\{a^-, b_1 \}.
	\end{equation*}
	If such $\lambda^-$ actually exists, then $\lambda^+$ will be defined by \eqref{lp}.
	
	\medskip
	From \eqref{A-Ale}, arguing in the same way as above when $\frac{H^-}{H^+}<0$ we obtain once again that \eqref{new-stability} is at least necessary for the existence of such a $\lambda^-$ satisfying both \eqref{A-Ale} and \eqref{B-Ale}, that in this case are equivalent to
	\begin{equation*}
		\max\{ -a^-, b_1\} < \lambda^- < \min\{a^-, b_2 \}\,,
	\end{equation*}
	where $b_1$ and $b_2$ are defined in \eqref{def-b1,b2}.
	
	\medskip
	\noindent
	To complete the proof, it remains to show that condition \eqref{new-stability} is also sufficient for the existence of functions $\lambda^\pm$ satisfying \eqref{0-equality} and \eqref{hyp-constraints}.
	\newline
	We already proved it in the simplest cases $H^+\neq0$ and $H^-=0$ or viceversa. Now let us do the same in the case $H^+\neq 0$ and $H^-\neq 0$. So let us assume that the background state satisfies condition \eqref{new-stability}. If we assume for a while that $\lambda^\pm$ exist and satisfy \eqref{0-equality} and \eqref{hyp-constraints}, we derive that
	\[
	\vert[u]\vert\le\vert\lambda^+\vert\vert H^+\vert+\vert\lambda^-\vert\vert H^-\vert< a^+\vert H^+\vert+a^-\vert H^-\vert\,,
	\]
	that is
	\begin{equation}\label{6-Ale}
		\frac{\vert[u]\vert}{a^+\vert H^+\vert+a^-\vert H^-\vert}<1\,.
	\end{equation}
	The last inequality \eqref{6-Ale} suggests to take $\lambda^\pm$ such that
	\begin{equation}\label{mod-lambda}
		\vert\lambda^\pm\vert=\frac{a^\pm\vert[u]\vert}{a^+\vert H^+\vert+a^-\vert H^-\vert}\,.
	\end{equation}
	Indeed from \eqref{6-Ale} it immediately follows that $\lambda^\pm$ as above satisfy the constraint \eqref{hyp-constraints}. Formula \eqref{mod-lambda} defines $\lambda^\pm$ up to the sign. One can directly check that taking
	\begin{equation}\label{lambdapm}
		\lambda^+=sgn(H^+)\frac{a^+ [u]}{a^+\vert H^+\vert+a^-\vert H^-\vert}\quad\mbox{and}\quad \lambda^-=-sgn(H^-)\frac{a^- [u]}{a^+\vert H^+\vert+a^-\vert H^-\vert}
	\end{equation}
	makes the equation \eqref{0-equality} to be satisfied. The above definitions of $\lambda^\pm$ are just equal to \eqref{def-lambda}. This ends the proof.
\end{proof}

	\begin{remark}\label{rhs-stability cond}
	In order to make a comparison between our stability condition \eqref{new-stability} and the {\em subsonic} part of the stability condition found in \cite{WangYARMA2013}, let us
consider the case of background piece-wise constant state
		\begin{equation}\label{backstate}
			\hat{\mathbf U}^\pm:=(\hat{\rho}^\pm, 0, \hat{u}^\pm_2, 0, \hat{H}_2^\pm, \hat S^\pm)^T\,.
		\end{equation}
		Then the right-hand side of estimate \eqref{new-stability} can be restated in terms of sound speeds $\hat c^\pm$ and the Alfv\'en speeds $\hat c^\pm_A=\frac{\vert\hat{\mathbf  H}^\pm\vert}{\sqrt{\rho^\pm}}=\frac{\vert\hat H^\pm_2\vert}{\sqrt{\rho^\pm}}$ as follows
		\[
		\hat a^+\vert \hat{H}_2^+\vert+\hat a^-\vert\hat{H}_2^-\vert=\frac{\hat c^+\hat c^+_A}{((\hat c^+)^2+(\hat c^+_A)^2)^{1/2}}+\frac{\hat c^-\hat c^-_A}{((\hat c^-)^2+(\hat c^-_A)^2)^{1/2}}
		\]
		and \eqref{new-stability} becomes
		\begin{equation}\label{new-stability-speed}
			\vert[\hat u_2]\vert<\frac{\hat c^+\hat c^+_A}{((\hat c^+)^2+(\hat c^+_A)^2)^{1/2}}+\frac{\hat c^-\hat c^-_A}{((\hat c^-)^2+(\hat c^-_A)^2)^{1/2}}\,.
		\end{equation}
		
\bigskip
In their two dimensional linear stability analysis \cite{WangYARMA2013}, Wang and Yu are able to perform a complete normal modes analysis of the linearized problem for an {\em isentropic flow}, when the planar piece-wise constant basic state \eqref{backstate} (without $\hat S^\pm$)	
%		\begin{equation*}
%			\hat{\mathbf U}^\pm:=(\hat{\rho}^\pm, 0, \hat{u}^\pm_2, 0, \hat{H}_2^\pm)^T
%		\end{equation*}
		satisfies suitable technical restrictions. Precisely, the constant components of $\hat{{\mathbf U}}^\pm$ are required to satisfy, besides the Rankine--Hugoniot conditions
		\begin{equation}\label{RH-Wang-Yu}
			\hat{u}_{2}^++\hat{u}_{2}^-=0,\quad \hat{u}_{2}^+>0,\quad \hat{p}^++\frac{(\hat{H}^+_2)^2}{2}=\hat{p}^-+\frac{(\hat{H}^+_{2})^2}{2}\,,
		\end{equation}
the following additional assumptions
		\begin{equation}\label{WY-restrictions}
			|\hat{H}_{2}^+|=|\hat{H}_{2}^-|=:|\hat{H}_2|,\quad \hat{c}>\hat{c}_A=\sqrt{\frac{|\hat{H}_2|^2}{\hat\rho}}
		\end{equation}
		(here the notation of \cite{WangYARMA2013} is adapted to our current setting).
	 Notice in particular that, since the flow is isentropic, the assumption
		$|\hat H_{2}^+|=|\hat H_{2}^-|$ and the last Rankine--Hugoniot condition in \eqref{RH-Wang-Yu} imply $\hat\rho^-=\hat\rho^+=:\hat\rho$ so that the sound speed and the Alfv\'{e}n speed take the same value $\hat c=c(\hat\rho):=\sqrt{p^\prime(\hat\rho)}$ and $\hat c_A:=\sqrt{\frac{\vert\hat H_2\vert^2}{\hat\rho}}$ on both $\pm$-states of the flow.
		\newline
		The linear stability conditions by Wang-Yu \cite{WangYARMA2013} read
		as follows:
\begin{equation}\label{WYstability}		
		(\hat{u}_{2}^+)^2<\hat c^2-\hat c^2\sqrt{\frac{\hat c^2-\hat c^2_A}{\hat c^2+\hat c^2_A}}\qquad\text{or}\qquad(\hat{u}^+_2)^2>\hat c^2+\hat c^2\sqrt{\frac{\hat c^2-\hat c^2_A}{\hat c^2+\hat c^2_A}}
		\,.
		\end{equation}
		The left inequality in \eqref{WYstability}
%		\begin{equation}\label{WYsubsonic}
%			(\hat u_2^+)^2<\hat c^2-\hat c^2\sqrt{\frac{\hat c^2-\hat c_A^2}{\hat c^2+\hat c_A^2}}\,.
%		\end{equation}
		identifies a ``subsonic'' weak stability region, becoming empty in the absence of a magnetic field (namely for compressible Euler equations), whereas the right inequality corresponds to the ``supersonic'' weak stability region of $2D$ compressible vortex sheets (to which it reduces, when formally setting $\hat c_A=0$). With respect to \eqref{WYstability}, our stability condition in \eqref{new-stability-speed} provides only a {\em subsonic} stability domain. For a planar isentropic background state obeying \eqref{WY-restrictions}, condition \eqref{new-stability-speed} reduces to
%		Then the subsonic part reads
%		\begin{equation}\label{WYsubsonic}
%			(\hat u_2^+)^2<\hat c^2-\hat c^2\sqrt{\frac{\hat c^2-\hat c_A^2}{\hat c^2+\hat c_A^2}}\,.
%		\end{equation}
%		In the case of a piecewise constant background state $\hat {\mathbf U}^\pm$ under the same restrictions as in \eqref{WY-restrictions}, condition \eqref{new-stability-speed} reduces to
		%\[
		%\vert[\hat u_2]\vert=2\vert\hat u_2^+\vert<2\frac{\hat c_A\hat c}{((\hat c)^2+(\hat c_A)^2)^{1/2}},
		%\]
		%which leads to
		\begin{equation}\label{new-stability-speed-reduced}
			(\hat u_2^+)^2<\frac{\hat c^2_A\hat c^2}{(\hat c)^2+(\hat c_A)^2}.
		\end{equation}
		To compare the left inequality in \eqref{WYstability} and \eqref{new-stability-speed-reduced}, one can easily check that
		\[
		\frac{\hat c^2_A\hat c^2}{(\hat c)^2+(\hat c_A)^2}<\hat c^2-\hat c^2\sqrt{\frac{\hat c^2-\hat c_A^2}{\hat c^2+\hat c_A^2}}\,,
		\]
		so that our subsonic condition \eqref{new-stability-speed-reduced} is more restrictive than the one in \cite{WangYARMA2013}.
		\newline
		Let us also note that we are not able, by our approach, to obtain a counterpart of the {\em supersonic} condition in \cite{WangYARMA2013} for isentropic flows in the constant coefficients case (that is the right inequality in \eqref{WYstability}). On the other hand, our stability condition applies also to {\em nonisentropic} flows and to general ({\em non piece-wise constant}) background states.
	\end{remark}

We are now in the position to state the well-posedness of the ``homogeneous" linearized problem \eqref{ho2}.
\begin{theorem}\label{th-wp-hom}
Let all assumptions \eqref{hy}-\eqref{c2} be satisfied for the basic state \eqref{basic} (note that \eqref{stability1} implies the stability condition \eqref{new-stability}). Then, for all $\mathbf{F}\in H^1_{\ast}(\Omega_T)$ that vanish in the past, problem \eqref{ho2} has a unique solution $(\dot{{\mathbf U}}^\natural,\varphi)\in H^{1}_{\ast}(\Omega_T)\times H^1(\Gamma_T).$ The solution satisfies the following a priori estimate
\begin{equation}\label{est-wp-hom}
||\dot{{\mathbf U}}^\natural||_{1,\ast,T}+||\varphi||_{H^1(\Gamma_T)}\leq C ||\mathbf{F}||_{1,\ast,T},
\end{equation}
where $C=C(K,T)>0$ is a constant independent of the data $\mathbf{F}$.
\end{theorem}
The proof of Theorem \ref{th-wp-hom} is given in Appendix \ref{proof-Thm-5-1}.

\smallskip
As a consequence of Lemma \ref{lemma4.2} we have the following well-posedness result for the nonhomogeneous problem \eqref{IVP2}.
\begin{theorem}
Let all assumptions of Theorem \ref{th-wp-hom} be satisfied for the basic state \eqref{basic}. Then, for all $(\mathbf{f},\mathbf{g})\in H^3_{\ast}(\Omega_T)\times H^{3}(\Gamma_T)$ that vanish in the past, problem \eqref{IVP2} has a unique solution $(\dot{{\mathbf U}},\varphi)\in H^{1}_{\ast}(\Omega_T)\times H^1(\Gamma_T)$. The solution satisfies the following a priori estimate
\begin{equation}
||\dot{{\mathbf U}}||_{1,\ast,T}+||\varphi||_{H^1(\Gamma_T)}\leq C\Big(||\mathbf{f}||_{3,\ast,T}+||\mathbf{g}||_{H^3(\Gamma_T)}\Big),
\end{equation}
where $C=C(K,T)>0$ is a constant independent of the data $\mathbf{f},\mathbf{g}.$
\end{theorem}
\section{Higher-order energy estimates for the homogeneous problem \eqref{ho2}}\label{tameestimate}
In order to get an a \textit{priori} tame estimate in $H^s_{\ast}$ for solution of problem \eqref{IVP2} with $s$ sufficiently large, as a preliminary step we derive a tame estimate for the homogenous problem \eqref{ho2}, that we state in the following theorem.

For shortness, in all this section we write $\dot{{\mathbf U}}^\natural=\dot{{\mathbf U}}$.

\begin{theorem}\label{tame}
	Let $T>0$ and $s$ be positive integer, $s\geq6.$
	Assume that the basic state $(\breve{\mathbf U}^{\pm},\hat{\varphi})\in H^{s+4}_{\ast}(\Omega_T)\times H^{s+5}(\Gamma_T),$
	\begin{equation}\label{assumption}
		||\breve{\mathbf U}^{\pm}||_{9,\ast,T}+||\hat{\varphi}||_{H^{10}(\Gamma_T)}\leq\hat{K}, \text{ where } \breve{\mathbf U}^{\pm}:=\hat{\mathbf {U}}^{\pm}-\bar{\mathbf{U}}.
	\end{equation}
	Assume that $\mathbf{F}\in H^s_{\ast}(\Omega_T)$ vanishes in the past. Then, there exists a positive constant $K_0,$ that does not depend on $s$ and $T,$ and there exists a constant $C(K_0)>0$ such that if $\hat{K}\leq K_0,$ then there exists a unique solution $(\dot{{\mathbf U}},\varphi)\in H^s_{\ast}(\Omega_T)\times H^s(\Gamma_T)$ to homogeneous problem \eqref{ho2} satisfying the following estimate:
	\begin{equation}\label{tameess2}
		\begin{split}
			||\dot{{\mathbf U}}||_{s,\ast,T}+||\varphi||_{H^s(\Gamma_T)}\leq C(K_0)\Big(||\mathbf{F}||_{s,\ast,T}+||\mathbf{F}||_{6,\ast,T}||\hat{W}||_{s+4,\ast,T}\Big)\,,
		\end{split}
	\end{equation}
	for $T$ small enough, where $\hat{W}=(\breve{\mathbf{U}},\nabla_{t,\bf x}\hat{\Psi}).$
\end{theorem}

To prove the above theorem we need to obtain several higher-order energy estimates.

\subsection{Estimate of the normal derivative of the ``non-characteristic'' unknown}
In this section we will prove the estimate of the normal derivative of the ``non-characteristic'' unknown
$${\mathbf V}_n=({\mathbf V}^+_n,{\mathbf V}^-_n),\quad {\mathbf V}^{\pm}_n=(\dot{q}^{\pm},\dot{u}^{\pm}_n,\dot{H}^{\pm}_n).$$
Let $s$ be positive integer, we need to estimate $\partial_1{\mathbf V}_n$ in $H^{s-1}_{\ast},$  obtained from \eqref{A} and the divergence constraint \eqref{div2} as follows
\begin{equation}\label{Vn}
\partial_1{\mathbf V}^+_n=\left[\begin{array}{c}
 \mathcal{K}_{1}\\
 \mathcal{K}_{2}\\
  -\partial_2(\dot{H}^+_2\partial_1\hat{\Phi}^+)
  \end{array}\right],\quad \partial_1{\mathbf V}^-_n=\left[\begin{array}{c}
 \mathcal{K}_{7}\\
 \mathcal{K}_{8}\\
  -\partial_2(\dot{H}^-_2\partial_1\hat{\Phi}^-)
  \end{array}\right]\,.
\end{equation}
Here $\mathcal{K}_{i}\in\R$ is the $i-$th scalar component of the vector
\begin{equation}\label{R}
\mathcal{K}:=\tilde{\mathcal{A}}\Big(\mathcal{F}-\mathcal{A}_0\partial_t{\mathbf V}-\mathcal{A}_2\partial_2{\mathbf V}-\mathcal{A}_3{\mathbf V}-\mathcal{A}_{(0)}\partial_1{\mathbf V}\Big),
\end{equation}
where $\tilde{\mathcal{A}}=\mathrm{diag}\Big((\partial_1\hat{\Phi}^+)\mathcal{E}_{12},(\partial_1\hat{\Phi}^-)\mathcal{E}_{12}\Big),$ $\partial_1\hat{\Phi}^{\pm}=\pm1+\partial_1\hat{\Psi}^{\pm},$ $\mathcal{F}=J^T\mathbf{F}.$

Now, we are ready to prove the following Lemma, which is needed in the proof of weighted normal derivatives and non-weighted tangential derivatives, see Section \ref{wd} and Section \ref{nonwt}.
\begin{lemma}\label{noncha} The following estimate
\begin{equation}\label{noncha3}
||\partial_1{\mathbf V}_n||^2_{s-1,\ast,t}\leq C(K)\mathcal{M}(t),
\end{equation}
with
\begin{equation}\label{defM}
\begin{split}
\mathcal{M}(t)&=||({\mathbf V},\mathbf{F})||^2_{s,\ast,t}+(||\dot{{\mathbf U}}||^2_{W^{2,\infty}_{\ast}(\Omega_t)}+||\mathbf{F}||^2_{W^{1,\infty}_{\ast}(\Omega_t)})||\hat{W}||^2_{s+2,\ast,t}\\
\end{split}
\end{equation}
holds for problem \eqref{ho2} for all $t\leq T.$
\end{lemma}
\begin{proof}
Using Moser-type calculus inequalities \eqref{moser4} and \eqref{moser6}, we can estimate the right-hand side of \eqref{R}:
\begin{equation}\label{Fes}
\begin{split}
||\tilde{\mathcal{A}}\mathcal{F}||^2_{s-1,\ast,t}\leq C(K)\Big(||\mathbf{F}||^2_{s-1,\ast,t}+||\mathbf{F}||^2_{W^{1,\infty}_{\ast}(\Omega_t)}||\hat{W}||^2_{s-1,\ast,t}\Big).
\end{split}
\end{equation}
For $j=0,2,$
\begin{equation}
\begin{split}
||\tilde{\mathcal{A}}\mathcal{A}_j\partial_j{\mathbf V}||^2_{s-1,\ast,t}&\lesssim ||\tilde{\mathcal{A}}\mathcal{A}_j{\mathbf V}||^2_{s,\ast,t}\\
&\leq C(K)\Big(||{\mathbf V}||^2_{s,\ast,t}+||\dot{{\mathbf U}}||^2_{W^{1,\infty}_{\ast}(\Omega_t)}||\hat{W}||^2_{s,\ast,t}\Big),
\end{split}
\end{equation}
\begin{equation}
\begin{split}
||\tilde{\mathcal{A}}\mathcal{A}_3{\mathbf V}||^2_{s-1,\ast,t}\leq C(K)\Big( ||{\mathbf V}||^2_{s-1,\ast,t}+||\dot{{\mathbf U}}||^2_{W^{1,\infty}_{\ast}(\Omega_t)}||\hat{W}||^2_{s+1,\ast,t}\Big),
\end{split}
\end{equation}
where $C(K)$ stands for positive constants that depends on $K$ and we used that $\tilde{\mathcal{A}}J^T$, $\tilde{\mathcal{A}}\mathcal{A}_j$ for $j=0,2$ and $\tilde{\mathcal{A}}\mathcal{A}_3$ are all nonlinear smooth functions of the basic state $\hat{W}$.

\smallskip
Now, we estimate the last term in \eqref{R}, where for simplicity we denote $\mathcal A=\tilde{\mathcal A}\mathcal A_{(0)}$. We get
\begin{equation}
\begin{split}
||\mathcal{A}\partial_1{\mathbf V}||_{s-1,\ast,t}&=||\frac{\mathcal{A}}{\sigma}\sigma\partial_1{\mathbf V}||_{s-1,\ast,t}\\
&\le||\frac{\mathcal{A}}{\sigma}||_{W^{1,\infty}_\ast(\Omega_t)}||\sigma\partial_1{\mathbf V}||_{s-1,\ast,t}+||\frac{\mathcal{A}}{\sigma}||_{s-1,\ast,t}||\sigma\partial_1{\mathbf V}||_{W^{1,\infty}_\ast(\Omega_t)}\\
&\le ||\mathcal{A}||_{W^{2,\infty}_\ast(\Omega_t)}||{\mathbf V}||_{s,\ast,t}+||\mathcal{A}||_{s+1,\ast,t}||{\mathbf V}||_{W^{2,\infty}_\ast(\Omega_t)}\\
&\le||\hat{W}||_{W^{2,\infty}_\ast(\Omega_t)}||{\mathbf V}||_{s,\ast,t}+||\hat{W}||_{s+1,\ast,t}||{\mathbf V}||_{W^{2,\infty}_\ast(\Omega_t)}\,,
\end{split}
\end{equation}
where we used that $\mathcal A\vert_{x_1=0}=0$, thus the $L^\infty-$norm of $\mathcal A/\sigma$ can be estimated by the $L^\infty$-norms of $\mathcal A$ and $\partial_1\mathcal A$, see \cite[Lemma B.9]{MTP2009}, \cite{Secchi2012}.

For the rest of the term in \eqref{Vn}, using \eqref{moser4} and \eqref{moser6}, we obtain that
\begin{equation}\label{333}
\begin{split}
||\partial_2(\dot{H}^+_2\partial_1\hat{\Phi}^{+})||^2_{s-1,\ast,t}& \leq||\dot{H}^+_2\partial_1\hat{\Phi}^{+}||^2_{s,\ast,t}\\
 & \leq C(K)\Big(||{\mathbf V}||^2_{s,\ast,t}+||\dot{{\mathbf U}}||^2_{W^{1,\infty}_{\ast}(\Omega_t)}||\hat{W}||^2_{s,\ast,t}\Big).
 \end{split}
\end{equation}
The second term in \eqref{Vn} can be controlled similarly.
Summarizing all the estimates \eqref{Fes}-\eqref{333} for the terms in \eqref{Vn} and \eqref{R},  Lemma \ref{noncha} is concluded.
\end{proof}

We also need the following estimate, still for $\partial_1{\mathbf V}_n$, which is also essential in the proof of non-weighted tangential derivatives, see \eqref{6.39} in Section \ref{nonwt}.
\begin{lemma}
The estimate
\begin{equation}\label{noncha32}
|||\partial_1{\mathbf V}_n(t)|||^2_{s-1,\ast}\leq C(K)\Big(|||{\mathbf V}(t)|||^2_{s,\ast}+\mathcal{M}(t)\Big)
\end{equation}
holds for problem \eqref{ho2} for all $t\leq T$, where $s$ is a positive integer and $\mathcal{M}(t)$ is defined in \eqref{defM}.
\end{lemma}
\begin{proof}
Denote the differential operator $D^{\alpha}_{\ast}=\partial^{\alpha_0}_t(\sigma\partial_1)^{\alpha_1}\partial^{\alpha_2}_2\partial^{\alpha_3}_1$, $\langle\alpha\rangle:=|\alpha|+\alpha_3.$ We estimate the right-hand side of \eqref{R}. Using the elementary estimate
\begin{equation}\label{ele}
|||u(t)|||^2_{s-1,\ast}\lesssim||u||^2_{s,\ast,t},
\end{equation}
we get
\begin{equation}\label{FF}
\begin{split}
|||(\tilde{\mathcal{A}}\mathcal{F})(t)|||^2_{s-1,\ast}&\lesssim||\tilde{\mathcal{A}}J^T\mathbf{F}||^2_{s,\ast,t}\\
&\leq C(K)\Big(||\mathbf{F}||^2_{s,\ast,t}+||\mathbf{F}||^2_{W^{1,\infty}_{\ast}(\Omega_t)}||\hat{W}||^2_{s,\ast,t}\Big)\\
&\leq C(K)\mathcal{M}(t).
\end{split}
\end{equation}
The second term of \eqref{R} can be controlled as follows:
\begin{equation*}
\begin{split}
|||\tilde{\mathcal{A}}\mathcal{A}_{0}&\partial_t{\mathbf V}(t)|||^2_{s-1,\ast}\lesssim\sum\limits_{\langle\alpha\rangle\le s-1}||\tilde{\mathcal{A}}\mathcal{A}_{0}D^\alpha_\ast\partial_t{\mathbf V}(t)||^2_{L^2(\mathbb R^2_+)}\\
&+\sum\limits_{\langle\alpha\rangle\le s-1}||D_\ast(\tilde{\mathcal{A}}\mathcal{A}_{0})D^{\alpha-1}_\ast\partial_t{\mathbf V}(t)||^2_{L^2(\mathbb R^2_+)}+\sum\limits_{\langle\alpha\rangle=2}||D^\alpha_\ast(\tilde{\mathcal{A}}\mathcal{A}_{0})\partial_t{\mathbf V}(t)||^2_{s-3,\ast}\\
&=\Sigma_1^\prime+\Sigma^\prime_2+\Sigma_3^\prime\,,
\end{split}
\end{equation*}
with $D_\ast$ denoting any tangential derivative in $t$ or ${\bf x}$ of order one and $D^{\alpha-1}_\ast$ denoting the derivative of order $\langle\alpha\rangle-1$ obtained \lq\lq subtracting" $D_\ast$ from $D_{\ast}^\alpha$.
\newline
We estimate separately each $\Sigma^\prime_i$, $i=1,2,3$ as follows:
\begin{equation*}
\Sigma_1^\prime\lesssim||\hat W||^2_{L^\infty(\Omega_t)}|||\mathbf{V}(t)|||^2_{s,\ast}\le C(K)|||\mathbf{V}(t)|||^2_{s,\ast}\,;
\end{equation*}
\begin{equation*}
\Sigma_2^\prime\lesssim||\hat W||^2_{W^{1,\infty}_\ast(\Omega_t)}|||\mathbf{V}(t)|||^2_{s-1,\ast}\le C(K)|||\mathbf{V}(t)|||^2_{s,\ast}\,;
\end{equation*}
\begin{equation*}
\begin{split}
\Sigma_3^\prime&\lesssim\sum\limits_{\langle\alpha\rangle=2}||D^\alpha_\ast(\tilde A\mathcal A_0)\partial_t\mathbf{V}|||^2_{s-2,\ast,t}\\
&\lesssim\sum\limits_{\langle\alpha\rangle=2}\left(||D^\alpha_\ast(\tilde A\mathcal A_0)||^2_{W^{1,\infty}_\ast(\Omega_t)}||\partial_t\mathbf{V}||^2_{s-2,\ast,t}+||D^\alpha_\ast(\tilde{\mathcal A}\mathcal A_0)||^2_{s-2,\ast,t}||\partial_t\mathbf{V}||^2_{W^{1,\infty}_\ast(\Omega_t)}\right)\\
&\le C(K)||\mathbf V||^2_{s-1,\ast,t}+||\hat W||^2_{s,\ast,t}||\mathbf V||^2_{W^{2,\infty}_\ast(\Omega_t)}\,.
\end{split}
\end{equation*}
Adding the above inequalities, we get
\begin{equation}\label{AAA}
|||\tilde{\mathcal{A}}\mathcal{A}_{0}\partial_t{\mathbf V}(t)|||^2_{s-1,\ast}\leq C(K)|||{\mathbf V}(t)|||^2_{s,\ast}+C(K)||\mathbf V||^2_{s-1,\ast,t}+||\hat W||^2_{s,\ast,t}||\mathbf V||^2_{W^{2,\infty}_\ast(\Omega_t)}\,.
\end{equation}
Similarly, the third term can be estimated by
\begin{equation}\label{A000}
|||\tilde{\mathcal{A}}\mathcal{A}_2\partial_2{\mathbf V}(t)|||^2_{s-1,\ast}\leq C(K)|||{\mathbf V}(t)|||^2_{s,\ast}+C(K)||\mathbf V||^2_{s-1,\ast,t}+||\hat W||^2_{s,\ast,t}||\mathbf V||^2_{W^{2,\infty}_\ast(\Omega_t)}.
\end{equation}
Using Moser-type calculus inequalities \eqref{moser4} and \eqref{moser6},
we can estimate
\begin{equation}\label{AA33}
\begin{split}
|||(\tilde{\mathcal{A}}\mathcal{A}_3\mathbf V)(t)|||^2_{s-1,\ast}&\lesssim||\tilde{\mathcal{A}}\mathcal{A}_3\mathbf V||^2_{s,\ast,t}\\
&\leq C(K)\Big(||\mathbf V||^2_{s,\ast,t}+||\dot{{\mathbf U}}||^2_{W^{1,\infty}_{\ast}(\Omega_t)}||\hat{W}||^2_{s+2,\ast,t}\Big)\\
&\leq C(K)\mathcal{M}(t).
\end{split}
\end{equation}
Now, we estimate the last term $\tilde{\mathcal{A}}\mathcal{A}_{(0)}\partial_1{\mathbf V}:$
\begin{equation*}
|||(\tilde{\mathcal{A}}\mathcal{A}_{(0)}\partial_1{\mathbf V})(t)|||^2_{s-1,\ast}\lesssim \Sigma_1+\Sigma_2+\Sigma_3\,,
\end{equation*}
where:
\begin{equation*}
\begin{split}
\Sigma_1&=\sum_{\langle\alpha\rangle\leq s-1}||\tilde{\mathcal{A}}\mathcal{A}_{(0)}D^{\alpha}_{\ast}\partial_1{\mathbf V}(t)||^2_{L^2(\mathbb R^2_+)}\\
&=\sum_{\langle\alpha\rangle\leq s-1}||\frac{\tilde{\mathcal{A}}\mathcal{A}_{(0)}}{\sigma}\sigma D^{\alpha}_{\ast}\partial_1{\mathbf V}(t)||^2_{L^2(\mathbb R^2_+)}\\
&\lesssim ||\tilde{\mathcal{A}}\mathcal{A}_{(0)}||^2_{W^{1,\infty}(\Omega_t)}|||\mathbf{V}(t)|||^2_{s,\ast}\lesssim ||\hat W||_{W^{1,\infty}(\Omega_t)}^2|||\mathbf{V}(t)|||^2_{s,\ast}\,;
\end{split}
\end{equation*}
\begin{equation*}
\begin{split}
\Sigma_2 &=\sum\limits_{\langle\alpha\rangle\leq s-1}||\frac{D_\ast(\tilde{\mathcal{A}}\mathcal{A}_{(0)})}{\sigma}\sigma D^{\alpha-1}_{\ast}\partial_1{\mathbf V}(t)||^2_{L^2(\mathbb R^2_+)}\\
&\lesssim ||D_\ast(\tilde{\mathcal A}\mathcal A_{(0)})||^2_{W^{1,\infty}(\Omega_t)}|||\mathbf{V}(t)|||^2_{s-1,\ast}\lesssim ||\hat W||^2_{W^{2,\infty}_\ast(\Omega_t)}|||\mathbf{V}(t)|||^2_{s-1,\ast}\,;
\end{split}
\end{equation*}
\begin{equation*}
\begin{split}
\Sigma_3 &=\sum\limits_{\langle\alpha\rangle=2}|||D^\alpha_\ast(\tilde{\mathcal A}\mathcal A_{(0)})\partial_1\mathbf V(t)|||^2_{s-3,\ast}\lesssim\sum\limits_{\langle\alpha\rangle=2}|||D^\alpha_\ast(\tilde{\mathcal A}\mathcal A_{(0)})\partial_1\mathbf V|||^2_{s-2,\ast,t}\\
&\lesssim ||\hat W||^2_{W^{3,\infty}(\Omega_t)}||\mathbf V||^2_{s,\ast,t}+||\hat W||^2_{s,\ast,t}||\mathbf V||^2_{W^{2,\infty}_\ast(\Omega_t)},
\end{split}
\end{equation*}
and where we exploit again the vanishing of $\tilde{\mathcal A}\mathcal A_{(0)}$ and $D_\ast(\tilde{\mathcal A}\mathcal A_{(0)})$ along the boundary $\{x_1=0\}$ in the estimates of $\Sigma_1$ and $\Sigma_2$, see \cite[Lemma B.9]{MTP2009}, \cite{Secchi2012}.
\newline
Adding the estimates of $\Sigma_i$, $i=1,2,3$ above, we get
\begin{equation}\label{AA0}
|||(\tilde{\mathcal{A}}\mathcal{A}_{(0)}\partial_1{\mathbf V})(t)|||^2_{s-1,\ast}\leq C(K)|||{\mathbf V}(t)|||^2_{s,\ast}+C(K)||\mathbf V||^2_{s,\ast,t}+||\hat W||^2_{s,\ast,t}||\mathbf V||^2_{W^{2,\infty}_\ast(\Omega_t)}\,.
\end{equation}

%Denote commutator $[a,b]c:=a(bc)-b(ac).$ Using \eqref{ele}, we obtain that
%\begin{equation}\nonumber
%\begin{split}
%\mathcal{K}_4&=\sum_{\langle\beta\rangle\leq s-1}||[D^{\beta}_{\ast},\tilde{\mathcal{A}}\mathcal{A}_{(0)}]\partial_1{\mathbf V}(t)||^2_{L^2(\Omega)}\\
%&\leq \sum_{\langle\beta\rangle\leq s-1}||[D^{\beta}_{\ast},\tilde{\mathcal{A}}\mathcal{A}_{(0)}]\partial_1{\mathbf V}||^2_{1,\ast,t}\\
%&\leq C(K)\Big(||\mathbf V||^2_{s,\ast,t}+||\dot{{\mathbf U}}||^2_{W^{1,\infty}_{\ast}(\Omega_t)}||\hat{W}||^2_{s+2,\ast,t}\Big).
%\end{split}
%\end{equation}
%We write $D^{\beta}=D^{\alpha}_{\star}\partial^k_1,$ $\langle\beta\rangle=\alpha+2k, D^{\alpha}_{\star}=\partial^{\alpha_0}_t(\sigma\partial_1)^{\alpha_1}\partial^{\alpha_2}_2.$
%Using elementary estimate \eqref{ele}, we can estimate $\mathcal{K}_5$ as follows:
%\begin{equation}\nonumber
%\begin{split}
%\mathcal{K}_5&=\sum_{\langle\beta\rangle\geq2,\langle\beta'\rangle+\langle\beta''\rangle\leq s-1}||(D^{\beta'}_{\ast}(\tilde{\mathcal{A}}\mathcal{A}_{(0)})D^{\beta''}_{\ast}{\partial_1\mathbf V})(t)||^2_{L^2(\Omega)}\\
%&\leq C(K)\Big(||{\mathbf V}||^2_{s,\ast,t}+||\dot{{\mathbf U}}||^2_{W^{2,\infty}_{\ast}(\Omega_t)}(1+||\hat{W}||^2_{s+4,\ast,t})\Big)\\
%&\leq C(K)\mathcal{M}(t).
%\end{split}
%\end{equation}
For the rest of the term in \eqref{Vn}, using \eqref{moser4} and \eqref{moser6}, we obtain that
\begin{equation}\label{3333}
\begin{split}
|||\partial_2(\dot{H}^+_2\partial_1\hat{\Phi}^{+})(t)|||^2_{s-1,\ast,t}& \leq||\dot{H}^+_2\partial_1\hat{\Phi}^{+}||^2_{s,\ast,t}\\
 & \leq C(K)\Big(||{\mathbf V}||^2_{s,\ast,t}+||\dot{{\mathbf U}}||^2_{W^{1,\infty}_{\ast}(\Omega_t)}||\hat{W}||^2_{s+2,\ast,t}\Big)\\
 & \leq C(K)\mathcal{M}(t).
 \end{split}
\end{equation}
The second term in \eqref{Vn} can be controlled similarly.
Summarizing from \eqref{FF}-\eqref{3333}, we conclude that \eqref{noncha32} holds.
\end{proof}

The following lemma gives the estimate of the normal derivative of the entropy, which is treated differently from the other components of the vector ${\mathbf V}$.
\begin{lemma}
	The estimate
	\begin{equation}\label{tame-entropia}
		|||S^\pm(t)|||^2_{s,\ast}\leq C(K)\mathcal{M}(t)
	\end{equation}
	holds for problem \eqref{ho2} for all $t\leq T$, where $s$ is a positive integer and $\mathcal{M}(t)$ is defined in \eqref{defM}.
\end{lemma}
\begin{proof}
The linearized equation for the entropy $ S^\pm$ is an evolution-like equation because the coefficient of the normal derivative of the entropy vanishes on the boundary; this yields that no boundary conditions are needed to be coupled to the equation in order to derive an apriori energy estimate. So, to estimate $ S^\pm$, we just handle the equation of $ S^\pm$ alone by the standard energy method tools.
The details of the proof are similar to those of the following Lemma \ref{estimate1}.
\end{proof}

Now, we derive weighted derivative estimates.
\subsection{Estimate of weighted derivatives}\label{wd}
Since the differential operators $(\sigma\partial_1)^{\alpha_1}$ and $\sigma^{\alpha_1}\partial^{\alpha_1}_1$ are equivalent, see \cite{Ohno}, in the following we discuss the term $D^{\alpha}_{\ast}{\mathbf V}$, with $\alpha_1>0$ and $\langle\alpha\rangle=|\alpha|+\alpha_3\leq s$, in its equivalent form $\sigma^{\alpha_1}D^{\alpha^\prime}_{t,x}\partial^{\alpha_3}_1\mathbf V$, where $D^{\alpha^\prime}_{t,x}:=\partial^{\alpha_0}_t\partial^{\alpha_1}_{1}\partial^{\alpha_2}_{2}$ ($\alpha^\prime=(\alpha_0,\alpha_1,\alpha_2)$).
\begin{lemma}\label{estimate1}
The following estimate holds for \eqref{ho2} for all $t\leq T:$
\begin{equation}\label{DVV}
\sum_{\langle\alpha\rangle\leq s,\alpha_1>0}||D^{\alpha}_{\ast}{\mathbf V}(t)||^2_{L^2(\Omega)}\leq C(K)\mathcal{M}(t),
\end{equation}
where $s$ is a positive integer and $\mathcal{M}(t)$ is defined in \eqref{defM}.
\end{lemma}
\begin{proof}

It is obvious that when $\alpha_1>0,$ $D^{\alpha}_{\ast}{\mathbf V}|_{x_1=0}=0.$ Note that
$$\sigma^{\alpha_1}\partial^{m+1}_1=\partial_1(\sigma^{\alpha_1}\partial^m_1)-\alpha_1\sigma'\sigma^{\alpha_1-1}\partial^m_1,$$ for some nonnegative integer $m.$
Applying $D^{\alpha}_{\ast}$ to \eqref{A} and using the standard energy method,
we obtain that
$$\int_{\R^2_+}(\mathcal{A}_0D^{\alpha}_{\ast}{\mathbf V}\cdot D^{\alpha}_{\ast}{\mathbf V})(t)d\mathbf{x}=\int_{\Omega_t}\Big(((\partial_t\mathcal{A}_0+\partial_1\mathcal{A}_1+\partial_2\mathcal{A}_2)D^{\alpha}_{\ast}{\mathbf V}+2\mathcal{R})\cdot D^{\alpha}_{\ast}{\mathbf V}\Big)d\mathbf{x}d\tau,$$
where
$$\quad \mathcal{R}=D^{\alpha}_{\ast}\mathcal{F}+\mathcal{R}_0+\mathcal{R}_1,\quad\mathcal{R}_0=\alpha_1\sigma'\mathcal{A}_1\sigma^{\alpha_1-1}D^{\alpha'}_{t,x}\partial^{\alpha_3}_1{\mathbf V},$$
$$\mathcal{R}_1=-[D^{\alpha}_{\ast},\mathcal{A}_0]\partial_t{\mathbf V}-\sum^2_{j=1}[D^{\alpha}_{\ast},\mathcal{A}_j]\partial_j{\mathbf V}-D^{\alpha}_{\ast}(\mathcal{A}_3{\mathbf V})$$
(the brackets $[\cdot,\cdot]$ here denotes the commutator between the operators).
 Notice that $\mathcal{A}_0$ is positive definite, then
$$\int_{\R^2_+}(\mathcal{A}_0D^{\alpha}_{\ast}{\mathbf V}\cdot D^{\alpha}_{\ast}{\mathbf V})(t)d\mathbf{x}\geq c_0||D^{\alpha}_{\ast}\mathbf V(t)||^2_{L^2(\mathbb R^2_+)},$$
where $c_0$ depends on the number $k$ in \eqref{hy} and \eqref{stability1}. Hence,  we obtain that
\begin{equation}\label{DV33}
||D^{\alpha}_{\ast}{\mathbf V}(t)||^2_{L^{2}(\R^2_+)}\leq C(K)\Big(||{\mathbf V}||^2_{s,\ast,t}+||\mathcal{R}||^2_{L^2(\Omega_t)}\Big).
\end{equation}
Now, we estimate $||\mathcal{R}||^2_{L^2(\Omega_t)}$ in \eqref{DV33}. Recall that $\mathcal{R}=D^{\alpha}_{\ast}\mathcal{F}+\mathcal{R}_0+\mathcal{R}_1.$
Using Moser-type calculus inequalities \eqref{moser4} and \eqref{moser6}, we can prove that
\begin{equation}\label{FFFF}
\begin{split}
||D^{\alpha}_{\ast}\mathcal{F}||^2_{L^2(\Omega_t)}&\lesssim||J^T\mathbf{F}||^2_{s,\ast,t}\\
&\leq C(K)\Big(||\mathbf{F}||^2_{s,\ast,t}+||\mathbf{F}||^2_{W^{1,\infty}_{\ast}(\Omega_t)}||\hat{W}||^2_{s,\ast,t}\Big).
\end{split}
\end{equation}
Using the decomposition of boundary matrix $\mathcal{A}_1$ in \eqref{AA}, $\mathcal{A}_1=\mathcal{A}+\mathcal{A}_{(0)}$, following arguments similar to those used in Lemma \ref{noncha}, we obtain that
\begin{equation}\label{R0}
\begin{split}
||\mathcal{R}_0||^2_{L^2(\Omega_t)}&\lesssim ||\mathcal A\sigma^{\alpha_1-1}D^{\alpha'}_{t,x}\partial^{\alpha_3}_1{\mathbf V}||^2_{L^2(\Omega_t)}+||\frac{\mathcal A_{(0)}}{\sigma}\sigma^{\alpha_1}D^{\alpha'}_{t,x}\partial^{\alpha_3}_1{\mathbf V}||^2_{L^2(\Omega_t)}\\
&\lesssim C(K)\left(||\sigma^{\alpha_1-1}D^{\alpha'}_{t,x}\partial^{\alpha_3}_1{\mathbf V}_n||^2_{L^2(\Omega_t)}+ ||\sigma^{\alpha_1}D^{\alpha'}_{t,x}\partial^{\alpha_3}_1{\mathbf V}||^2_{L^2(\Omega_t)}\right)\\
&\leq C(K)\Big(||\partial_1{\mathbf V}_n||^2_{s-1,\ast,t}+||{\mathbf V}||^2_{s,\ast,t}\Big)\\
&\leq C(K)\mathcal{M}(t),
\end{split}
\end{equation}
recall here above that the matrix $\mathcal A$ acts only on the noncharacteristic part ${\mathbf V}_n$ of the unknown $\mathbf V$.
\newline
Now, we estimate the commutators in  $\mathcal{R}_1:$
For $j=0,2,$
we obtain that
\begin{equation}\label{AJJ}
\begin{split}
||[D^{\alpha}_{\ast},\mathcal{A}_j]\partial_j{\mathbf V}||^2_{L^2(\Omega_t)}\leq & \sum_{0<\beta\leq\alpha}||D^{\beta}_{\ast}\mathcal{A}_jD^{\alpha-\beta}_{\ast}\partial_j{\mathbf V}||^2_{L^2(\Omega_t)}\,.
\end{split}
\end{equation}
For $\langle\beta\rangle=1$, we get
\begin{equation}
\sum_{\langle\beta\rangle=1\,,\,\beta\leq\alpha}||D^{\beta}_{\ast}\mathcal{A}_jD^{\alpha-\beta}_{\ast}\partial_j{\mathbf V}||^2_{L^2(\Omega_t)}\le C(K)||{\mathbf V}||^2_{s,\ast,t}\,.
\end{equation}
For $\langle\beta\rangle\geq2,$ we obtain that
\begin{equation}\label{AJJ222}
\begin{split}
\sum_{\langle\beta\rangle\geq 2\,,\,\beta\leq\alpha}||D^{\beta}_{\ast}\mathcal{A}_jD^{\alpha-\beta}_{\ast}\partial_j{\mathbf V}||^2_{L^2(\Omega_t)}&\lesssim\sum_{\langle\beta'\rangle=2,\beta'\leq\beta\leq\alpha}||D^{\beta-\beta'}_{\ast}(D^{\beta'}_{\ast}\mathcal{A}_j)D^{\alpha-\beta}_{\ast}(\partial_j{\mathbf V})||^2_{L^2(\Omega_t)}\\
&\leq C(K)\Big(||{\mathbf V}||^2_{s-1,\ast,t}+||\dot{{\mathbf U}}||^2_{W^{2,\infty}_{\ast}(\Omega_t)}||\hat{W}||^2_{s,\ast,t}\Big)\,.
\end{split}
\end{equation}

For $j=1,$ we need to be very careful. We have
\begin{equation}\label{j1}
\begin{split}
||[D^{\alpha}_{\ast},\mathcal{A}_1]\partial_1{\mathbf V}||^2_{L^2(\Omega_t)}&\lesssim\sum_{0<\beta\leq\alpha}||D^{\beta}_{\ast}\mathcal{A}_1D^{\alpha-\beta}_{\ast}(\partial_1{\mathbf V})||^2_{L^2(\Omega_t)}\,.
\end{split}
\end{equation}
For $\langle\beta\rangle=1$ we get $\beta_3=0.$ Therefore, from \eqref{AA} it follows that $D^{\beta}_{\ast}\mathcal{A}_1|_{x_1=0}=0$ and we use Moser-type calculus inequalities \eqref{moser4}, \eqref{moser6} to obtain
\begin{equation}\label{beta1}
\sum_{\langle\beta\rangle=1\,,\,\beta\leq\alpha}||D^{\beta}_{\ast}\mathcal{A}_1D^{\alpha-\beta}_{\ast}(\partial_1{\mathbf V})||^2_{L^2(\Omega_t)}\leq C(K)||{\mathbf V}||^2_{s,\ast,t}\,.
\end{equation}
For $\langle\beta\rangle\geq 2$, we have
\begin{equation}\label{beta maggiore 2}
\begin{split}
&\sum_{\langle\beta\rangle\ge 2\,,\,\beta\leq\alpha}||D^{\beta}_{\ast}\mathcal{A}_1D^{\alpha-\beta}_{\ast}(\partial_1{\mathbf V})||^2_{L^2(\Omega_t)}\\
&\lesssim\sum_{\langle\beta'\rangle=2,\beta'\leq\beta\leq\alpha}||D^{\beta-\beta'}_{\ast}(D^{\beta'}_{\ast}\mathcal{A}_1)D^{\alpha-\beta}_{\ast}(\partial_1{\mathbf V})||^2_{L^2(\Omega_t)}\\
&\leq C(K)\Big(||{\mathbf V}||^2_{s,\ast,t}+||\dot{{\mathbf U}}||^2_{W^{2,\infty}_{\ast}(\Omega_t)}||\hat{W}||^2_{s,\ast,t}\Big).
\end{split}
\end{equation}
Using Moser-type calculus inequalities \eqref{moser4}, \eqref{moser6}, we can prove that
\begin{equation}\label{DV2}
\begin{split}
||D^{\alpha}_{\ast}(\mathcal{A}_3{\mathbf V})||^2_{L^2(\Omega_t)}&\lesssim||\mathcal{A}_3{\mathbf V}||^2_{s,\ast,t}\\
&\leq C(K)\Big(||{\mathbf V}||^2_{s,\ast,t}+||\dot{{\mathbf U}}||^2_{W^{1,\infty}_{\ast}(\Omega_t)}||\hat{W}||^2_{s+2,\ast,t}\Big).
\end{split}
\end{equation}
Notice that summing up \eqref{AJJ}-\eqref{DV2} gives
\begin{equation}\label{R1}
||\mathcal{R}_1||^2_{L^2(\Omega_t)}\leq C(K)\mathcal{M}(t).
\end{equation}
Hence, using  \eqref{FFFF}, \eqref{R0} and \eqref{R1}, we obtain \eqref{DVV}. Lemma \ref{estimate1} is concluded.
\end{proof}

Now, we are going to estimate the non-weighted normal derivatives.

\subsection{Estimate of non-weighted normal derivatives}
Now we perform the differential operator $D^{\alpha}_{\ast}$, in the case $\alpha_1=0, \alpha_3\geq1$, that is $D^\alpha_\ast=\partial^{\alpha_0}_t\partial^{\alpha_2}_2\partial^{\alpha_3}_1$ with $\langle\alpha\rangle\le s$. Now, we are ready to prove the following Lemma \ref{estimate2}.

\begin{lemma}\label{estimate2}
The estimate
\begin{equation}\label{DV}
\sum_{\langle\alpha\rangle\le s, \alpha_1=0, \alpha_3\geq1}||D^{\alpha}_{\ast}{\mathbf V}(t)||^2_{L^2(\mathbb R^2_+)}\leq C(K)\mathcal{M}(t)
\end{equation}
holds for problem \eqref{ho2} for all $t\leq T,$ where $s$ is a positive integer, $\mathcal{M}(t)$ is given in \eqref{defM}.
\end{lemma}
\begin{proof}
Similar as in Lemma \ref{estimate1}, applying the operator $D^{\alpha}_{\ast}$ on \eqref{A} and using the standard energy method,
we obtain that
\begin{equation}\nonumber
\begin{split}
\int_{\R^2_+}(\mathcal{A}_0D^{\alpha}_{\ast}{\mathbf V}\cdot D^{\alpha}_{\ast}{\mathbf V})(t)d\mathbf{x}&=\int_{\Omega_t}\Big(((\partial_t\mathcal{A}_0+\partial_1\mathcal{A}_1+\partial_2\mathcal{A}_2)D^{\alpha}_{\ast}{\mathbf V}+2\mathcal{R})\cdot D^{\alpha}_{\ast}{\mathbf V}\Big)d\mathbf{x}d\tau\\
&\quad+\int_{\Gamma_t}(\mathcal{A}_1D^{\alpha}_{\ast}{\mathbf V}\cdot D^{\alpha}_{\ast}{\mathbf V})|_{x_1=0}dx_2d\tau\,,
\end{split}
\end{equation}
where $\mathcal R$ is defined as in the proof of Lemma \ref{estimate1}. Thus
$$||D^{\alpha}_{\ast}{\mathbf V}(t)||^2_{L^2(\mathbb R^2_+)}\leq C(K)\Big(||{\mathbf V}||^2_{s,\ast,t}+||\mathcal{R}||^2_{L^2(\Omega_t)}+\tilde{\mathcal{M}}(t)\Big),$$
where
$$\tilde{\mathcal{M}}(t)=\Big|\int_{\Gamma_t}(\mathcal{A}_1D^{\alpha}_{\ast}{\mathbf V}\cdot D^{\alpha}_{\ast}{\mathbf V})|_{x_1=0}dx_2d\tau\Big|=\Big|\int_{\Gamma_t}(D^{\alpha}_{\ast}\dot{q}D^{\alpha}_{\ast}\dot{u}_n)|_{x_1=0}dx_2d\tau\Big|.$$
When $\alpha_1=0,$ by definition, $\mathcal{R}_0=0.$
Hence, we obtain the estimate
\begin{equation}\nonumber
||\mathcal{R}||^2_{L^2(\Omega_t)}\leq C(K)\mathcal{M}(t).
\end{equation}
This yields that
\begin{equation}\label{DV22}
||D^{\alpha}_\ast{\mathbf V}(t)||^2_{L^2(\R^2_+)}\leq C(K)(\mathcal{M}(t)+\tilde{\mathcal{M}}(t)).
\end{equation}
Using \eqref{Vn} and \eqref{R}, we obtain that
$$\tilde{\mathcal{M}}(t)\lesssim||\partial^{\alpha_0}_t\partial^{\alpha_2}_2\partial^{\alpha_3-1}_1\mathcal{K}||^2_{L^2(\Gamma_t)}\,,$$
where $\mathcal K$ is defined in \eqref{R}. Moreover, recalling that $\tilde{\mathcal{A}}|_{x_1=0}=\mathrm{diag}\Big(\mathcal{E}_{12},-\mathcal{E}_{12}\Big)$ because  $\partial_1\hat{\Phi}^{\pm}|_{x_1=0}=\pm1,$  we notice that $||\tilde{\mathcal{A}}|_{x_1=0}||_{L^{\infty}(\Gamma_t)}=1.$ Using $\mathcal{A}_{(0)}|_{x_1=0}=0$
we obtain that
\begin{equation}\label{S(t)}
\tilde{\mathcal{M}}(t)\leq C\Big(\Sigma_1(t)+\Sigma_2(t)+\Sigma_3(t)+||\partial^{\alpha_0}_t\partial^{\alpha_2}_2\partial^{\alpha_3-1}_1(\mathcal{A}_3{\mathbf V})||^2_{L^2(\Gamma_t)}+||\partial^{\alpha_0}_t\partial^{\alpha_2}_2\partial^{\alpha_3-1}_1\mathcal{F}||^2_{L^2(\Gamma_t)}\Big),
\end{equation}
where
$$\Sigma_1(t)=\sum_{j=0,2}||\mathcal{A}_j\partial^{\alpha_0}_t\partial^{\alpha_2}_2\partial^{\alpha_3-1}_1\partial_j{\mathbf V}||^2_{L^2(\Gamma_t)},$$
$$\Sigma_2(t)=\sum_{j=0,2}\sum_{\langle\beta'\rangle\geq1,\beta'+\beta''=(\alpha_0,\alpha_2,\alpha_3-1)}||D^{\beta'}_{\ast}\mathcal{A}_jD^{\beta''}_{\ast}\partial_j{\mathbf V}||^2_{L^2(\Gamma_t)};$$ notice that in spite of the notation here $D^{\beta'}_\ast$ and $D^{\beta''}_\ast$ dot not involve any weighted derivative $\sigma\partial_1$ (see the beginning of this section); notice also that $\beta'+\beta''=(\alpha_0,\alpha_2,\alpha_3-1)$ implies that $\langle\beta'\rangle+\langle\beta''\rangle\le s-2$;
$$\Sigma_3(t)=\sum_{\alpha_3'+\alpha_3''\leq \alpha_3,\alpha_3',\alpha_3''\geq1}||\partial^{\alpha_0}_t\partial^{\alpha_2}_2(\partial^{\alpha_3'}_1\mathcal{A}_{(0)}\partial^{\alpha_3''}_1{\mathbf V})||^2_{L^2(\Gamma_t)}.$$
Now, we estimate $\Sigma_2(t)$ and $\Sigma_3(t)$, which can be controlled by using Lemma \ref{tra} (i) and Moser-type calculus inequalities \eqref{moser4} and \eqref{moser6}:
\begin{equation}
\begin{split}
\Sigma_2(t)&\lesssim\sum_{j=0,2}\sum_{\langle\beta'\rangle=1}\Big(||D^{\beta'}_{\ast}\mathcal{A}_j\partial_j{\mathbf V}||^2_{s-3,\ast,t}+||\partial_1(D^{\beta'}_{\ast}\mathcal{A}_j\partial_j{\mathbf V})||^2_{s-3,\ast,t}\Big)\\
&\leq C(K)\Big(||{\mathbf V}||^2_{s,\ast,t}+||\dot{{\mathbf U}}||^2_{W^{2,\infty}_{\ast}(\Omega_t)}||\hat{W}||^2_{s,\ast,t}\Big).
\end{split}
\end{equation}
In the above estimate, it is noted that $\langle\beta''\rangle\leq s-3$.
\begin{equation}\label{K8}
\begin{split}
\Sigma_3(t)&\lesssim||\partial_1\mathcal{A}_{(0)}\partial_1{\mathbf V}||^2_{s-4,\ast,t}+||\partial_1(\partial_1\mathcal{A}_{(0)}\partial_1{\mathbf V})||^2_{s-4,\ast,t}\\
&\leq C(K)\Big(||{\mathbf V}||^2_{s,\ast,t}+||\dot{{\mathbf U}}||^2_{W^{2,\infty}_{\ast}(\Omega_t)}||\hat{W}||^2_{s,\ast,t}\Big).
\end{split}
\end{equation}
For the term $\Sigma_1(t),$ passing to the volume integral then using Leibniz's rule, we get (for shortness, in the sequel we denote $D^\alpha:=\partial^{\alpha_0}_t\partial^{\alpha_2}_2$ whereas $D^\alpha_\ast:=\partial^{\alpha_0}_t\partial^{\alpha_2}_2\partial^{\alpha_3}_1$):
\begin{equation}\nonumber
\begin{split}
\Sigma_1(t)&=-\sum_{j=0,2}\int_{\Omega_t}\partial_1|\mathcal{A}_jD^{\alpha}\partial^{\alpha_3-1}_1\partial_j{\mathbf V}|^2d\mathbf{x}d\tau\\
&=-2\sum_{j=0,2}\int_{\Omega_t}\Big((\mathcal{A}^2_jD^{\alpha}\partial^{\alpha_3-1}_1\partial_j{\mathbf V}\cdot D^{\alpha}_{\ast}\partial_j{\mathbf V})\\
&\quad+(\mathcal{A}_jD^{\alpha}\partial^{\alpha_3-1}_1\partial_j{\mathbf V}\cdot\partial_1\mathcal{A}_jD^{\alpha}\partial^{\alpha_3-1}_1\partial_j{\mathbf V})\Big)d\mathbf{x}d\tau\,.
\end{split}
\end{equation}
Then integration by parts with respect to $\partial_j$ (for $j=0,2)$ gives
\begin{equation}\nonumber
\Sigma_1(t)=\tilde{\Sigma}_1(t)+J_0(t)\,,
\end{equation}
where
\begin{equation}\nonumber
\begin{split}
\tilde{\Sigma}_1(t)&=2\sum_{j=0,2}\int_{\Omega_t}\Big(\mathcal{A}^2_jD^{\alpha}\partial^{\alpha_3-1}_1\partial^2_j{\mathbf V}\cdot D^{\alpha}_{\ast}{\mathbf V}\Big)\\
&\quad+\Big((\mathcal{A}_j\partial_j\mathcal{A}_j+\partial_j\mathcal{A}_j\mathcal{A}_j)D^{\alpha}\partial^{\alpha_3-1}_1\partial_j{\mathbf V}\cdot D^{\alpha}_{\ast}{\mathbf V}\Big)\\
&\quad-\Big(\mathcal{A}_jD^{\alpha}\partial^{\alpha_3-1}_1\partial_j{\mathbf V}\cdot\partial_1\mathcal{A}_jD^{\alpha}\partial^{\alpha_3-1}_1\partial_j{\mathbf V}\Big)d\mathbf{x}d\tau,
\end{split}
\end{equation}
\begin{equation}\nonumber
J_0(t)=-2\int_{\mathbb R^2_+}\Big(\mathcal{A}^2_0D^{\alpha}\partial^{\alpha_3-1}_1\partial_t{\mathbf V}\cdot D^{\alpha}_{\ast}{\mathbf V}\Big)(t)d\mathbf{x}.
\end{equation}
Since $|\alpha|+2+2(\alpha_3-1)\leq s$ we obtain that
\begin{equation}\label{k61}
\tilde{\Sigma}_1(t)\leq C(K)||{\mathbf V}||^2_{s,\ast,t}.
\end{equation}
Using Young's inequality, we obtain that
\begin{equation}\label{k62}
J_0(t)\leq C(K)\Big(\varepsilon||D^{\alpha}_{\ast}{\mathbf V}(t)||^2_{L^2(\mathbb R^2_+)}+\frac{1}{\varepsilon}|||{\mathbf V}(t)||||^2_{s-1,\ast}\Big),
\end{equation}
for small $\varepsilon.$
Therefore, we conclude from \eqref{k61}, \eqref{k62} and elementary inequalities \eqref{ele} that
\begin{equation}\label{K6}
\Sigma_1(t)\leq C(K)\Big(\varepsilon||D^{\alpha}_{\ast}{\mathbf V}(t)||^2_{L^2(\mathbb R^2_+)}+\frac{1}{\varepsilon}||{\mathbf V}||^2_{s,\ast,t}\Big).
\end{equation}
The last two terms in \eqref{S(t)} can be estimated by using Lemma \ref{tra} (i), Moser-type calculus inequalities \eqref{moser4} and \eqref{moser6}:
\begin{equation}\label{K66}
||\partial^{\alpha_0}_t\partial^{\alpha_2}_2\partial^{\alpha_3-1}_1(\mathcal{A}_3{\mathbf V})||^2_{L^2(\Gamma_t)}+||\partial^{\alpha_0}_t\partial^{\alpha_2}_2\partial^{\alpha_3-1}_1\mathcal{F}||^2_{L^2(\Gamma_t)}\leq C(K)\mathcal{M}(t).
\end{equation}
Summarizing \eqref{DV22}-\eqref{K8} and using \eqref{K6} and \eqref{K66}, taking $\varepsilon$ sufficiently small, we get the estimate \eqref{DV}. Therefore, Lemma \ref{estimate2} is proved.
\end{proof}

\subsection{Estimate of non-weighted tangential derivatives}\label{nonwt}
Now, we are going to obtain estimates of non-weighted tangential derivatives, i.e. $\alpha_1=\alpha_3=0$, that is $D^\alpha_\ast=\partial^{\alpha_0}_t\partial^{\alpha_2}_2$ with $|\alpha|\le s$. This is the most important case because we shall use the boundary conditions. This gives the loss of two additional derivatives that imply that in final tame estimate we will have the \lq\lq$s+4, \ast, t$" loss of derivatives from the coefficients, see Theorem \ref{tame}. This loss is caused by the presence of zero order terms in $\varphi$ (see \eqref{Be1}).
\begin{lemma}\label{estimate3}
The following estimate holds for \eqref{ho2} for all $t\leq T:$
\begin{equation}\label{nonweighted}
\begin{split}
\sum_{|\alpha|\leq s\,,\,\alpha_1=\alpha_3=0}||D^{\alpha}_{\ast}&{\mathbf V}(t)||^2_{L^2(\mathbb R^2_+)}\leq \varepsilon C(K)\Big(|||{\mathbf V}(t)|||^2_{s,\ast}+|||\varphi(t)|||^2_{H^{s-1}(\R)}\Big) \\
&+\frac{C(K)}{\varepsilon}\Big(\mathcal{M}(t)+||\varphi||^2_{H^{s-1}(\Gamma_t)}+||\varphi||^2_{W^{2,\infty}_{\ast}(\Gamma_t)}||\hat{W}||^2_{s+4,\ast,t}\Big),
\end{split}
\end{equation}
where  $s$ is positive integer, $\varepsilon$ is a positive constant and $\mathcal{M}(t)$ is defined in \eqref{defM}.
\end{lemma}
\begin{proof} We only need to estimate the highest-order tangential derivatives with $|\alpha|=s, $ since the lower order terms can be controlled by definition of the anisotropic Sobolev norm, through the following estimate:
\begin{equation}\label{DV1}
\sum_{|\alpha|\leq s-1}||D^{\alpha}_{\ast}{\mathbf V}(t)||^2_{L^2(\mathbb R^2_+)}\leq C||{\mathbf V}||^2_{s,\ast,t}.
\end{equation}
Therefore,
applying same argument as in Lemmata \ref{estimate1} and \ref{estimate2}, we obtain that
\begin{equation}\label{nV}
||D^{\alpha}_{\ast}{\mathbf V}(t)||^2_{L^2(\mathbb R^2_+)}\leq C(K)(\mathcal{M}(t)+\mathcal{J}(t)),
\end{equation}
for $\alpha=(\alpha_0,\alpha_2),$ with $|\alpha|=s,$ where
$$\mathcal{J}(t)=\Big|\int_{\Gamma_t}(\mathcal{B}_1D^{\alpha}_{\ast}{\mathbf V}\cdot D^{\alpha}_{\ast}{\mathbf V})|_{x_1=0}dx_2d\tau\Big|.$$
Taking into account the boundary conditions \eqref{ho2} and \eqref{Hn2}, and also the important dissipative structure \eqref{boundary-q-form_lambda_1}, the explicit quadratic boundary term in the integral can be rewritten in a suitable form: let us denote
$$c^{\pm}=\sum_{|\alpha'|+|\alpha''|=s,|\alpha'|\geq1}D^{\alpha'}_{\ast}\hat{\lambda}^{\pm}D^{\alpha''}_{\ast}\dot{H}^{\pm}_n,\quad \dot{w}^{\pm}=\dot{u}^{\pm}_n-\hat{\lambda}^{\pm}\dot{H}^{\pm}_n\,,$$
where
$$
[\partial_1\hat{q}]=(\partial_1\hat{q}^++\partial_1\hat{q}^-)|_{x_1=0}\,,\,\,\dot{w}^{\pm}|_{x_1=0}=(\dot{u}^{\pm}_N-\hat{\lambda}^{\pm}\dot{H}^{\pm}_N)|_{x_1=0}
$$
and $\hat\lambda^\pm$ were defined in \eqref{lambda}, see also Lemma \ref{lemma-Paolo-Ale-Paola}.
\newline
The boundary quadratic form becomes:
\begin{equation}\label{RHS}
\begin{split}
(\mathcal{B}_1D^{\alpha}_{\ast}{\mathbf V}\cdot D^{\alpha}_{\ast}{\mathbf V})|_{x_1=0}&=2[(D^{\alpha}_{\ast}\dot{u}_N-\hat{\lambda} D^{\alpha}_{\ast}\dot{H}_N)D^{\alpha}_{\ast}\dot{q}]=2[cD^{\alpha}_{\ast}\dot{q}]+2[D^{\alpha}_{\ast}\dot{w}D^{\alpha}_{\ast}\dot{q}]\\
&=2[cD^{\alpha}_{\ast}\dot{q}]+2D^{\alpha}_{\ast}\dot{w}^-|_{x_1=0}[D^{\alpha}_{\ast}\dot{q}]+2D^{\alpha}_{\ast}\dot{q}^+|_{x_1=0}[D^{\alpha}_{\ast}\dot{w}]\\
&=2D^{\alpha}_{\ast}\dot{q}^+|_{x_1=0}D^{\alpha}_{\ast}(\partial_2\varphi[\hat{u}_2-\hat{\lambda}\hat{H}_2])+\text{l.o.t},
\end{split}
\end{equation}
where lower order term l.o.t can be expressed by $$2[cD^{\alpha}_{\ast}\dot{q}]-2D^{\alpha}_{\ast}\dot{w}^-|_{x_1=0}D^{\alpha}_{\ast}([\partial_1\hat{q}]\varphi)-2D^{\alpha}_{\ast}\dot{q}^+|_{x_1=0}D^{\alpha}_{\ast}([\partial_1\hat{u}_N-\hat{\lambda} \partial_1\hat{H}_N]\varphi).$$
Since the boundary conditions are dissipative, it is noted that first term on the right-hand side of \eqref{RHS} vanish by Lemma \ref{lemma-Paolo-Ale-Paola}.
The boundary terms can be estimated separately:
\begin{equation}\label{decom}
\mathcal{J}(t)\leq \sum_{\pm}\sum^4_{i=1}\mathcal{J}^{\pm}_i(t),
\end{equation}
where
$$\mathcal{J}^{\pm}_1(t)=\Big|\int_{\Gamma_t}(c^{\pm}D^{\alpha}_{\ast}\dot{q}^{\pm})|_{x_1=0}d{x_2}d\tau\Big|,$$
$$\mathcal{J}^{\pm}_2(t)=\Big|\int_{\Gamma_t}(D^{\alpha}_{\ast}\dot{w}^-D^{\alpha}_{\ast}(\varphi\partial_1\hat{q}^{\pm}))|_{x_1=0}d{x_2}d\tau\Big|,$$
$$\mathcal{J}^{\pm}_3(t)=\Big|\int_{\Gamma_t}(D^{\alpha}_{\ast}\dot{q}^+D^{\alpha}_{\ast}(\varphi \partial_1\hat{u}^{\pm}_n))|_{x_1=0}d{x_2}d\tau\Big|,$$
$$\mathcal{J}^{\pm}_4(t)=\Big|\int_{\Gamma_t}(D^{\alpha}_{\ast}\dot{q}^+D^{\alpha}_{\ast}(\varphi \hat{\lambda}^{\pm}\partial_1\hat{H}^{\pm}_n))|_{x_1=0}d{x_2}d\tau\Big|.$$

Recall that $|\alpha|=s\geq2.$ We denote $D^{\alpha}_{\ast}=\partial_lD^{\gamma}$, $D^\gamma:=\partial^{\gamma_0}_t\partial^{\gamma_2}_2$ for $\gamma=(\gamma_0,\gamma_2),|\gamma|=s-1\geq1$, where
 \begin{equation}\nonumber
 l=
 \begin{cases}
 2, & \text{ if } \alpha_0\neq s,\\
 0,   & \text{ if } \alpha_0=s.
 \end{cases}
 \end{equation}

In the following estimate of $\mathcal{J}^{+}_1(t)$, we separate the analysis into two cases.\\

\noindent \textbf{Case A:} When $\alpha_0=s$, then $l=0, D^{\alpha}_{\ast}=\partial_tD^{\gamma}.$ Using the normal derivative estimates \eqref{noncha3}, \eqref{noncha32} for non-characteristic variables, the elementary inequality \eqref{ele} and integration by parts, we obtain that
\begin{equation}\label{6.39}
\begin{split}
\mathcal{J}^{+}_1(t)&=\Big|\int_{\Omega_t}(\partial_tc^{+}\partial_1D^{\gamma}\dot{q}^{+}-\partial_1c^{+}D^{\alpha}_{\ast}\dot{q}^{+})d\mathbf{x}d\tau-\int_{\mathbb R^2_+}c^{+}\partial_1D^{\gamma}\dot{q}^{+}d\mathbf{x}\Big|\\
&\leq C\Big(||\dot{q}^{+}||^2_{s,\ast,t}+||\partial_1\dot{q}^{+}||^2_{s-1,\ast,t}+||\partial_tc^{+}||^2_{L^2(\Omega_t)}\\
&\quad+||\partial_1c^{+}||^2_{L^2(\Omega_t)}+\frac{1}{\varepsilon}||c^{+}(t)||^2_{L^2(\mathbb R^2_+)}+\varepsilon||\partial_1\dot{q}^{+}||^2_{H^{s-1}_{\ast}(\mathbb R^2_+)}
\Big)\\
&\leq C\Big(||\dot{q}^{+}||^2_{s,\ast,t}+||\partial_1\dot{q}^{+}||^2_{s-1,\ast,t}+||\partial_tc^{+}||^2_{L^2(\Omega_t)}\\
&\quad+||\partial_1c^{+}||^2_{L^2(\Omega_t)}+\frac{1}{\varepsilon}||c^{+}(t)||^2_{L^2(\mathbb R^2_+)}+\varepsilon|||\partial_1\dot{q}^{+}(t)|||^2_{s-1,\ast}
\Big)\\
&\leq \varepsilon C(K)|||{\mathbf V}(t)|||^2_{s,\ast}+\frac{C(K)}{\varepsilon}\mathcal{M}(t).
\end{split}
\end{equation}
Here, $\varepsilon$ is an arbitrary fixed constant.
Similar argument also holds for $\mathcal J^-_1(t).$
Therefore, we obtain that $$\mathcal{J}^{+}_1(t)+\mathcal{J}^{-}_1(t)\leq \varepsilon C(K)|||{\mathbf V}(t)|||^2_{s,\ast}+\frac{C(K)}{\varepsilon}\mathcal{M}(t).$$

\noindent \textbf{Case B:} When $\alpha_0\neq s$, then $l=2, D^{\alpha}_{\ast}=\partial_2D^{\gamma}.$ Using integration by parts, we could deduce that
\begin{equation}\label{6.42}
|\mathcal{J}^{\pm}_1(t)|\leq C(K)\mathcal{M}(t).
\end{equation}

Next, we estimate $\mathcal{J}^+_3$ and separate this term into two parts
\begin{equation}\nonumber
\mathcal{J}^+_3\leq \mathcal{J}_0(t)+\Sigma_4(t),
\end{equation}
where
$$\mathcal{J}_0(t)=\Big|\int_{\Gamma_t}(\varphi D^{\alpha}_{\ast}\dot{q}^+D^{\alpha}_{\ast}\partial_1\hat{u}^+_n)|_{x_1=0}dx_2d\tau\Big|,$$
$$\Sigma_4(t)=\sum_{|\alpha'|+|\alpha''|=s,|\alpha'|\geq1}\Big|\int_{\Gamma_t}(D^{\alpha}_{\ast}\dot{q}^+D^{\alpha'}_{\ast}\varphi D^{\alpha''}_{\ast}\partial_1\hat{u}^+_n)\Big|_{x_1=0}dx_2d\tau\Big|.$$

For the term $\mathcal{J}_0(t),$ after integrating by parts, we obtain that
\begin{equation}\nonumber
\begin{split}
\mathcal{J}_0(t)&=\Big|\int_{\Omega_t}(\partial_1\partial_lD^{\alpha}_{\ast}\hat{u}^+_n)\varphi\partial_1D^{\gamma}\dot{q}^+-(\partial^2_1D^{\alpha}_{\ast}\hat{u}^+_n)\varphi D^{\alpha}_{\ast}\dot{q}^+\\
&\quad+\partial_1D^{\alpha}_{\ast}\hat{u}^+_n\partial_l\varphi\partial_1D^{\gamma}\dot{q}^+d\mathbf{x}d\tau-\frac{l-2}{2}\int_{\mathbb R^2_+}(\partial_1D^{\alpha}_{\ast}\hat{u}^+_{n})\varphi\partial_1D^{\gamma}\dot{q}^+d\mathbf{x}\Big|.
\end{split}
\end{equation}
We note that to estimate $\partial^2_1D^{\alpha}_{\ast}\hat{u}^+_n$, we need $\hat{W}$ with regularity $s+4.$
For the term $\mathcal{J}_0(t),$ using Moser-type calculus inequalities \eqref{moser4}, \eqref{moser6}, \eqref{noncha3} and \eqref{noncha32} we get
$$\mathcal{J}_0(t)\leq \varepsilon C(K)|||{\mathbf V}(t)|||^2_{s,\ast}+\frac{C(K)}{\varepsilon}\Big(\mathcal{M}(t)+||\varphi||^2_{W^{1,\infty}{(\Gamma_t)}}||\hat{W}||^2_{s+4,\ast,t}\Big).$$
Now, we start to estimate $\Sigma_4.$ It is noted that for $|\alpha'|\geq1,$ we can isolate one tangential derivative in the differential operator:
$$D^{\alpha'}_{\ast}\varphi=D^{\gamma'}(\partial_j\varphi),\quad |\gamma'|\leq s-1,\quad j=0 \text{ or } j=2.$$
Using \eqref{dtphi} and \eqref{d2phi}, we obtain that
\begin{equation}\label{Hc}
\nabla_{t,x_2}\varphi=\mathbf{H}(\hat{{\mathbf U}},\hat{\varphi}){\mathbf V}_n|_{x_1=0}+\mathbf{G}(\hat{{\mathbf U}},\hat{\varphi})\varphi.
\end{equation}
Here $\mathbf{H}(\hat{{\mathbf U}},\hat{\varphi}),\mathbf{G}(\hat{{\mathbf U}},\hat{\varphi})$ depend on $\hat{{\mathbf U}}|_{x_1=0},\partial_1\hat{{\mathbf U}}|_{x_1=0}$ and second order derivatives of $\hat{\varphi}.$ \\

Using \eqref{Hc}, we can write the derivatives of $\varphi$ by using the non-characteristic unknown ${\mathbf V}_n$.  We insert these derivatives $D^{\alpha'}_{\ast}\varphi=D^{\gamma'}(\cdots)$, with $|\gamma'|\leq s-1$, into $\Sigma_4$ to obtain
$$\Sigma_4(t)=\sum_{|\alpha'|+|\alpha''|=s,|\alpha'|\geq1}\Big|\int_{\Gamma_t}(D^{\alpha}_{\ast}\dot{q}^+D^{\gamma'}(\mathbf{H}{\mathbf V}_n+\mathbf{G}\varphi)D^{\alpha''}_{\ast}\partial_1\hat{u}^+_n)\Big|_{x_1=0}dx_2d\tau\Big|\,,$$
that can be controlled by the right-hand side of \eqref{nonweighted}. Indeed, by passing to the volume integral on $\Omega_t$, we note that the highest order of regularity for $D^{\alpha''}_{\ast}\partial_1^2\hat{u}^+_n$ is $s+3$ because $|\gamma'|+|\alpha''|=s-1$ (since $|\alpha'|= |\gamma'|+1$ and $|\alpha'|+|\alpha''|=s$). Similar estimates also hold for  $\mathcal{J}^-_3,\mathcal{J}^{\pm}_2,\mathcal{J}^{\pm}_4.$ Using \eqref{DV1}, \eqref{nV}, \eqref{decom}, \eqref{6.39}, \eqref{6.42}, we obtain the estimate \eqref{nonweighted}. Therefore, Lemma \ref{estimate3} is concluded.
\end{proof}

\subsection{Estimate of front}

Now, we are going to estimate the front $\varphi(t)$ in $H^{s-1}(\R)$ and its tangential derivatives $\nabla_{t,x_2}\varphi$ in $H^{s-1}(\Gamma_t).$
\begin{lemma}\label{front}
Given the solution $\varphi$ of \eqref{ho2}, for all $t\leq T$ and  positive integer $s,$ the following estimate holds:
\begin{equation}\label{varphit}
\begin{split}
&|||\varphi(t)|||^2_{H^{s-1}(\R)}+||\nabla_{t,x_2}\varphi||^2_{H^{s-1}(\Gamma_t)}
\\
&\quad\leq C(K)\Big(||{\mathbf V}||^2_{s,\ast,t}+||\varphi||^2_{H^{s-1}(\Gamma_t)}+(||\dot{{\mathbf U}}||^2_{W^{1,\infty}_{\ast}(\Omega_t)}+||\varphi||^2_{W^{1,\infty}_{\ast}(\Gamma_t)})||\hat{W}||^2_{s+2,\ast,t}\Big).
\end{split}
\end{equation}
%\begin{equation}\label{phiesitmate}
%\begin{split}
%&||\nabla_{t,x_2}\varphi||^2_{H^{s-1}(\Gamma_t)}\\
%&\quad\leq C(K)\Big(||{\mathbf V}||^2_{s,\ast,t}+||\varphi||^2_{H^{s-1}(\Gamma_t)}+(||\dot{{\mathbf U}}||^2_{W^{1,\infty}_{\ast}(\Omega_t)}+||\varphi||^2_{W^{1,\infty}_{\ast}(\Gamma_t)})||\hat{W}||^2_{s+2,\ast,t}\Big).\\
%\end{split}
%\end{equation}
\end{lemma}
\begin{proof} Applying tangential derivatives on the first boundary conditions in \eqref{ho2}, we obtain that
\begin{equation}\label{vt}
\begin{split}
|||\varphi(t)|||^2_{H^{s-1}(\R)}&\leq C(K)\Big(||\dot{u}_n|_{x_1=0}||^2_{H^{s-1}(\Gamma_t)}+||\varphi||^2_{H^{s-1}(\Gamma_t)}+\sum_{|\alpha|\leq s-1}||\mathcal{G}_{\alpha}||^2_{L^2(\Gamma_t)}\Big),
\end{split}
\end{equation}
where $$\mathcal{G}_{\alpha}=D^{\alpha}_{\ast}(\varphi\partial_1\hat{u}^+_N)-([D^{\alpha}_{\ast},\hat{u}^+_2]\partial_2\varphi+\frac{1}{2}\partial_2\hat{u}^+_2D^{\alpha}_{\ast}\varphi).$$
For the first term on the right-hand side of \eqref{vt}, using trace Theorem in Lemma \ref{tra}, we obtain that
\begin{equation}\label{vt1}
||\dot{u}_n|_{x_1=0}||^2_{H^{s-1}(\Gamma_t)}\lesssim||{\mathbf V}||^2_{s,\ast,t}.
\end{equation}
For the third term on the right-hand side of \eqref{vt}, when $|\alpha|\leq s-1,$ we obtain that
\begin{equation}\label{vt2}
\begin{split}
||\mathcal{G}_{\alpha}||^2_{L^2(\Gamma_t)}&\lesssim ||\varphi\partial_1\hat{u}^+_N||^2_{H^{s-1}(\Gamma_t)}+||[D^{\alpha}_{\ast},\hat{u}^+_2]\partial_2\varphi||^2_{L^2(\Gamma_t)}+||\partial_2\hat{u}^+_2D^{\alpha}_{\ast}\varphi||^2_{L^2(\Gamma_t)}\\
&\leq C(K)\Big(||\varphi||^2_{H^{s-1}(\Gamma_t)}+||\varphi||^2_{W^{1,\infty}(\Gamma_t)}||\hat{W}||^2_{s+2,\ast,t}\Big).
\end{split}
\end{equation}
Hence, we obtain the estimate of the first term in the left-hand side of \eqref{varphit}.

For the estimate of the second term in the left-hand side of \eqref{varphit}, we use the relation \eqref{Hc}, $||\hat{\varphi}||_{H^s(\Gamma_t)}\lesssim||\hat{\Psi}||_{s,\ast,t}\leq||\hat{\varphi}||_{H^s(\Gamma_t)},$ and Moser-type calculus inequalities \eqref{moser4} and \eqref{moser6} to obtain that
\begin{equation}\nonumber
\begin{split}
&||\nabla_{t,x_2}\varphi||^2_{H^{s-1}(\Gamma_t)}\leq||\mathbf{H}(\hat{\mathbf U},\hat{\varphi}){\mathbf V}_n|_{x_1=0}||^2_{H^{s-1}(\Gamma_t)}+||\mathbf{G}(\hat{\mathbf U},\hat{\varphi})\varphi||^2_{H^{s-1}(\Gamma_t)}\\
&\leq C(K)\Big(||{\mathbf V}||^2_{s,\ast,t}+||\varphi||^2_{H^{s-1}(\Gamma_t)}+(||\dot{{\mathbf U}}||^2_{W^{1,\infty}_{\ast}(\Omega_t)}+||\varphi||^2_{W^{1,\infty}(\Gamma_t)})||\hat{W}||^2_{s+2,\ast,t}\Big).
\end{split}
\end{equation}
Therefore, Lemma \ref{front} is concluded.
\end{proof}
Collecting all the previously established higher order estimates, we can prove the following lemma.
\begin{lemma}\label{7.1}
The solution of the homogeneous problem \eqref{ho2} satisfies the following a priori estimate
\begin{equation}\label{tamesss2}
||\dot{{\mathbf U}}||^2_{s,\ast,T}+||\varphi||^2_{H^s(\Gamma_T)}\leq C(K)Te^{C(K)T}\mathcal{N}(T)
\end{equation}
for positive integer $s,$ where
\begin{equation}\nonumber
\begin{split}
\mathcal{N}(T)&=||\mathbf{F}||^2_{s,\ast,T}+(||\dot{{\mathbf U}}||^2_{W^{2,\infty}_{\ast}(\Omega_T)}+||\varphi||^2_{W^{2,\infty}(\Gamma_T)}+||\mathbf{F}||^2_{W^{1,\infty}_{\ast}(\Omega_T)})||\hat{W}||^2_{s+4,\ast,T}.
\end{split}
\end{equation}
\end{lemma}
\begin{proof}
Combining the estimates \eqref{DVV}, \eqref{DV}, \eqref{nonweighted} choosing $\varepsilon$ small enough and \eqref{varphit}, we obtain that
\begin{equation}\label{IT}
\mathcal{I}(t)\leq C(K)\Big(\mathcal{N}(T)+\int^t_0\mathcal{I}(\tau)d\tau\Big),
\end{equation}
where $$\mathcal{I}(t)=|||{\mathbf V}(t)|||^2_{s,\ast}+|||\varphi(t)|||^2_{H^{s-1}(\R)}.$$
Notice that $$\mathcal{I}(0)=0.$$
Then, using Gr\"onwall's lemma to \eqref{IT}, we obtain
$$\mathcal{I}(t)\leq C(K)e^{C(K)T}\mathcal{N}(T), \text{ for } 0\leq t\leq T.$$
Integrating \eqref{IT} with respect to $t\in(-\infty,T],$ we get
\begin{equation}\label{Vestimate}
||{\mathbf V}||^2_{s,\ast,T}+||\varphi||^2_{H^{s-1}(\Gamma_T)}\leq C(K)Te^{C(K)T}\mathcal{N}(T).
\end{equation}
Notice that $\dot{{\mathbf U}}=J{\mathbf V}.$ Then, using the decomposition $J=J(\hat{W})=I+J_0(\hat{W}),$ where $J_0$ satisfies  $J_0(0)=0,$
we apply the Moser-type calculus inequalities \eqref{moser4} and \eqref{moser6} and derive
\begin{equation}\label{Uestimate}
\begin{split}
||\dot{{\mathbf U}}||^2_{s,\ast,T}&=||{\mathbf V}+J_0{\mathbf V}||^2_{s,\ast,T}\\
&\leq C(K)\Big(||{\mathbf V}||^2_{s,\ast,T}+||\dot{{\mathbf U}}||^2_{W^{1,\infty}_ {\ast}(\Omega_T)}||\hat{W}||^2_{s,\ast,T}\Big)\\
&\leq C(K)||{\mathbf V}||^2_{s,\ast,T}+TC(K)||\dot{{\mathbf U}}||^2_{W^{1,\infty}_{\ast}(\Omega_T)}||\hat{W}||^2_{s+1,\ast,T},
\end{split}
\end{equation}
and, using \eqref{varphit} with $t=T$,
\begin{equation}\label{Uestimate2.0}
\begin{split}
&||\nabla_{t,x_2}\varphi||^2_{H^{s-1}(\Gamma_T)}\\
&\quad\lesssim\Big(||{\mathbf V}||^2_{s,\ast,T}+||\varphi||^2_{H^{s-1}(\Gamma_T)}+(||\dot{{\mathbf U}}||^2_{W^{1,\infty}_{\ast}(\Omega_T)}+||\varphi||^2_{W^{1,\infty}(\Gamma_T)})||\hat{W}||^2_{s+2,\ast,T}\Big)\\
&\quad\lesssim\Big(||{\mathbf V}||^2_{s,\ast,T}+||\varphi||^2_{H^{s-1}(\Gamma_T)}\\
&\qquad+TC(K)(||\dot{{\mathbf U}}||^2_{W^{1,\infty}_{\ast}(\Omega_T)}+||\varphi||^2_{W^{1,\infty}(\Gamma_T)})||\hat{W}||^2_{s+3,\ast,T}\Big)\,.
\end{split}
\end{equation}
In the proof of last inequalities in \eqref{Uestimate}, \eqref{Uestimate2.0}, we have used the following relation by applying Sobolev imbedding:
\begin{equation}\nonumber
\begin{split}
||\hat{W}||^2_{s,\ast,T}&=\sum^s_{j=0}\int^T_0||\partial^j_t\hat{W}(t)||^2_{s-j,\ast}dt
\leq\sum^s_{j=0}T\max_{t\in[0,T]}||\partial^j_t\hat{W}(t)||^2_{s-j,\ast}\\
&\lesssim T\sum^s_{j=0}\Big\{\int^T_0||\partial^j_t\hat{W}(t)||^2_{s-j,\ast}dt+\int^T_0||\partial^{j+1}_t\hat{W}(t)||^2_{s-j,\ast}dt\Big\}\\
&\lesssim T\int^T_0|||\partial^j_t\hat{W}(t)|||^2_{s+1-j,\ast}dt
\lesssim T||\hat{W}||^2_{s+1,\ast,T}.
\end{split}
\end{equation}
Using \eqref{Vestimate} and \eqref{Uestimate}, we obtain
\begin{equation}\label{Uestimate2}
||\dot{{\mathbf U}}||^2_{s,\ast,T}+||\varphi||^2_{H^{s-1}(\Gamma_T)}\leq C(K)Te^{C(K)T}\mathcal{N}(T).
\end{equation}
Similar to \eqref{Uestimate}, we can get
\begin{equation}\label{Vestimate2}
||{\mathbf V}||^2_{s,\ast,T}\leq C(K)||\dot{{\mathbf U}}||^2_{s,\ast,T}+TC(K)||\dot{{\mathbf U}}||^2_{W^{1,\infty}_{\ast}(\Omega_T)}||\hat{W}||^2_{s+1,\ast,T}.
\end{equation}
Adding \eqref{Uestimate2.0} and \eqref{Uestimate2}, and using \eqref{Vestimate2}, we conclude Lemma \ref{7.1}.
\end{proof}

Using \eqref{tamesss2}, we are ready to prove the tame estimate for the homogeneous problem \eqref{ho2}.

\medskip

\noindent
{\it Proof of Theorem \ref{tame}}: \;Using Lemma \ref{7.1}, the Sobolev inequalities \eqref{sobolev1}, \eqref{sobolev2} for $s\geq6,$ we obtain
\begin{equation}\label{7.9}
\begin{split}
||\dot{{\mathbf U}}||_{s,\ast,T}+||\varphi||_{H^s(\Gamma_T)}&\leq C(K)T^{\frac{1}{2}}e^{C(K)T}\Big(||\mathbf{F}||_{s,\ast,T}+(||\dot{{\mathbf U}}||_{6,\ast,T}\\
& \quad+||\varphi||_{H^6(\Gamma_T)}+||\mathbf{F}||_{4,\ast,T})||\hat{W}||_{s+4,\ast,T}\Big).
\end{split}
\end{equation}
Taking $T$ sufficiently small and $s=6,$ using \eqref{assumption}, we obtain that
\begin{equation}\label{7.10}
||\dot{{\mathbf U}}||_{6,\ast,T}+||\varphi||_{H^6(\Gamma_T)}\leq C(K_0)||\mathbf{F}||_{6,\ast,T}.
\end{equation}
Hence, \eqref{7.9} and \eqref{7.10} implies \eqref{tameess2}. The existence and uniqueness of the solution comes from Theorem \ref{th-wp-hom}. The proof of Theorem \ref{tame} is complete.

\section{Higher order energy estimate for problem \eqref{IVP2}}\label{tameestimate2}
%%%%%%%%%%%%%%%%%%%%%%%%%%%%%%%%%%%%%%%%%%%%%%%%%%%%%%%%%%%%%%%%%%%%%%
Now, we are ready to obtain an a \textit{priori} tame estimate in $H^s_{\ast}$ for the nonhomogeneous problem \eqref{IVP2}.
\begin{theorem}\label{keytame}
Let $T>0$ and $s$ be an integer, $s\geq6.$
Assume that the basic state $(\hat{{\mathbf U}},\hat{\varphi})$ satisfies \eqref{hy}-\eqref{c2}, and $(\breve{\mathbf{U}}^{\pm},\hat{\varphi})\in H^{s+4}_{\ast}(\Omega_T)\times H^{s+5}(\Gamma_T)$ satisfies \eqref{assumption}.
Assume that $\mathbf{f}\in H^{s+2}_{\ast}(\Omega_T)$, $\mathbf{g}\in H^{s+2}(\Gamma_T)$ vanish in the past. Then, there exists a positive constant $K_0,$ that does not depend on $s$ and $T,$ and there exists a constant $C(K_0)>0$ such that if $\hat{K}\leq K_0$, then there exists a unique solution $(\dot{{\mathbf U}},\varphi)\in H^{s}_{\ast}(\Omega_T)\times H^s({\Gamma_T})$ to the problem \eqref{IVP2} that allows the tame estimate
\begin{equation}\label{tamesss}
\begin{split}
&||\dot{{\mathbf U}}||_{s,\ast,T}+||\varphi||_{H^s(\Gamma_T)}\\
&\quad\leq C(K_0)\Big(||\mathbf{f}||_{s+2,\ast,T}+||\mathbf{g}||_{H^{s+2}(\Gamma_T)}+(||\mathbf{f}||_{8,\ast,T}+||\mathbf{g}||_{H^{8}(\Gamma_T)})||\hat{W}||_{s+4,\ast,T}\Big),
\end{split}
\end{equation}
for $T$ small enough, where $\hat{W}=(\breve{\mathbf{U}},\nabla_{t, \bf x}\hat{\Psi}).$
\end{theorem}
%%%%%%%%%%%%%%%%%%%%%%%%%%%%%%%%%%%%%%%%%%%%%%%%%%%%%%%%%%%%%%%%%%%%%%%%%%%
\begin{remark}
The lower regularity in \eqref{assumption} and low norms in \eqref{tameess2} and \eqref{tamesss}, for both the even and odd case, differ from Trakhinin \cite{Trakhinin2009}, see Theorem 3 and Theorem 4, due to finer Sobolev imbeddings \eqref{sobolev1}, \eqref{sobolev2}.
\end{remark}
\begin{proof}
Using the Moser-type calculus inequalities \eqref{moser4} and \eqref{moser6}, we obtain a refined version of estimate \eqref{2} in tame form:
$$||\mathbb{L}'_e(\hat{{\mathbf U}},\hat{\Psi})\tilde{{\mathbf U}}||_{s,\ast,T}\leq C(K)\Big(||\tilde{{\mathbf U}}||_{s+2,\ast,T}+||\tilde{{\mathbf U}}||_{W^{2,\infty}_{\ast}(\Omega_T)}||\hat{\mathbf U},\hat{\Psi}||_{s+2,\ast,T}\Big).$$
Then, using Sobolev embedding inequalities \eqref{sobolev2}, we get
$$||\mathbb{L}'_e(\hat{{\mathbf U}},\hat{\Psi})\tilde{{\mathbf U}}||_{s,\ast,T}\leq C(K)\Big(||\tilde{{\mathbf U}}||_{s+2,\ast,T}+||\tilde{{\mathbf U}}||_{6,\ast,T}||\hat{\mathbf U},\hat{\Psi}||_{s+2,\ast,T}\Big).$$
Using the above estimate, \eqref{estimate} and recalling the definition of $\tilde{\mathbf U}$, see Section \ref{hom_Sect}, it holds that
\begin{equation}\label{Festimate}
\begin{split}
||\mathbf{F}||_{s,\ast,T}\leq C(K)\Big(||\mathbf{f}||_{s+2,\ast,T}+||\mathbf{g}||_{H^{s+2}(\Gamma_T)}+(||\mathbf{f}||_{8,\ast,T}+||\mathbf{g}||_{H^8(\Gamma_T)})||\hat{\mathbf U},\hat{\Psi}||_{s+2,\ast,T}\Big).
\end{split}
\end{equation}
Using the assumption \eqref{assumption} and \eqref{Festimate} with $s=6$, we get
\begin{equation}\label{F6}
||\mathbf{F}||_{6,\ast,T}\leq C(K_0)\Big({||\mathbf{f}||_{8,\ast,T}+||\mathbf{g}||_{H^8(\Gamma_T)}}\Big).
\end{equation}
Combining the estimates \eqref{7.9}, \eqref{Festimate} and \eqref{F6}, we obtain the tame estimate \eqref{tamesss}.
\end{proof}

\section{Construction of Approximate Solutions}\label{compa}
Suppose the initial data
$$({\mathbf U}^{\pm}_0,\varphi_0)=(\bar{\mathbf U}^{\pm}+\tilde{\mathbf U}^{\pm}_0,\varphi_0)=(p^{\pm}_0, u^{\pm}_{1,0}, u^{\pm}_{2,0}, H^{\pm}_{1,0},H^{\pm}_{2,0}, S^{\pm}_0, \varphi_0)$$
satisfy the stability condition \eqref{stability1} and restriction \eqref{Hnn} at $x_1=0$ for $x_2\in \R$.
Since ${H}^+_{2,0}\neq0$ or ${H}^-_{2,0}\neq0$ at ${x_1=0}$, see also Remark \ref{H_2nonzero}, from \eqref{Hnn} we can solve $\partial_2\varphi$ as follows (we drop the sub-index 0 for simplicity):
\begin{equation}\label{8.1}
\partial_2\varphi=\mu({\mathbf U})|_{x_1=0}=\frac{H^+_1H^+_2+H^-_1H^-_2}{(H^{+}_2)^2+(H^-_2)^2}\Big|_{x_1=0}\,,
\end{equation}
where ${\mathbf U}:=({\mathbf U}^+,{\mathbf U}^-).$ Then, using the boundary condition \eqref{jump}, we have
\begin{equation}\label{ei}
\partial_t\varphi=\eta({\mathbf U})|_{x_1=0},
\end{equation}
with $$\eta({\mathbf U})=u^+_1-u^+_2\mu({\mathbf U}).$$
By using the hyperbolicity condition \eqref{hyper}, we can write the system in \eqref{IVP}:
\begin{equation}\label{hyperb}
\partial_t{\mathbf U}=-(A_0({\mathbf U}))^{-1}\Big(\tilde{A}_1({\mathbf U},\Psi)\partial_1{\mathbf U}+A_2({\mathbf U})\partial_2{\mathbf U}\Big),
\end{equation}
where $\Psi:=(\Psi^+,\Psi^-)$, and the matrices $A_0,A_2,\tilde{A}_1$ are defined by \eqref{Am} and \eqref{A11}.
The traces
$${\mathbf U}_j=(p^+_j,u^+_{1,j},u^+_{2,j},H^+_{1,j},H^+_{2,j},S^+_j,p^-_j,u^-_{1,j},u^-_{2,j},H^-_{1,j},H^-_{2,j},S^-_j)=\partial^j_t{\mathbf U}|_{t=0}$$ and $$\varphi_j=\partial^j_t\varphi|_{t=0},\quad j\geq1$$ can be defined step by step by applying operator $\partial^{j-1}_t$ to \eqref{ei} and \eqref{hyperb}, for $j\geq1$ and evaluating $\partial^j_t{\mathbf U}$ and $\partial^j_t\varphi$ at $t=0$ in terms of the initial data. Notice that
$$\Psi^{\pm}_j=\partial^j_t\Psi^{\pm}|_{t=0}=\chi(\pm x_1)\varphi_j.$$
Define the zero-th order compatibility condition:
\begin{equation}\label{com}
[u_{1,0}]-[u_{2,0}]\partial_2\varphi_0=0,\quad [p_0+\frac{|H_0|^2}{2}]=0.
\end{equation}
Taking \eqref{8.1}, \eqref{ei} evaluated at $t=0,$ and using \eqref{com}, we obtain that
\begin{equation}\label{ekonal}
\varphi_1=u^{\pm}_{1,0}-u^{\pm}_{2,0}\partial_2\varphi_0\Big|_{x_1=0}.
\end{equation}
Denote $(H^{\pm}_N)_j=\partial^j_tH^{\pm}_N\Big|_{t=0}.$ Using \eqref{hyperb} and \eqref{ekonal}, taking $t=0,$ we obtain that
$$(H^{\pm}_N)_1=-\Big(u^{\pm}_{2,0}\partial_2(H^{\pm}_N)_0+\partial_2u^{\pm}_{2,0}(H^{\pm}_N)_0\Big|_{x_1=0}\Big).$$
Therefore, $(H^{\pm}_N)_0|_{x_1=0}=0$ implies $(H^{\pm}_N)_1|_{x_1=0}=0.$
%Then,
%$$(H^{\pm}_N)_1=H^{\pm}_{1,1}-(H^{\pm}_{2,1}\partial_2\varphi_0+H^{\pm}_{2,0}\partial_2\varphi_1)\Big|_{x_1=0}.$$
Once we have defined ${\mathbf U}_1,\varphi_1,$ we can deduce ${\mathbf U}_2,\varphi_2$ and so on. Moreover, at each step, we can prove that
\begin{equation}\nonumber
(H^{\pm}_N)_j|_{x_1=0}=0,\quad j\geq2,
\end{equation}
provided that ${\mathbf U}_j$, $\varphi_j$ satisfy the compatibility condition (see Definition \ref{def}).
\newline
The following Lemma \ref{ccc} is necessary for the approximate solutions; we refer to \cite{Coulombel2008} and \cite[Lemma 4.2.1]{Metivier}. Differently from \cite{Trakhinin2009}, we take the initial data in the standard Sobolev spaces.
\begin{lemma}\label{ccc}
Let $\mu\in \mathbb{N},\mu\geq3,$ $\tilde{{\mathbf U}}_0:= \mathbf{U}_0-\bar{\mathbf{U}}\in H^{\mu+1.5}(\mathbb R^2_+)$ and $\varphi_0\in H^{\mu+1.5}(\R)$.
Then, we can determine $\tilde{{\mathbf U}}_j\in H^{\mu+1.5-j}(\mathbb R^2_+)$ and $\varphi_j\in H^{\mu+1.5-j}(\R)$ by induction and set $\mathbf{U}_j=\tilde{\mathbf{U}}_j +\bar{\mathbf{U}},$ for $j=1,\cdots,\mu$.
Besides we prove
\begin{equation}
\sum^{\mu}_{j=0}(||\tilde{{\mathbf U}}_j||_{H^{\mu+1.5-j}(\mathbb R^2_+)}+||\varphi_j||_{H^{\mu+1.5-j}(\R)})\leq C(M_0),
\end{equation}
where $C>0$ depends only on $\mu$, $||\tilde{{\mathbf U}}_0||_{W^{1,\infty}(\mathbb R^2_+)}$ and $||\varphi_0||_{W^{1,\infty}(\R)}$, and
\begin{equation}\label{M0}
M_0:=||\tilde{{\mathbf U}}_0||_{{H^{\mu+1.5}}(\mathbb R^2_+)}+||\varphi_0||_{H^{\mu+1.5}(\R)}\,.
\end{equation}
\end{lemma}
\begin{Definition}\label{def}
Let $\mu\in \mathbb{N},\mu\geq3$. The initial data $(\tilde{{\mathbf U}}_0,\varphi_0)\in H^{\mu+1.5}(\mathbb R^2_+)\times H^{\mu+1.5}(\R)$ are defined to be compatible up to order $\mu$ if $(\tilde{{\mathbf U}}_j,\varphi_j)$ satisfy \eqref{com} for $j=0$
and
\begin{equation}\nonumber
\sum^j_{l=0}([u_{1,j-l}]-[u_{2,j-l}]\partial_2\varphi_l)=0,\quad [p_j]+\sum^{j-1}_{l=0}C_{l,j-1}[(H_l,H_{j-l})]=0,\quad\mbox{on}\,\,\{x_1=0\}\,,
\end{equation}
for $j=1,\cdots,\mu$, where $C_{l,j-1}$ are suitable constants.
\end{Definition}
To use the tame estimate for the proof of convergence of the Nash--Moser iteration, we should reduce our nonlinear problem to that whose solution vanishes in the past. This is achieved by the construction of the so-called \lq\lq approximate solution" that allows to \lq\lq absorb" the initial data into the interior equation. The \lq\lq approximate solution" is in the sense of Taylor's series at $t=0$.
\newline
Below, we will use the notation
\begin{equation}\nonumber
\mathbb{L}({\mathbf U},\Psi):=
\left[\begin{array}{c}
   \mathbb{L}({\mathbf U}^+,\Psi^+)\\
   \mathbb{L}({\mathbf U}^-,\Psi^-)
  \end{array}\right].\quad
\end{equation}
\begin{lemma}\label{const}
Let $\mu\in \mathbb{N}, \mu\geq3$ and let $\delta>0$. Suppose the initial data  $(\tilde{{\mathbf U}}_0,\varphi_0)\in H^{\mu+1.5}(\mathbb R^2_+)\times H^{\mu+1.5}(\R)$ are compatible up to order $\mu$ and satisfy the assumptions \eqref{hyper}, \eqref{div}, \eqref{Hnn}, \eqref{stability1}. Then, there exist $T>0$ and $({\tilde{\mathbf U}}^a,\varphi^a)\in H^{\mu+2}(\Omega_T)\times H^{\mu+2}(\Gamma_T)$ such that
\begin{equation}\label{8.12}
\partial^j_t\mathbb{L}({\mathbf U}^a,\Psi^a)|_{t=0}=0 \text{ in } \Omega, \text{ for } j\in \{0,\cdots, \mu-1\},
\end{equation}
where
$$
\mathbf U^a:=\tilde{\mathbf U}^a+\overline{\mathbf U}\,,\quad \Psi^{a\,\pm}=\chi(\pm x_1)\varphi^a\,.
$$
We call $(\mathbf U^a, \varphi^a)$ the approximate solution to problem \eqref{IVP2}. Moreover the approximate solution satisfies the following estimate
\begin{equation}\label{app}
||\tilde{{\mathbf U}}^a||_{H^{\mu+2}(\Omega_T)}+||\varphi^a||_{H^{\mu+2}(\Gamma_T)}<\delta\,,
\end{equation}
the stability conditions \eqref{stability1} on $\Gamma_T$, the hyperbolicity condition \eqref{hy} on $\Omega_T$.
\end{lemma}
\begin{proof} Let us first denote $\Phi^{a\,\pm}=\pm x_1+\Psi^{a\,\pm}$,
$\tilde{\mathbf{U}}^a=(\tilde{\mathbf{U}}^{a+},\tilde{\mathbf{U}}^{a-})^T$, $\tilde p^a=(\tilde p^{a+}, \tilde{p}^{a-})^T$, $\tilde{u}^a=(\tilde{u}^{a+},\tilde{u}^{a-})^T$, $\tilde{H}^a=(\tilde{H}^{a+},\tilde{H}^{a-})^T$,
$\tilde{S}^a=(\tilde{S}^{a+},\tilde{S}^{a-})^T$.
Consider $\tilde{\mathbf{U}}^a\in H^{\mu+2}(\R\times \mathbb R^2_+),\varphi^a\in H^{\mu+2}(\R^2)$, such that
$$\partial^j_t\tilde{\mathbf U}^a|_{t=0}=\tilde{\mathbf U}_j\in H^{\mu-j+2}(\mathbb R^2_+), \text{ for } j=0,\cdots,\mu,$$
$$\partial^j_t\varphi^a|_{t=0}=\varphi_j\in H^{\mu-j+2}(\R),\quad\text{ for } j=0,\cdots,\mu,$$
where $\tilde{\mathbf{U}}_j$ and $\varphi_j$ are given by Lemma \ref{ccc}. Since $(\tilde{\mathbf{U}}^a,\varphi^a)$ satisfies the hyperbolicity condition \eqref{hy} and the stability condition \eqref{stability1} at $t=0$, by continuity $(\tilde{\mathbf{U}}^a,\varphi^a)$ satisfy \eqref{stability1} at $x_1=0$ and \eqref{hy} for small times.  By multiplication of $(\tilde{\mathbf{U}}^a,\varphi^a)$ by a cut-off function in time supported on $[-T,T]$ we can assume that \eqref{hy}, \eqref{stability1} hold for all times (in this regard, recall Remark \ref{dato_iniziale}). Given any $\delta>0$, by taking $T>0$ sufficiently small, we can assume that $\tilde{\mathbf{U}}^a$, $\varphi^a$ are small in the sense of \eqref{app}.
\end{proof}
\begin{remark}\label{cutoff argument}
Let us remark that we do not require any constraint (that is interior equations or boundary condition) to be satisfied by the approximate solution constructed above. This allows us to the use of cut-off argument making the hyperbolicity condition \eqref{hy} and the stability condition \eqref{stability1} to be satisfied globally in time, without any trouble.
\end{remark}
\begin{remark}\label{mu e m}
In the sequel, in the proof of the main Theorem \ref{maintheorem}, estimate \eqref{app} will be used with $\mu=m+10$, being $m$ an integer as in the statement of that theorem.
\end{remark}

We assume that
\begin{equation}\nonumber
||\varphi_0||_{L^{\infty}(\R)}<\frac{1}{2},
\end{equation}
then we fix $T>0$ sufficiently small so that $||\varphi^a||_{L^{\infty}([0,T]\times\R)}\leq\frac{1}{2}.$
Hence, we get
\begin{equation}\nonumber
\partial_1\Phi^{a+}\geq\frac{1}{2},\quad \partial_1\Phi^{a-}\leq-\frac{1}{2}
\end{equation}
(recall that $||\chi^\prime||_{L^\infty(\R)}\le 1/2$, see Section \ref{straighten}).
\newline
The approximate solution $({\mathbf U}^a,\varphi^a)$ enables us to reformulate the original problem \eqref{IVP} as a nonlinear problem with zero initial data. Set
\begin{equation}\label{f}
\mathcal{F}^a:=
\begin{cases}
-\mathbb{L}({\mathbf U}^a,\Psi^a),\;&t>0,\\
~~~~0,\;&t<0.\\
\end{cases}
\end{equation}
From $\tilde{{\mathbf U}}^a\in H^{\mu+2}(\Omega_T)$ and $\varphi^a\in H^{\mu+2}(\Gamma_T),$ we have $\mathcal{F}^a\in H^{\mu+1}(\Omega_T)$.
\newline
Given the approximate solution $(\tilde{{\mathbf U}}^a,\varphi^a)$ of Lemma \ref{const}  and $\mathcal{F}^a$   defined in \eqref{f}, we see that $({\mathbf U},\varphi)=({\mathbf U}^a,\varphi^a)+({\mathbf V},\psi)$ is a solution of the original problem \eqref{IVP} if ${\mathbf V}=({\mathbf V}^+,{\mathbf V}^-)^T,\Psi=(\Psi^+,\Psi^-)^T,$ $\Psi|_{x_1=0}:=\psi$ satisfy the following problem:
\begin{equation}\label{system}
 \begin{cases}
 \mathcal{L}({\mathbf V},\Psi):=\mathbb{L}({\mathbf U}^a+{\mathbf V},\Psi^a+\Psi)-\mathbb{L}({\mathbf U}^a,\Psi^a)=\mathcal{F}^a, & \text{in } \Omega_T,\\
 \mathcal{B}({\mathbf V}|_{x_1=0},\psi):=\mathbb{B}({\mathbf U}^a|_{x_1=0}+{\mathbf V}|_{x_1=0},\varphi^a+\psi)=0,& \text{on } \Gamma_T,\\
 ({\mathbf V},\psi)=0, & \text{for } t<0.
\end{cases}
\end{equation}
The original nonlinear problem on $[0,T]\times\mathbb R^2_+$ is thus reformulated as a  problem on $\Omega_T$ whose solutions vanish in the past.
\bigskip
\section{Nash-Moser Iteration}\label{nash}
In this section, we recall the Nash-Moser iteration for reader's convenience. First, we introduce the smoothing operators $S_{\theta}$ and describe the iterative scheme for problem \eqref{system}. For more details refer to \cite{ChenG2008,Coulombel2008,Trakhinin2009}.
\begin{lemma}\label{smooth}
Let $\mu\in \mathbb{N},$ with $\mu\geq 4$. $\mathcal{F}^s_{\ast}(\Omega_T):=\{u\in H^s_{\ast}(\Omega_T): u=0 \text{ for }t<0\}.$ Define a family of smoothing operators $\{S_{\theta}\}_{{\theta\geq1}}$ on the anisotropic Sobolev space from $\mathcal{F}^{3}_{\ast}(\Omega_T)$ to $\bigcap_{s\geq3}\mathcal{F}^{s}_{\ast}(\Omega_T)$, such that
\begin{equation}\label{as1}
||S_{\theta}u||_{k,\ast,T}\leq C\theta^{(k-j)_+}||u||_{j,\ast,T}, \text{ for all } k,j\in \{1,\cdots,\mu\},
\end{equation}
\begin{equation}\label{as2}
||S_{\theta}u-u||_{k,\ast,T}\leq C\theta^{k-j}||u||_{j,\ast,T}, \text{ for all } 1\leq k\leq j\leq \mu,
\end{equation}
\begin{equation}\label{as3}
||\frac{d}{d\theta}S_{\theta}u||_{k,\ast,T}\leq C\theta^{k-j-1}||u||_{j,\ast,T}, \text{ for all } k,j\in \{1,\cdots,\mu\},
\end{equation}
where $C$ is positive constant and $k,j\in \mathbb{N},(k-j)_+:=\max\{0,k-j\}.$  In particular, if $u=v$ on $\Gamma_T,$ then $S_{\theta}u=S_{\theta}v$ on $\Gamma_T$. The definition of $\mathcal F^s(\Gamma_T)$ is entirely similar.
%Furthermore, there exists another family of smoothing operators (still denoted by $S_{\theta}$) acting on the functions defined on the boundary $\omega_T$ and satisfying the properties \eqref{s1}-\eqref{s4} with norms $||\cdot||_{H^{\alpha}_{\lambda}(\omega_T)}.$
\end{lemma}
%The proof of \eqref{s4} is based on the following lifting operator, refer to \cite{Coulombel2008}.
%\begin{lemma}\label{lift}
%Let $T>0$ and $\lambda>1,$ and let $\mu\geq1$ be an integer. Then, there exists an operator $\mathcal{R}_T$, which is continuous from $\mathcal{F}^{s}_{\lambda}(\omega_T)$ to $\mathcal{F}^{s+\frac{1}{2}}_{\lambda}(\Omega_T)$ for all $s\in[1,\mu],$ such that, if $s\geq1$ and $u\in \mathcal{F}^s_{\lambda}(\omega_T)$, then $(\mathcal{R}_Tu)|_{x_2=0}=u.$
%\end{lemma}
%Here, we note that in the following section, we will take $m:=\mu+9.$
Now, we begin to formulate the Nash-Moser iteration scheme.

The iteration scheme starts from $({\mathbf V}_0,\Psi_0,\psi_0)=(0,0,0),$  and $({\mathbf V}_i,\Psi_i,\psi_i)$ is given such that
\begin{equation}\label{i1}
({\mathbf V}_i,\Psi_i,\psi_i)|_{t<0}=0,\quad \Psi^+_{i}|_{x_1=0}=\Psi^-_{i}|_{x_1=0}=\psi_i.
\end{equation}
Let us consider
\begin{equation}\label{i2}
{\mathbf V}_{i+1}={\mathbf V}_i+\delta {\mathbf V}_i, \; \Psi_{i+1}=\Psi_i+\delta\Psi_i,\; \psi_{i+1}=\psi_i+\delta\psi_i,
\end{equation}
where the  differences $(\delta {{\mathbf V}}_i,\delta\psi_i)$ will be determined below.
First, we can obtain $(\delta \dot{{\mathbf V}}_i,\delta\psi_i)$ by solving the effective linear problem:
\begin{equation}\label{effective3}
 \begin{cases}
\displaystyle \mathbb{L}_e'({\mathbf U}^a+{\mathbf V}_{i+\frac{1}{2}},\Psi^a+\Psi_{i+\frac{1}{2}})\delta \dot{{\mathbf V}}_i=f_i &\text{ in } \Omega_T,\\
\displaystyle \mathbb{B}_e'({\mathbf U}^a+{\mathbf V}_{i+\frac{1}{2}},\Psi^a+\Psi_{i+\frac{1}{2}})(\delta\dot{ {\mathbf V}}_i,\delta\psi_i)=g_i &\text{ on }\Gamma_T    ,\\
\displaystyle (\delta \dot{{\mathbf V}}_i,\delta\psi_i)=0 &\text{ for } t<0,
\end{cases}
\end{equation}
where operators $\mathbb{L}_e', \mathbb{B}_e'$ are defined in \eqref{Le1} and \eqref{Be1},
\begin{equation}\label{goodunkown}
\delta \dot{{\mathbf V}}_i:=\delta {\mathbf V}_i-\frac{\partial_1({\mathbf U}^a+{\mathbf V}_{i+\frac{1}{2}})}{\partial_1(\Phi^a+\Psi_{i+\frac{1}{2}})}\delta\Psi_i
\end{equation}
is the Alinhac ``good unknown" and $({\mathbf V}_{i+\frac{1}{2}},\Psi_{i+\frac{1}{2}})$ is a smooth modified state such that $({\mathbf U}^a+{\mathbf V}_{i+\frac{1}{2}},\Psi^a+\Psi_{i+\frac{1}{2}})$ satisfies \eqref{hy}--\eqref{c} and \eqref{c2}. The source terms $(f_i,g_i)$ will be defined through the accumulated errors at step $i$.
 $S_{\theta_i}$ is the smoothing operator with   ${\theta_i}$ defined by
\begin{equation}\label{theta}
\theta_0\geq1, \; \theta_i=\sqrt{\theta^2_0+i}\,.
\end{equation}
The errors at step $i$ can be defined from the following decompositions:
\begin{equation}\label{l}
\begin{split}
&\mathcal{L}({\mathbf V}_{i+1},\Psi_{i+1})-\mathcal{L}({\mathbf V}_i,\Psi_i)\\
&=\mathbb{L}'({\mathbf U}^a+{\mathbf V}_i,\Psi^a+\Psi_i)(\delta {\mathbf V}_i,\delta\Psi_i)+e'_i\\
&=\mathbb{L}'({\mathbf U}^a+S_{\theta_i}{\mathbf V}_{i},\Psi^a+S_{\theta_i}\Psi_i)(\delta {\mathbf V}_i,\delta\Psi_i)+e'_i+e''_i\\
&=\mathbb{L}'({\mathbf U}^a+{\mathbf V}_{i+\frac{1}{2}},\Psi^a+\Psi_{i+\frac{1}{2}})(\delta {\mathbf V}_i,\delta\Psi_i)+e'_i+e''_i+e'''_i\\
&=\mathbb{L}'_e({\mathbf U}^a+{\mathbf V}_{i+\frac{1}{2}},\Psi^a+\Psi_{i+\frac{1}{2}})\delta \dot{{\mathbf V}}_i+e'_i+e''_i+e'''_i+D_{i+\frac{1}{2}}\delta\Psi_i\\
\end{split}
\end{equation}
and
\begin{equation}\label{B}
\begin{split}
&\mathcal{B}({\mathbf V}_{i+1}|_{x_1=0},\psi_{i+1})-\mathcal{B}({\mathbf V}_i|_{x_1=0},\psi_i)\\
&=\mathbb{B}'(({\mathbf U}^a+{\mathbf V}_i)|_{x_1=0},\varphi^a+\psi_i)(\delta {\mathbf V}_i|_{x_1=0},\delta\psi_i)+\tilde{e}'_i\\
&=\mathbb{B}'(({\mathbf U}^a+S_{\theta_i}{\mathbf V}_{i})|_{x_1=0},\varphi^a+S_{\theta_i}\Psi_i|_{x_1=0})(\delta {\mathbf V}_i|_{x_1=0},\delta\psi_i)+\tilde{e}'_i+\tilde{e}''_i\\
&=\mathbb{B}'_e(({\mathbf U}^a+{\mathbf V}_{i+\frac{1}{2}})|_{x_1=0},\varphi^a+\psi_{i+\frac{1}{2}})(\delta \dot{{\mathbf V}}_i|_{x_1=0},\delta\psi_i)+\tilde{e}'_i+\tilde{e}''_i+\tilde{e}'''_i,\\
\end{split}
\end{equation}
where we write
\begin{equation}\label{D}
D_{i+\frac{1}{2}}:=\frac{1}{\partial_1(\Phi^a+\Psi_{i+\frac{1}{2}})}\partial_1\mathbb{L}({\mathbf U}^a+{\mathbf V}_{i+\frac{1}{2}},\Psi^a+\Psi_{i+\frac{1}{2}}),
\end{equation}
and have used \eqref{linearized22} to get the last identity in \eqref{l}.
Denote
\begin{equation}\label{e}
e_i:=e'_i+e''_i+e'''_i+D_{i+\frac{1}{2}}\delta\Psi_i, \quad \tilde{e}_i:=\tilde{e}'_i+\tilde{e}''_i+\tilde{e}'''_i.
\end{equation}
We assume $f_0:=S_{\theta_0}\mathcal F^a,(E_0,\tilde{E}_0,g_0):=(0,0,0)$ and $(f_k,g_k,e_k,\tilde{e}_k)$ are already given and vanish in the past for $k\in\{0,\cdots,i-1\}.$
We can calculate the accumulated errors at step $i,i\geq1$, by
\begin{equation}\label{ae}
E_i:=\sum^{i-1}_{k=0}e_k,\quad \tilde{E}_i:=\sum^{i-1}_{k=0}\tilde{e}_k.
\end{equation}
Then, we obtain $f_i$ and $g_i$ for $i\geq1$ from the equations:
\begin{equation}\label{fg}
\sum^{i}_{k=0}f_k+S_{\theta_i}E_i=S_{\theta_i}\mathcal{F}^a,\quad \sum^{i}_{k=0}g_k+S_{\theta_i}\tilde{E}_i=0.
\end{equation}
Then, given suitable $({\mathbf V}_{i+\frac12},\Psi_{i+\frac12})$, we can obtain $(\delta \dot{{\mathbf V}}_i,\delta\psi_i)$ as the solutions of the linear problem \eqref{effective3}, $\delta{\mathbf V}_i$ from \eqref{goodunkown}, $({\mathbf V}_{i+1},\Psi_{i+1}, \psi_{i+1})$ from \eqref{i2}. Since $S_{\theta_i}\rightarrow I $ as $i\rightarrow\infty,$ we can formally obtain the solution to problem \eqref{system} from $\mathcal{L}({\mathbf V}_{i},\Psi_{i})\rightarrow \mathcal{F}^a,\mathcal{B}({\mathbf V}_{i}|_{x_1=0},\psi_{i})\rightarrow0,$ as error terms $(e_i,\tilde{e}_i)\rightarrow0.$

\section{Proof of the Main Result}\label{proof}
 Now, we prove the local existence of solutions to \eqref{system} by a modified iteration scheme of Nash-Moser type.
From the sequence $\{\theta_i\}$ defined in \eqref{theta},  we set $\Delta_i:=\theta_{i+1}-\theta_i.$ Then, the sequence $\{\Delta_i\}$ is decreasing and tends to $0$ as $i$ goes to infinity. Moreover, we have
$$\frac{1}{3\theta_i}\leq\Delta_i=\sqrt{\theta^2_i+1}-\theta_i\leq\frac{1}{2\theta_i}, \; \forall \, i\in \mathbb{N}.$$

\subsection{Inductive analysis}\label{inductive}
Given a small fixed $\delta>0$, and an integer $\tilde\alpha$ that will be chosen later on, we assume that the following estimate holds:
\begin{equation}\label{small}
||\tilde{{\mathbf U}}^a||_{\tilde\alpha+6,\ast,T}+||\varphi^a||_{H^{\tilde\alpha+6}(\Gamma_T)}+||\mathcal{F}^a||_{\tilde\alpha+4,\ast,T}\leq\delta.
\end{equation}
We may assume that \eqref{small} holds, by taking $T>0$ sufficiently small.

\smallskip
\noindent
Given the integer $\alpha$, our inductive assumptions read
\begin{eqnarray}\label{Hn-1}
(H_{i-1}) \left\{ \begin{split}
&\displaystyle (a)\quad||(\delta {\mathbf V}_k,\delta\Psi_k)||_{s,\ast,T}+||\delta\psi_k||_{H^{s}(\Gamma_T)}\leq \delta\theta^{s-\alpha-1}_{k}\Delta_k, \\
&\qquad \forall k\in\{0,\cdots, i-1\},\quad \forall s\in \{6,\cdots,\tilde{\alpha}\}.\\
&\displaystyle (b)\quad||\mathcal{L}({\mathbf V}_k,\Psi_k)-\mathcal{F}^a||_{s,\ast,T}\leq2\delta\theta^{s-\alpha-1}_k,\\
&\qquad \forall k\in\{0,\cdots,i-1\},\quad \forall s\in \{6,\cdots,\tilde{\alpha}-2\}.\\
&\displaystyle (c)\quad||\mathcal{B}({\mathbf V}_k|_{x_1=0},\psi_k)||_{H^{s}(\Gamma_T)}\leq\delta\theta^{s-\alpha-1}_k,\\
&\qquad \forall k\in\{0,\cdots,i-1\},\quad \forall s\in \{6,\cdots,\tilde{\alpha}-2\}.\\
\end{split}
\right.
\end{eqnarray}
Our goal is to show that   $(H_0)$ holds   and  $(H_{i-1})$ implies  $(H_i)$,  for a suitable choice of the parameters $\alpha$, $\tilde\alpha$, for $\delta>0$ and $T>0$ sufficiently small, for $\theta_0\geq1$ sufficiently large. Then, we conclude that $(H_i)$ holds for all $i\in \mathbb{N}.$
\begin{lemma}\label{H0}
If $T>0$ is sufficiently small, then  $(H_0)$ holds.
\end{lemma}
\begin{proof}
The proof follows as in \cite[Lemma 17]{Trakhinin2009}.
\end{proof}

 Now we prove that $(H_{i-1})$ implies  $(H_i)$. The hypothesis $(H_{i-1})$ yields the following lemma.
\begin{lemma}\cite[Lemma 7]{Trakhinin2009}, \cite[Lemma 7]{Coulombel2008}\label{estimate5}
If $\theta_0$ is large enough, then, for each $k\in\{0,\cdots,i\}$,  and each integer $s\in\{6,\cdots,\tilde{\alpha}\},$
\begin{equation}\label{estimate6}
||({\mathbf V}_k,\Psi_k)||_{s,\ast,T}+||\psi_k||_{H^{s}(\Gamma_T)}\lesssim
 \begin{cases}
\delta \theta^{(s-\alpha)_+}_k, &\text{if }s\neq\alpha,\\
\delta \log\theta_k, &\text{if } s=\alpha,\\
\end{cases}
\end{equation}
\begin{equation}\label{estimate7}
||(I-S_{\theta_k})({\mathbf V}_k,\Psi_k)||_{s,\ast,T}+||(I-S_{\theta_k})\psi_k||_{H^s(\Gamma_T)}\lesssim \delta\theta^{s-\alpha}_k.
\end{equation}
Furthermore, for each $k\in\{0,\cdots ,i\}$,  and each integer $s\in\{6,\cdots,\tilde{\alpha}+8\},$
\begin{equation}\label{estimate8}
||(S_{\theta_k}{\mathbf V}_k,S_{\theta_k}\Psi_k)||_{s,\ast,T}+||S_{\theta_k}\psi_k||_{H^s(\Gamma_T)}\lesssim
\begin{cases}
\delta \theta^{(s-\alpha)_+}_k, &\text {if } s\neq\alpha,\\
\delta \log\theta_k, &\text {if } s=\alpha.\\
\end{cases}
\end{equation}
\end{lemma}
\subsection{Estimate of the error terms}
To derive $(H_{i})$ from  $(H_{i-1})$, we need to estimate the quadratic error terms $e'_k$ and  $\tilde{e}'_k,$ the first substitution error terms $e''_k$ and $\tilde{e}''_k,$ the second substitution error terms $e'''_k$ and $\tilde{e}'''_k$ and the last error term $D_{k+\frac{1}{2}}\delta\Psi_k$ (\textit{cfr.} \eqref{l}-\eqref{D}).

First, we denote the quadratic error terms by
\begin{equation}\label{e1}
e'_k:=\mathcal{L}({\mathbf V}_{k+1},\Psi_{k+1})-\mathcal{L}({\mathbf V}_{k},\Psi_{k})-\mathcal{L}'({\mathbf V}_{k},\Psi_{k})(\delta {\mathbf V}_k,\delta\Psi_k),
\end{equation}
\begin{equation}\label{e3}
\tilde{e}'_k:=\mathcal{B}({\mathbf V}_{k+1}|_{x_1=0},\psi_{k+1})-\mathcal{B}({\mathbf V}_{k}|_{x_1=0},\psi_{k})-\mathcal{B}'({\mathbf V}_{k}|_{x_1=0},\psi_{k})(\delta {\mathbf V}_k|_{x_1=0},\delta\psi_k).
\end{equation}

Then, we get
\begin{equation}\nonumber
e'_k=\int^1_0\mathbb{L}''({\mathbf U}^a+{\mathbf V}_k+\tau\delta {\mathbf V}_k,\Psi^a+\Psi_k+\tau\delta \Psi_k)((\delta {\mathbf V}_k,\delta\Psi_k),(\delta {\mathbf V}_k,\delta\Psi_k))(1-\tau)d\tau,
\end{equation}
\begin{equation}\nonumber
\tilde{e}'_k=\frac{1}{2}\mathbb{B}''((\delta {\mathbf V}_k,\delta\psi_k),(\delta {\mathbf V}_k,\delta\psi_k)),
\end{equation}
where $\mathbb{L}'',\mathbb{B}''$ denote the second order derivatives of the operators $\mathbb{L}$ and $\mathbb{B}.$
To be more precise, we define
\begin{equation}\nonumber
\mathbb{L}''(\hat{\mathbf U},\hat{\Psi})(({\mathbf V},\Psi),(\tilde{{\mathbf V}},\tilde{\Psi})):=\frac{d}{d\varepsilon}\mathbb{L}'(\hat{\mathbf U}+\varepsilon \tilde{{\mathbf V}},\hat{\Psi}+\varepsilon\tilde{\Psi})({\mathbf V},\Psi)\Big|_{\varepsilon=0},
\end{equation}
\begin{equation}\nonumber
\mathbb{B}''(({\mathbf V},\psi),(\tilde{{\mathbf V}},\tilde{\psi})):=\frac{d}{d\varepsilon}\mathbb{B}'(\hat{\mathbf U}+\varepsilon \tilde{{\mathbf V}},\hat{\varphi}+\varepsilon\tilde{\psi})({\mathbf V},\psi)\Big|_{\varepsilon=0},
\end{equation}
where $\mathbb{L}'$ and $\mathbb{B}'$ are defined in \eqref{Lop} and \eqref{Bop}. Simple calculations yield that
\begin{equation}
\mathbb{B}''(({\mathbf V},\psi),(\tilde{{\mathbf V}},\tilde{\psi})):=\left[
\begin{array}{c}
\tilde{u}^+_2\partial_2\psi+\partial_2\tilde{\psi}u^+_2\\
\tilde{u}^-_2\partial_2\psi+\partial_2\tilde{\psi}u^-_2\\
\mathbf{H}^+\cdot\tilde{\mathbf{H}}^+-\mathbf{H}^-\cdot\tilde{\mathbf{H}}^-
\end{array}\right]\,.
\end{equation}
To estimate the error terms, we need to estimate the operators $\mathbb{L}''$ and $\mathbb{B}''.$ Applying the Moser-type calculus inequalities in Lemma \ref{moser1} and Lemma \ref{moser2} and the explicit forms of $\mathbb{L}''$ and $\mathbb{B}''$ , we can obtain the necessary estimates. Omitting the detailed calculation, we have the following Lemma \ref{lb}:
\begin{lemma}\label{lb}
Let $T>0,$ and $s\in \mathbb{N}$ with $s\geq6.$ Assume that $(\breve{\mathbf{U}},\hat{\Psi})\in H^{s+2}_{\ast}(\Omega_T)$ satisfies
$$||(\hat{{\mathbf U}},\hat{\Psi})||_{W^{2,\infty}_{\ast}(\Omega_T)}\leq\tilde{K},$$
(recall that $\hat{{\mathbf U}}=\bar{{\mathbf U}}+\breve{{\mathbf U}}$) for some constant $\tilde{K}>0.$ Then, there exists a positive $C$ depending on $\tilde{K},$ but not on $T,$ such that if $({\mathbf V}_i,\Psi_i)\in H^{s+2}_{\ast}(\Omega_T)$ and $(W_i,\psi_i)\in H^{s}(\Gamma_T)\times H^{s+1}(\Gamma_T),$ for $i=1,2,$
then
\begin{equation}\nonumber
\begin{split}
&||\mathbb{L}''(\hat{\mathbf{U}},\hat{\Psi})(({\mathbf V}_1,\Psi_1),({\mathbf V}_2,\Psi_2))||_{s,\ast,T}\\
&\quad\leq C||(\breve{{\mathbf U}},\hat{\Psi})||_{s+2,\ast,T}||({\mathbf V}_1,\Psi_1)||_{W^{2,\infty}_{\ast}(\Omega_T)}||({\mathbf V}_2,\Psi_2)||_{W^{2,\infty}_{\ast}(\Omega_T)}\\
&\quad\quad +C\sum_{i\neq j}||({\mathbf V}_i,\Psi_i)||_{s+2,\ast,T}||({\mathbf V}_j,\Psi_j)||_{W^{2,\infty}_{\ast}(\Omega_T)},
\end{split}
\end{equation}
and
\begin{equation}\nonumber
\begin{split}
&||\mathbb{B}''((W_1,\psi_1),(W_2,\psi_2))||_{H^s(\Gamma_T)}\\
&\quad \leq C\sum_{i\neq j}\Big(||W_i||_{H^s(\Gamma_T)}||\psi_j||_{W^{1,\infty}(\Gamma_T)} +||W_i||_{L^{\infty}(\Gamma_T)}||\psi_j||_{H^{s+1}(\Gamma_T)}\\
&\quad\quad+ ||W_i||_{H^s(\Gamma_T)}||W_j||_{L^{\infty}(\Gamma_T)}\Big).
\end{split}
\end{equation}
\end{lemma}

\subsubsection{Estimate of the quadratic errors}\label{quadratic}

We now apply Lemma \ref{lb} to prove the following estimate for the quadratic error terms.

\begin{lemma}\label{eestimate}
Let $\alpha\geq7.$ There exist $\delta>0$ sufficiently small and $\theta_0\geq1$ sufficiently large such that, for all $k\in\{0,\cdots,i-1\},$ and all integers $s\in\{6,\cdots,\tilde{\alpha}-2\},$ we have
\begin{equation}\label{ee1}
||e'_k||_{s,\ast,T}\lesssim\delta^2\theta^{L_1(s)-1}_k\Delta_k,
\end{equation}
\begin{equation}\label{ee3}
||\tilde{e}'_k||_{H^{s}(\Gamma_T)}\lesssim\delta^2\theta^{L_1(s)-1}_k\Delta_k,
\end{equation}
where $L_1(s):=\max\{(s+2-\alpha)_++10-2\alpha;s+6-2\alpha\}.$
\end{lemma}
\begin{proof}
Using \eqref{small}, the hypothesis $(H_{i-1})$ and the estimate \eqref{estimate6}, we use the Sobolev inequalities \eqref{sobolev2} to get
$$||(\mathbf{U}^a,{\mathbf V}_k,\delta {\mathbf V}_k,\Psi^a,\Psi_k,\delta\Psi_k)||_{W^{2,\infty}_{\ast}(\Omega_T)}\lesssim 1.$$
Then, we apply Lemma \ref{lb} and use Sobolev inequalities \eqref{sobolev2}, the assumption \eqref{small} and the  hypothesis $(H_{i-1})$ to give
\begin{equation}\nonumber
\begin{split}
||e'_k||_{s,\ast,T}&\lesssim \delta^2\theta^{10-2\alpha}_k\Delta^2_k(1+||({\mathbf V}_k,\Psi_k)||_{s+2,\ast,T}+\delta\theta^{s+1-\alpha}_k\Delta_k)\\
&\quad\quad+||(\delta {\mathbf V}_k,\delta\Psi_k)||_{s+2,\ast,T}||(\delta {\mathbf V}_k,\delta\Psi_k)||_{6,\ast,T}\\
&\lesssim \delta^2\theta^{10-2\alpha}_k\Delta^2_k(1+||({\mathbf V}_k,\Psi_k)||_{s+2,\ast,T})+\delta^2\theta^{s+6-2\alpha}_k\Delta^2_k,
\end{split}
\end{equation}
for all $s\in\{6,\cdots,\tilde{\alpha}-2\}.$
If $s+2\neq\alpha,$ then it follows from \eqref{estimate6} and $2\theta_k\Delta_k\leq1,$ that
\begin{equation}\nonumber
\begin{split}
||e'_k||_{s,\ast,T}&\lesssim \delta^2\Delta^2_k(\theta^{(s+2-\alpha)_++10-2\alpha}_k+\theta^{s+6-2\alpha}_k)\\
&\lesssim \delta^2\theta^{L_1(s)-1}_k\Delta_k.
\end{split}
\end{equation}
If $s+2=\alpha,$ then it follows from \eqref{estimate6} and $\alpha\geq7,$ that
\begin{equation}\nonumber
\begin{split}
||e'_k||_{s,\ast,T}&\lesssim \delta^2\Delta^2_k(\theta^{11-2\alpha}_k+\theta^{4-\alpha}_k)\\
&\lesssim \delta^2\theta^{L_1(\alpha-2)-1}_k\Delta_k.
\end{split}
\end{equation}
Therefore, we obtain \eqref{ee1}.
Now, we prove \eqref{ee3}. Using Lemma \ref{lb} and trace Theorem \ref{teoA1}, we obtain
\begin{equation}\nonumber
\begin{split}
||\tilde{e}'_k||_{H^s(\Gamma_T)}&\lesssim ||\delta {\mathbf V}_k||_{H^s(\Gamma_T)}||\delta \psi_k||_{W^{1,\infty}(\Gamma_T)}+||\delta {\mathbf V}_k||_{L^{\infty}(\Gamma_T)}||\delta \psi_k||_{H^{s+1}(\Gamma_T)}\\
&\quad + ||\delta {\mathbf V}_k||_{H^s(\Gamma_T)}||\delta {\mathbf V}_k||_{L^{\infty}(\Gamma_T)}\\
&\lesssim \delta^2\theta^{L_1(s)-1}_k\Delta_k.
\end{split}
\end{equation}
This completes the proof of Lemma \ref{eestimate}.
\end{proof}

\subsubsection{Estimate of the first substitution errors}\label{first}
We can estimate the first substitution errors $e''_k, \tilde{e}''_k$ of the iteration scheme, defined in \eqref{l} and \eqref{B}. We rewrite
\begin{equation}\label{e4}
e''_k:=\mathcal{L}'({\mathbf V}_k,\Psi_k)(\delta {\mathbf V}_k,\delta\Psi_k)-\mathcal{L}'(S_{\theta_k}{\mathbf V}_k,S_{\theta_k}\Psi_k)(\delta {\mathbf V}_k,\delta\Psi_k),
\end{equation}
\begin{equation}\label{e6}
\tilde{e}''_k:=\mathcal{B}'({\mathbf V}_k|_{x_1=0},\psi_k)(\delta {\mathbf V}_k|_{x_1=0},\delta\psi_k)-\mathcal{B}'(S_{\theta_k}{\mathbf V}_k|_{x_1=0},S_{\theta_k}\psi_k)(\delta {\mathbf V}_k|_{x_1=0},\delta\psi_k).
\end{equation}
\begin{lemma}\label{se}
Let $\alpha\geq7.$ There exist $\delta>0$ sufficiently small and $\theta_0\geq1$ sufficiently large, such that for all $k\in\{0,\cdots,i-1\}$ and for all integer $s\in\{6,\cdots,\tilde{\alpha}-2\},$ we have
\begin{equation}\label{e6666}
||e''_k||_{s,\ast,T}\lesssim\delta^2\theta^{L_2(s)-1}_k\Delta_k,
\end{equation}
\begin{equation}\label{e82}
||\tilde{e}''_k||_{H^{s}(\Gamma_T)}\lesssim\delta^2\theta^{L_2(s)-1}_k\Delta_k,
\end{equation}
where $L_2(s):=\max\{(s+2-\alpha)_++12-2\alpha;s+8-2\alpha\}.$
\end{lemma}

\begin{proof}
In view of \eqref{e4} and \eqref{e6} we have
\begin{equation}\nonumber
\begin{split}
e''_k&=\int^1_0\mathbb{L}''(\mathbf{U}^a+S_{\theta_k}{\mathbf V}_k+\tau(I-S_{\theta_k}){\mathbf V}_k,\Psi^a+S_{\theta_k}\Psi_k+\tau(I-S_{\theta_k})\Psi_k)\\
&\quad\quad((\delta {\mathbf V}_k,\delta\Psi_k),((I-S_{\theta_k}){\mathbf V}_k,(I-S_{\theta_k})\Psi_k))d\tau,\\
\tilde{e}''_k&=\mathbb{B}''((\delta {\mathbf V}_k|_{x_1=0},\delta\psi_k),((I-S_{\theta_k}){\mathbf V}_k|_{x_1=0},(I-S_{\theta_k})\psi_k)).
\end{split}
\end{equation}
Using \eqref{estimate7} and \eqref{estimate8}, we have
$$||(S_{\theta_k}{\mathbf V}_k,{\mathbf V}_k,S_{\theta_k}\Psi_k,\Psi_k)||_{W^{2,\infty}_{\ast}(\Omega_T)}\lesssim ||(S_{\theta_k}{\mathbf V}_k,{\mathbf V}_k,S_{\theta_k}\Psi_k,\Psi_k)||_{6,\ast,T}\lesssim 1.$$
Next, we apply Lemma \ref{lb}, use Sobolev inequalities \eqref{sobolev2}, \eqref{small}, the Hypothesis $(H_{i-1})$ and \eqref{estimate7} to get that
\begin{equation}\nonumber
||e''_k||_{s,\ast,T}\lesssim \delta^2\theta^{11-2\alpha}_k\Delta_k(1+||(S_{\theta_k}{\mathbf V}_k,S_{\theta_k}\Psi_k)||_{s+2,\ast,T})+\delta^2\theta^{s+7-2\alpha}_k\Delta_k,
\end{equation}
for all $s\in\{6,\cdots,\tilde{\alpha}-2\}.$
Similar to the proof of Lemma \ref{eestimate}, we can discuss $s+2\neq \alpha$ and $s+2=\alpha$ separately. Hence, using \eqref{estimate8}, we can obtain \eqref{e6666} and \eqref{e82}. The proof of Lemma \ref{se} is completed.
\end{proof}

\subsubsection{Estimate of the modified state}\label{modifye}

We need to construct a smooth modified state $({\mathbf V}_{i+\frac{1}{2}},\psi_{i+\frac{1}{2}})$ such that $({\mathbf U}^a+{\mathbf V}_{i+\frac12},\varphi^a+\psi_{i+\frac12})$ satisfies the nonlinear constraints \eqref{hy}--\eqref{c}, \eqref{c2} and \eqref{assumption}. In this regard, we remark it is crucial that the sum of the approximate solution $({\mathbf U}^a,\varphi^a)$ and the modified state $({\mathbf V}_{i+\frac12},\psi_{i+\frac12})$, instead of the latter two separately, satisfies the aforementioned nonlinear constraints; indeed it is just this sum which plays the role of basic state around which we need to linearize problem \eqref{IVP} in the iteration scheme leading to its solution, see problem \eqref{effective3}.
In the construction of some components of the modified state we follow the similar approach of \cite{Alinhac,Coulombel2008,S-T plasma vuoto}, while for the magnetic field we are inspired by \cite{S-T plasma vuoto,Trakhinin2009}.

\begin{lemma}\label{mo}
Let $\alpha\geq10$. There exist some functions ${\mathbf V}_{i+\frac{1}{2}},\Psi_{i+\frac{1}{2}},\psi_{i+\frac{1}{2}}$  vanishing in the past, such that
$({\mathbf U}^a+{\mathbf V}_{i+\frac{1}{2}}, \Psi^a+\Psi_{i+\frac{1}{2}}, \varphi^a+\psi_{i+\frac{1}{2}})$ satisfy the constraints \eqref{hy}-- \eqref{c}, \eqref{c2} and \eqref{assumption}; moreover,
\begin{equation}\label{mo1}
\Psi_{i+\frac{1}{2}}^\pm=S_{\theta_i}\Psi^{\pm}_i\,,\quad \psi_{i+\frac{1}{2}}=(S_{\theta_i}\Psi^{\pm}_i)|_{x_1=0,},
\end{equation}
\begin{equation}\label{mo2}
p^{\pm}_{i+\frac12}=S_{\theta_i}p^\pm_i\,,\quad u^{\pm}_{2,i+\frac{1}{2}}=S_{\theta_i}u^{\pm}_{2,i}\,,\quad S^{\pm}_{i+\frac{1}{2}}=S_{\theta_i}S^{\pm}_i,
\end{equation}
\begin{equation}\label{mo32}
||{\mathbf V}_{i+\frac{1}{2}}-S_{\theta_i}{\mathbf V}_i||_{s,\ast,T}\lesssim\delta \theta^{s+4-\alpha}_i, \text{ for } s\in\{6,\cdots,\tilde{\alpha}+4\},
\end{equation}
for sufficiently small $\delta>0$ and $T>0,$ and sufficiently large $\theta_0\geq1.$
\end{lemma}
\begin{proof}
To shortcut notation, in the proof the $\pm$ indices are omitted. Let us define $\Psi_{i+1/2}$, $\psi_{i+1/2}$, $p_{i+1/2}$, $S_{i+1/2}$ and the tangential component $u_{2,i+1/2}$ of the velocity as in \eqref{mo1}, \eqref{mo2}; this is similar as in \cite{Alinhac}, \cite[Proposition 7]{Coulombel2008}, \cite[Proposition 28]{S-T plasma vuoto}. It is easily checked that all these functions vanish in the past.

\medskip
\noindent
{\em Construction of the modified normal velocity.}

\smallskip
\noindent
In order to construct the normal component $u_{1,i+1/2}$ of the velocity, we follow the idea of \cite{Alinhac,Coulombel2008}. We first introduce the following function $\mathcal G$:
\begin{equation*}\label{GG}
\begin{split}
\mathcal G:=&\partial_t(\varphi^a+\psi_{i+1/2})-(u^a_1+S_{\theta_i}u_{1,i})\vert_{x_1=0}+u^a_2\vert_{x_1=0}\partial_2\varphi^a\\
&+(u^a_2+u_{2,i+1/2})\vert_{x_1=0}\partial_2\psi_{i+1/2}+u_{2,i+1/2}\vert_{x_1=0}\partial_2\varphi^a\,.
\end{split}
\end{equation*}
The normal component of the velocity $u_{1,i+1/2}$ is defined by
\begin{equation}\label{mod-nor}
u_{1,i+1/2}:=S_{\theta_i}u_{1,i}+\mathcal R_T\mathcal G\,,
\end{equation}
where $\mathcal R_T$ is the lifting operator $H^{s-1}(\Gamma_T)\rightarrow H^s_\ast(\Omega_T)$, $s>1$, see \cite{Ohno-Shizuta1994}. It is easily checked that $u_{1,i+1/2}$ vanishes in the past.
\newline
Let us note that, by construction, $u_{1,i+1/2}$ satisfies the following equation on the boundary
\begin{equation}\label{boundary a}
\partial_t(\varphi^a+\psi_{i+1/2})-(u_1^a+u_{1,i+1/2})+(u_2^a+u_{2,i+1/2})\partial_2(\varphi^a+\psi_{i+1/2})=0\,,\quad\mbox{on}\,\,\,\{x_1=0\}\,.
\end{equation}
We prove the estimate \eqref{mo32} for the part regarding ${\mathbf u}_{i+1/2}$. We have
\begin{equation*}
||\mathbf{u}_{i+1/2}-S_{\theta_i}{\mathbf u}_i||_{s,\ast,T}=||\mathcal R_T\mathcal G||_{s,\ast,T}\lesssim||\mathcal G||_{H^{s-1}(\Gamma_T)}\,.
\end{equation*}
Now we rewrite $\mathcal G$ in a more convenient form, by using the error $\varepsilon^i$ defined by
\begin{equation}\label{epsilon i}
\begin{split}
\varepsilon^i:=&\partial_t\psi_i-u_{1,i}\vert_{x_1=0}+(\partial_t\varphi^a-u_1^a\vert_{x_1=0}+u_2^a\vert_{x_1=0}\partial_2\varphi^a)+(u^a_2+u_{2,i})\vert_{x_1=0}\partial_2\psi_i\\
&+u_{2,i}\vert_{x_1=0}\partial_2\varphi^a=\mathcal B({\mathbf V}_i\vert_{x_1=0},\psi_i)_1\,.
\end{split}
\end{equation}
In view of \eqref{mo2}, by means of $\varepsilon^i$ we may rewrite $\mathcal G$ as:
\begin{equation}\label{G'}
\begin{split}
\mathcal G&= S_{\theta_i}\varepsilon^i-[S_{\theta_i},\partial_t]\psi_i+(I-S_{\theta_i})(\partial_t\varphi^a-u^a_1+u^a_2\partial_2\varphi^a)-S_{\theta_i}((u^a_2+u_{2,i})\partial_2\psi_i)\\
&\quad-S_{\theta_i}(u_{2,i}\partial_2\varphi^a)+(u^a_2+S_{\theta_i}u_{2,i})\partial_2\psi_{i+1/2}+S_{\theta_i}u_{2,i}\partial_2\varphi^a\\
&=S_{\theta_i}\varepsilon^i-[S_{\theta_i},\partial_t]\psi_i+(I-S_{\theta_i})(\partial_t\varphi^a-u^a_1+u^a_2\partial_2\varphi^a))\\
&\quad-\left(S_{\theta_i}(u^a_2\partial_2\psi_i)-u_2^a\partial_2 S_{\theta_i}\psi_i\right)-\left(S_{\theta_i}u_{2,i}\partial_2 S_{\theta_i}\psi_i-S_{\theta_i}(u_{2,i}\partial_2\psi_i)\right)\\
&\quad-\left(S_{\theta_i}u_{2,i}\partial_2\varphi^a-S_{\theta_i}(u_{2,i}\partial_2\varphi^a)\right).
\end{split}
\end{equation}
To estimate the first term $S_{\theta_i}\varepsilon^i$ on the right-hand side, we use the decomposition:
\[
\begin{split}
\varepsilon^i&=\left(\mathcal B({\mathbf V}_i,\psi_i)_1-\mathcal B({\mathbf V}_{i-1},\psi_{i-1})_1\right)+\mathcal B({\mathbf V}_{i-1},\psi_{i-1})_1\\
&=\partial_t\delta\psi_{i-1}-\delta u_{1,i-1}+(u_2^a+u_{2,i-1})\partial_2\delta\psi_{i-1}+\delta u_{2,i-1}\partial_2(\varphi^a+\delta\psi_{i-1}+\psi_{i-1})\\
&\quad+\mathcal B({\mathbf V}_{i-1},\psi_{i-1})_1\,.
\end{split}
\]
Then we exploit point {\em (c)} of ($H_{i-1}$) and the properties of smoothing operators, to get
\begin{equation}\label{stima-Stheta epsilon}
||S_{\theta_i}\varepsilon^i||_{H^{s-1}(\Gamma_T)}\lesssim ||\varepsilon^i||_{H^{s-1}(\Gamma_T)}\lesssim\delta\theta_i^{s-\alpha-1}\,.
\end{equation}
In order to make an estimate of the commutator term $[S_{\theta_i},\partial_t]\psi_i$, we use different arguments for large and small orders $s$.
\newline
For all $s\in\{6,\dots,\alpha\}$ we write the commutator as
\[
[\partial_t,S_{\theta_i}]\psi_i=\partial_t(S_{\theta_i}-I)\psi_i+(I-S_{\theta_i})\partial_t\psi_i\,,
\]
then we use estimates \eqref{as2}, \eqref{estimate6}, \eqref{estimate7} to get
\begin{equation}\label{stima comm}
||[\partial_t,S_{\theta_i}]\psi_i||_{s,\ast,T}\leq C\delta\theta_i^{s-\alpha}\,,\quad s\in\{6,\dots,\alpha\}\,.
\end{equation}
In order to get the similar estimate as \eqref{stima comm} for $s\in\{\alpha+1,\dots,\tilde\alpha+6\}$, we directly estimate the two terms of the commutator $[\partial_t,S_{\theta_i}]\psi_i=\partial_tS_{\theta_i}\psi_i-S_{\theta_i}\partial_t\psi_i$, using \eqref{as1} and \eqref{estimate8}.
\newline
To estimate the third term in the right-hand side of \eqref{G'}, we proceed as above by applying different arguments to small and large orders $s$.
\newline
For integers $s\in\{6,\dots,\alpha\}$, we apply \eqref{as2} to get
\[
\begin{split}
||(I-S_{\theta_i})&(\partial_t\varphi^a-u^a_1+u_2^a\partial_2\varphi^a)||_{H^{s-1}(\Gamma_T)}\le C\theta_i^{s-\alpha}||\partial_t\varphi^a-u^a_1+u_2^a\partial_2\varphi^a||_{H^{\alpha-1}(\Gamma_T)}\\
&\le C\theta_i^{s-\alpha}(||\varphi^a||_{H^{\alpha-1}(\Gamma_T)}+||u^a||_{\alpha,\ast,T})\le  C\delta\theta_i^{s-\alpha}\,,
\end{split}
\]
in view of \eqref{small}.
\newline
For integers $s\in\{\alpha+1,\dots.\tilde\alpha+6\}$, we use \eqref{as1} and \eqref{small} to estimate directly
\[
\begin{split}
||(I-S_{\theta_i})&(\partial_t\varphi^a-u^a_1+u_2^a\partial_2\varphi^a)||_{H^{s-1}(\Gamma_T)}\le ||E^a||_{H^{s-1}(\Gamma_T)}+||S_{\theta_i}E^a||_{H^{s-1}(\Gamma_T)}\\
&\le||E^a||_{H^{s-1}(\Gamma_T)}+C\theta_i^{(s-1-(\alpha-1))_+}||E^a||_{H^{\alpha-1}(\Gamma_T)}\le C\theta_i^{s-\alpha}||E^a||_{H^{s-1}(\Gamma_T)}\\
&\le C\theta_i^{s-\alpha}(||\varphi^a||_{H^{s}(\Gamma_T)}+||u^a||_{s,\ast,T})\le  C\delta\theta_i^{s-\alpha}\,,
\end{split}
\]
where $E^a:=\partial_t\varphi^a-u^a_1+u_2^a\partial_2\varphi^a$ has been set.
\newline
Let us now estimate the fourth term $-\left(S_{\theta_i}(u^a_2\partial_2\psi_i)-u_2^a\partial_2 S_{\theta_i}\psi_i\right)$; once again we need to argue separately on different values of $s$.
\newline
For small integers $s\in\{6,\dots,\tilde\alpha\}$ we rewrite the above in the form
\begin{equation}\label{split}
u_2^a\partial_2 S_{\theta_i}\psi_i-S_{\theta_i}(u^a_2\partial_2\psi_i)=u_2^a\partial_2(S_{\theta_i}-I)\psi_i+(I-S_{\theta_i})(u_2^a\partial_2\psi_i)\,,
\end{equation}
which takes advantage of the appearing of the difference $I-S_{\theta_i}$; thus estimate \eqref{estimate7} and Moser-type calculus inequalities of Lemma \ref{moser1} yield
\begin{equation}\label{est-split1}
\begin{split}
||u_2^a\partial_2(S_{\theta_i}-I)&\psi_i||_{H^{s-1}(\Gamma_T)}\lesssim ||u^a_2||_{L^\infty(\Gamma_T)}||(I-S_{\theta_i})\psi_i||_{H^s(\Gamma_T)}\\
&+||u^a||_{s,\ast,T}||(I-S_{\theta_i})\psi_i||_{H^2(\Gamma_T)}\leq C\delta\theta_i^{s-\alpha}\,,\quad\mbox{for}\,\,\,s\in\{6,\dots,\tilde\alpha\}\,.
\end{split}
\end{equation}
As for the second term in the right-hand side of \eqref{split}, a further splitting of the range of $s$ covered by estimate \eqref{est-split1} is required. For integers $s$ such that $6\le s\le\alpha+1$, we apply \eqref{as2}, Moser-type calculus inequalities of Lemma \ref{moser1} and \eqref{estimate6} to obtain
\begin{equation}\label{est-split2.1}
\begin{split}
||(I-S_{\theta_i})&(u_2^a\partial_2\psi_i)||_{H^{s-1}(\Gamma_T)}\le C\theta_i^{s-1-(\alpha+1)}||u_2^a\partial_2\psi_i||_{H^{\alpha+1}(\Gamma_T)}\\
&\le C\theta_i^{s-\alpha-2}\left\{||u^a||_{L^\infty(\Gamma_T)}||\psi_i||_{H^{\alpha+1}(\Gamma_T)}+||u^a||_{\alpha+2,\ast,T}||\psi_i||_{H^3(\Gamma_T)}\right\}\\
&\le C\theta_i^{s-\alpha-2}(\delta\theta_i+\delta)\le C\delta\theta_i^{s-\alpha-1}\,,\quad\mbox{for}\,\,\,6\le s\le\alpha+1\,.
\end{split}
\end{equation}
On the other hand for integers $s$ such that $\alpha+1<s\le\tilde\alpha$, we use \eqref{as1}, Moser-type calculus inequalities and \eqref{estimate6} to get
\begin{equation}\label{est-split2.2}
\begin{split}
||(I-S_{\theta_i})&(u_2^a\partial_2\psi_i)||_{H^{s-1}(\Gamma_T)}\le C||u_2^a\partial_2\psi_i||_{H^{s-1}(\Gamma_T)}\\
&\le C\left\{||u^a||_{L^\infty(\Gamma_T)}||\psi_i||_{H^{s}(\Gamma_T)}+||u^a||_{s,\ast,T}||\psi_i||_{H^3(\Gamma_T)}\right\}\\
&\le C\theta_i^{(s-\alpha)_+}+C\delta\le C\delta \theta_i^{s-\alpha} \,,\quad\mbox{for}\,\,\,\alpha+1<s\le\tilde\alpha\,.
\end{split}
\end{equation}
Gathering \eqref{est-split1}--\eqref{est-split2.2} we end up with
\begin{equation}\label{est1}
||u_2^a\partial_2 S_{\theta_i}\psi_i-S_{\theta_i}(u^a_2\partial_2\psi_i)||_{H^{s-1}(\Gamma_T)}\le C\delta\theta_i^{s-\alpha}\,,\quad\mbox{for}\,\,\,s\in\{6,\dots,\tilde\alpha\}\,.
\end{equation}
For higher integers $s\in\{\tilde\alpha+1,\dots,\tilde\alpha+6\}$, we estimate separately the two terms of the difference $u_2^a\partial_2 S_{\theta_i}\psi_i-S_{\theta_i}(u^a_2\partial_2\psi_i)$; using again estimates \eqref{as1}, \eqref{estimate6}, \eqref{estimate8} and Moser-type calculus inequalities, we obtain
\begin{equation*}
\begin{split}
||u^a_2\partial_2 S_{\theta_i}\psi_i||_{H^{s-1}(\Gamma_T)}\lesssim ||u^a_2||_{L^\infty(\Gamma_T)}||S_{\theta_i}\psi_i||_{H^s(\Gamma_T)}+||u^a_2||_{s,\ast,T}||S_{\theta_i}\psi_i||_{H^3(\Gamma_T)}\le C\delta\theta_i^{s-\alpha}\,;
\end{split}
\end{equation*}
\begin{equation*}
\begin{split}
||S_{\theta_i}&(u^a_2\partial_2\psi_i)||_{H^{s-1}(\Gamma_T)}\leq C\theta_i^{s-1-\alpha}||u^a_2\partial_2\psi_i||_{H^{\alpha}(\Gamma_T)}\\
&\le C\theta_i^{s-1-\alpha}\left\{ ||u^a_2||_{L^\infty(\Gamma_T)}||\psi_i||_{H^{\alpha+1}(\Gamma_T)}+||u^a_2||_{\alpha+1,\ast,T}||\psi_i||_{H^3(\Gamma_T)}\right\}\\
&\le C\theta_i^{s-1-\alpha}(\delta\theta_i+C\delta)\le C\delta\theta_i^{s-\alpha} \,,\quad\mbox{for}\,\,\,\tilde\alpha+1\le s\le\tilde\alpha+6\,.
\end{split}
\end{equation*}
Adding the last two inequalities we end up with
\begin{equation}\label{est2}
||u_2^a\partial_2 S_{\theta_i}\psi_i-S_{\theta_i}(u^a_2\partial_2\psi_i)||_{H^{s-1}(\Gamma_T)}\le C\delta\theta_i^{s-\alpha}\,,\quad\mbox{for}\,\,\,s\in\{\tilde\alpha+1,\dots,\tilde\alpha+6\}\,.
\end{equation}
Gathering estimates \eqref{est1} and \eqref{est2} provide the estimate for all integers $s\in\{6,\dots,\tilde\alpha+6\}$.
\newline
For the last two terms in the right-hand side of \eqref{G'} we use the same arguments as above, where we still separate small and large orders $s$; in the case of small $s$ we manage to rewrite the expression in order to make advantage from the boundedness properties \eqref{as2}, \eqref{estimate7} of $I-S_{\theta_i}$; for large $s$ we estimate separately each term of the difference using Moser-type calculus inequalities and \eqref{as1}, \eqref{estimate6}, \eqref{estimate8}. Doing so, we derive:
\[
||S_{\theta_i}u_{2,i}\partial_2S_{\theta_i}\psi_i-S_{\theta_i}(u_{2,i}\partial_2\psi_i)||_{H^{s-1}(\Gamma_T)}\le C\delta\theta_i^{s-\alpha+1}\,,\quad\mbox{for}\,\,\,s\in\{6,\dots,\tilde\alpha+6\}\,;
\]
\[
||S_{\theta_i}u_{2,i}\partial_2\varphi^a-S_{\theta_i}(u_{2,i}\partial_2\varphi^a)||_{H^{s-1}(\Gamma_T)}\le C\delta\theta_i^{s-\alpha}\,,\quad\mbox{for}\,\,\,s\in\{6,\dots,\tilde\alpha+6\}\,.
\]
Gathering all the previously found estimates we end up with
\begin{equation}\label{stima_errore_u_i+1/2}
||\mathbf{u}_{i+1/2}-S_{\theta_i}{\mathbf u}_i||_{s,\ast,T}\lesssim||\mathcal G||_{H^{s-1}(\Gamma_T)}\le C\delta\theta_i^{s-\alpha+1}\,,\quad\mbox{for}\,\,\,s\in\{6,\dots,\tilde\alpha+6\}\,.
\end{equation}
In the above estimate it is fundamental that $s\le\tilde{\alpha}+6$ in order to prove the following \eqref{10.46}. Let us recall that the above estimate holds under the smallness assumption \eqref{small}.

\medskip
\noindent
{\em Construction of the modified magnetic field.}

\smallskip
\noindent
Let us see now how to define the modified magnetic field $\mathbf {H}_{i+1/2}$, following somehow \cite[Proposition 28]{S-T plasma vuoto}, \cite[Proposition 12]{Trakhinin2009}.
\newline
Let us denote the nonlinear equation satisfied by the magnetic field in \eqref{IVP} by
\begin{equation*}
\mathbb L_H({\mathbf u},\mathbf{H},\Psi)=0\,,\quad\mbox{in}\,\,\,\Omega_T\,.
\end{equation*}
The field $\mathbf{H}_{i+1/2}$ should be such that $\mathbf {H}^a+\mathbf {H}_{i+1/2}$ satisfies \eqref{c}, that is
\begin{equation}\label{nonl_H}
\mathbb L_H({\mathbf u}^a+\mathbf{u}_{i+1/2},\mathbf{H}^a+\mathbf{H}_{i+1/2},\Psi^a+\Psi_{i+1/2})=0\,,\quad\mbox{in}\,\,\,\Omega_T\,.
\end{equation}
We note that equation \eqref{nonl_H} is linear in $\mathbf{H}^a+\mathbf{H}_{i+1/2}$ and does not need to be supplemented with any boundary condition; in fact, the coefficient of $\partial_1(\mathbf{H}^a+\mathbf{H}_{i+1/2})$ is zero along the boundary because of \eqref{boundary a} (the left-hand side of \eqref{boundary a} is nothing but $w_1\vert_{x_1=0}$, computed for ${\mathbf u}^a+{\mathbf u}_{i+1/2}$ and $\Psi^a+\Psi_{i+1/2}$ instead of $\mathbf{u}$ and $\Psi$ respectively).
\newline
Therefore, for given ${\mathbf u}_{i+1/2}$, $\Psi_{i+1/2}$, ${\mathbf U}^a$ and $\Psi^a$, \eqref{nonl_H} has a unique solution $\mathbf H'$, from which we derive the existence of a unique $\mathbf H_{i+1/2}=\mathbf H'-\mathbf H^a$, vanishing in the past.
\newline
In order to estimate ${\mathbf H}_{i+1/2}-S_{\theta_i}{\mathbf H}_i$, we first observe that \eqref{nonl_H} yields
\[
\begin{split}
\mathbb L_H&({\mathbf u}^a+\mathbf{u}_{i+1/2},{\mathbf H}_{i+1/2}-S_{\theta_i}{\mathbf H}_i,\Psi^a+\Psi_{i+1/2})\\
&=\mathbb L_H({\mathbf u}^a+\mathbf{u}_{i+1/2},{\mathbf H}^a+{\mathbf H}_{i+1/2}-S_{\theta_i}{\mathbf H}_i,\Psi^a+\Psi_{i+1/2})\\
&\quad\quad -\mathbb L_H({\mathbf u}^a+\mathbf{u}_{i+1/2},{\mathbf H}^a,\Psi^a+\Psi_{i+1/2})\\
&=-\mathbb L_H({\mathbf u}^a+\mathbf{u}_{i+1/2},{\mathbf H}^a+S_{\theta_i}{\mathbf H}_i,\Psi^a+S_{\theta_i}\Psi_{i})\,.
\end{split}
\]
Then $\mathbf H_{i+1/2}-S_{\theta_i}\mathbf H_i$ solves the equation
\begin{equation}\label{eq Hi+1/2-Stheta}
\mathbb L_H({\mathbf u}^a+\mathbf{u}_{i+1/2},{\mathbf H}_{i+1/2}-S_{\theta_i}{\mathbf H}_i,\Psi^a+\Psi_{i+1/2})=F^{i+1/2}_H\,,
\end{equation}
where
\begin{equation}\label{FH}
F_H^{i+1/2}:=\Delta_1+\Delta_2-S_{\theta_i}\mathbb L_H({\mathbf u}^a+\mathbf{u}_{i},{\mathbf H}^a+{\mathbf H}_i,\Psi^a+\Psi_{i})\,;
\end{equation}
\begin{equation*}\label{delta1}
\Delta_1:=S_{\theta_i}\mathbb L_H({\mathbf u}^a+\mathbf{u}_{i},{\mathbf H}^a+{\mathbf H}_i,\Psi^a+\Psi_{i})-\mathbb L_H({\mathbf u}^a+S_{\theta_i}\mathbf{u}_{i},{\mathbf H}^a+S_{\theta_i}{\mathbf H}_i,\Psi^a+S_{\theta_i}\Psi_{i})\,;
\end{equation*}
\begin{equation*}\label{delta2}
\begin{split}
\Delta_2&:=\mathbb L_H({\mathbf u}^a+S_{\theta_i}\mathbf{u}_{i},{\mathbf H}^a+S_{\theta_i}{\mathbf H}_i,\Psi^a+S_{\theta_i}\Psi_{i})\\
&\qquad\qquad\qquad -\mathbb L_H({\mathbf u}^a+\mathbf{u}_{i+1/2},{\mathbf H}^a+S_{\theta_i}{\mathbf H}_i,\Psi^a+S_{\theta_i}\Psi_{i})\,.
\end{split}
\end{equation*}

Let us first write the explicit form of $\Delta_1$. Using the definition of the nonlinear operator $\mathbb L_H({\bf u},{\bf H},\Psi)$ we have
\begin{equation*}\label{delta1_exp}
\begin{split}
\Delta_1=&S_{\theta_i}\Big\{\partial_t({\bf H}^a+{\bf H}_i)+\frac1{\partial_1(\Phi^a+\Psi_i)}\big(({\bf w}[{\bf u}^a+{\bf u}_i]\cdot\nabla)({\bf H}^a+{\bf H}_i)\\
& -({\bf h}[{\bf H}^a+{\bf H}_i]\cdot\nabla)({\bf u}^a+{\bf u}_i)+({\bf H}^a+{\bf H}_i){\rm div}({\bf v}[{\bf u}^a+{\bf u}_i]))\big)\Big\}\\
&-\Big\{\partial_t({\bf H}^a+S_{\theta_i}{\bf H}_i)+\frac1{\partial_1(\Phi^a+S_{\theta_i}\Psi_{i})}\left({\bf w}[{\bf u}^a+S_{\theta_i}{\bf u}_i]\cdot\nabla\right)({\bf H}^a+S_{\theta_i}{\bf H}_i)\\
& -({\bf h}[{\bf H}^a+S_{\theta_i}{\bf H}_i]\cdot\nabla)({\bf u}^a+S_{\theta_i}{\bf u}_i)+({\bf H}^a+S_{\theta_i}{\bf H}_i){\rm div}({\bf v}[{\bf u}^a+S_{\theta_i}{\bf u}_i]))\Big\}\,,
\end{split}
\end{equation*}
where
\begin{equation*}\label{v[]}
\begin{split}
&{\bf v}[{\bf u}^a+{\bf u}_i]:=\Bigg(({\bf u}^a+{\bf u_i})\cdot\Big(1,-\partial_2(\Psi^a+\Psi_i)\Big),({\bf u}^a+{\bf u}_i)_2\partial_1(\Phi^a+\Psi_i)\Bigg)\,,\\
&{\bf w}[{\bf u}^a+{\bf u}_i]:={\bf v}[{\bf u}^a+{\bf u}_i]-\Big(\partial_t(\Psi^a+\Psi_i),0\Big)\,,\\
&{\bf h}[{\bf H}^a+{\bf H}_i]:=\Bigg(({\bf H}^a+{\bf H_i})\cdot\Big(1,-\partial_2(\Psi^a+\Psi_i)\Big),({\bf H}^a+{\bf H}_i)_2\partial_1(\Phi^a+\Psi_i)\Bigg)\,.\\
\end{split}
\end{equation*}
The vectors ${\bf v}[{\bf u}^a+S_{\theta_i}{\bf u}_i]$, ${\bf w}[{\bf u}^a+S_{\theta_i}{\bf u}_i]$, ${\bf h}[{\bf H}^a+S_{\theta_i}{\bf H}_i]$ are defined by completely similar expressions with $S_{\theta_i}{\bf u}_i$, $S_{\theta_i}{\bf H}_i$, $S_{\theta_i}\Psi_i$ instead of ${\bf u}_i$, ${\bf H}_i$, $\Psi_i$.

We split the range of $s$ into small and large values (in order to take advantage of the continuity estimates of $I-S_{\theta_i}$, see \eqref{as2} and \eqref{estimate7}).
For small values of $s$, by using \eqref{estimate6}, \eqref{estimate7}, the smallness assumption \eqref{small} and Moser-type calculus inequalities we get
\begin{equation}\label{10.41}
	||{\bf v}[{\bf u}^a+{\bf u}_{i}]||_{s,\ast,T}\le C\delta^2
	\begin{cases}
\theta_i^{(s+2-\alpha)_+}\quad &s+2\not=\alpha,\\
\theta_i \quad &s+2=\alpha,
	\end{cases}\qquad s\in\{6,\dots,\tilde{\alpha}-2\},
\end{equation}
\begin{equation}\label{10.42}
	||{\bf v}[{\bf u}^a+S_{\theta_i}{\bf u}_{i}]-{\bf v}[{\bf u}^a+{\bf u}_{i}]||_{s,\ast,T}\le C\delta\theta_i^{s+2-\alpha},\qquad s\in\{6,\dots,\tilde{\alpha}-2\},\,\, \alpha\ge 9.
\end{equation}
Adding \eqref{10.41}, \eqref{10.42} gives an estimate of ${\bf v}[{\bf u}^a+S_{\theta_i}{\bf u}_{i}]$. We obtain similar estimates for ${\bf w}[{\bf u}^a+{\bf u}_{i}],{\bf h}[{\bf H}^a+{\bf H}_{i}], {\bf w}[{\bf u}^a+S_{\theta_i}{\bf u}_{i}]$ and ${\bf h}[{\bf H}^a+S_{\theta_i}{\bf H}_{i}]$.

We decompose $\Delta_1$ as sum of terms with differences $I-S_{\theta_i}$ put in evidence. Making repeated use of estimates \eqref{as1}, \eqref{as2},  \eqref{estimate6}--\eqref{estimate8}, the smallness assumption \eqref{small}, Moser-type calculus inequalities of Lemmata \ref{moser1} , \ref{moser2}, we get for $\alpha\ge 9$
\begin{equation}\label{est-delta1-bassi}
	||\Delta_1||_{s,\ast,T}\le C\delta\theta_i^{s+4-\alpha}\,,\quad\mbox{for}\,\,\,s\in\{6,\dots,\alpha-3\}\,.
\end{equation}
For large values $s\ge{\alpha}-3$ we use \eqref{as1}, \eqref{estimate8} to obtain
\begin{equation}\label{altrav}
	||{\bf v}[{\bf u}^a+S_{\theta_i}{\bf u}_{i}]||_{s,\ast,T}\le C\delta^2
	\begin{cases}
		\theta_i^{(s+2-\alpha)_+}\quad &s+2\not=\alpha,\\
		\theta_i \quad &s+2=\alpha
	\end{cases}\qquad s\in\{6,\dots,\tilde{\alpha}+6\},
\end{equation}
 with similar estimates for ${\bf w}[{\bf u}^a+S_{\theta_i}{\bf u}_{i}]$ and ${\bf h}[{\bf H}^a+S_{\theta_i}{\bf H}_{i}]$.
 With a direct estimate of the terms in $\Delta_1$ we extend \eqref{est-delta1-bassi} to the cases $s\ge{\alpha}-2$ and finally obtain
\begin{equation}\label{est-delta1}
	||\Delta_1||_{s,\ast,T}\le C\delta\theta_i^{s+4-\alpha}\,,\quad\mbox{for}\,\,\,s\in\{6,\dots,\tilde{\alpha}+4\}\,.
\end{equation}
In the above estimate \eqref{altrav} it is fundamental that $s\le\tilde{\alpha}+6$. This is used for the estimate \eqref{est-delta1} of $\Delta_1$, where the estimate of the normal derivative of $\mathbf v[\cdot]$  in the anisotropic space $H^s_\ast$, for $s\le \tilde{\alpha}+4$, requires an estimate of  $\mathbf v[\cdot]$ in $H^{s+2}_\ast$.
\newline
As for $\Delta_2$, we have the explicit expression
\begin{equation*}\label{delta2_exp}
\begin{split}
\Delta_2=&\frac1{\partial_1(\Phi^a+\Psi_{i+1/2})}\left\{({\bf w}^{i+1/2}\cdot\nabla)({\bf H}^a+S_{\theta_i}{\bf H}_i)-({\bf h}[{\bf H}^a+S_{\theta_i}{\bf H}_i]\cdot\nabla)(S_{\theta_i}{\bf u}_i-{\bf u}_{i+1/2})\right.\\
&\left.\qquad +({\bf H}^a+S_{\theta_i}{\bf H}_i){\rm div}({\bf v}[S_{\theta_i}{\bf u}_i-{\bf u}_{i+1/2}])\right\}\,,
\end{split}
\end{equation*}
where
\begin{equation}\label{v[]i+1/2}
\begin{split}
&{\bf v}[S_{\theta_i}{\bf u}_i-{\bf u}_{i+1/2}]:={\bf v}[{\bf u}^a+S_{\theta_i}{\bf u}_i]-{\bf v}[{\bf u}^a+{\bf u}_{i+1/2}]=\\
&=\Bigg((S_{\theta_i}{\bf u_i}-{\bf u}_{i+1/2})\cdot\Big(1,-\partial_2(\Psi^a+\Psi_{i+1/2})\Big),(S_{\theta_i}{\bf u}_i-{\bf u}_{i+1/2})_2\partial_1(\Phi^a+\Psi_{i+1/2})\Bigg)\,,\\
&{\bf w}^{i+1/2}:={\bf w}[{\bf u}^a+S_{\theta_i}{\bf u}_i]-{\bf w}[{\bf u}^a+{\bf u}_{i+1/2}]={\bf v}[S_{\theta_i}{\bf u}_i-{\bf u}_{i+1/2}]\,.
\end{split}
\end{equation}
From the definition \eqref{v[]i+1/2}, using Moser-type calculus inequalities, \eqref{small}, \eqref{estimate8}, \eqref{mo1}, and the estimate \eqref{stima_errore_u_i+1/2}, we obtain
\begin{equation}\label{10.46}
||{\bf v}[S_{\theta_i}{\bf u}_i-{\bf u}_{i+1/2}]||_{s,\ast,T}\le C\delta\theta_i^{s+1-\alpha}, \qquad s\in\{6,\dots,\tilde{\alpha}+6\},\,\, \alpha\ge 9.
\end{equation}
Again in \eqref{10.46} we need $s\le\tilde{\alpha}+6$ in order to get the following estimate \eqref{est-delta2} of $\Delta_2$, where the estimate of the normal derivative of $\mathbf v[\cdot]$  in the anisotropic space $H^s_\ast$, for $s\le \tilde{\alpha}+4$, requires an estimate of  $\mathbf v[\cdot]$ in $H^{s+2}_\ast$.

Making repeated use of Moser-type calculus inequalities, \eqref{small}, \eqref{estimate8}, \eqref{mo1}, the estimates \eqref{stima_errore_u_i+1/2} and \eqref{10.46}, we get for $\alpha\ge 9$
\begin{equation}\label{est-delta2}
	||\Delta_2||_{s,\ast,T}\le C\delta\theta_i^{s+3-\alpha}\,,\quad\mbox{for}\,\,\,s\in\{6,\dots,\tilde\alpha+4\}\,.
\end{equation}
To estimate the last term of $F_H^{i+1/2}$ we write
\begin{equation}\label{10.48}
	\begin{split}
&S_{\theta_i}\mathbb L_H({\mathbf u}^a+\mathbf{u}_{i},{\mathbf H}^a+{\mathbf H}_i,\Psi^a+\Psi_{i})
=S_{\theta_i}\big(\mathcal L_H(\mathbf{u}_{i},{\mathbf H}_i,\Psi_{i})+\mathbb L_H({\mathbf u}^a,{\mathbf H}^a,\Psi^a)\big)\\
&=S_{\theta_i}\big(\mathcal L_H(\mathbf{u}_{i},{\mathbf H}_i,\Psi_{i})-\mathcal F^a_H\big)\\
&=S_{\theta_i}\big(\mathcal L_H(\mathbf{u}_{i-1}+\delta\mathbf{u}_{i-1},{\mathbf H}_{i-1}+\delta{\mathbf H}_{i-1},\Psi_{i-1}+\delta\Psi_{i-1})-\mathcal L_H(\mathbf{u}_{i-1},{\mathbf H}_{i-1},\Psi_{i-1})\big)\\
&\quad+S_{\theta_i}\big(\mathcal L_H(\mathbf{u}_{i-1},{\mathbf H}_{i-1},\Psi_{i-1})-\mathcal F^a_H\big),
	\end{split}
\end{equation}
where $\mathcal L_H,\mathcal F^a_H$ are the $\mathbf{H}$-component in \eqref{system}, \eqref{f}, respectively. From \eqref{as1},  $(H_{i-1})(b)$ we readily obtain
\begin{equation}\label{10.43}
||S_{\theta_i}\big(\mathcal L_H(\mathbf{u}_{i-1},{\mathbf H}_{i-1},\Psi_{i-1})-\mathcal F^a_H\big)||_{s,\ast,T}\le C\delta\theta_i^{s-\alpha-1}
\quad\mbox{for}\,\,\,s\in\{6,\dots,\tilde\alpha-2\}\,.
\end{equation}
For the first difference in \eqref{10.48} let us denote
\begin{equation*}
	\begin{split}
\Delta_3:=&\mathcal L_H(\mathbf{u}_{i-1}+\delta\mathbf{u}_{i-1},{\mathbf H}_{i-1}+\delta{\mathbf H}_{i-1},\Psi_{i-1}+\delta\Psi_{i-1})\\
&-\mathcal L_H(\mathbf{u}_{i-1},{\mathbf H}_{i-1},\Psi_{i-1}).
	\end{split}
\end{equation*}
Hence we have
\begin{equation}\label{10.44}
	\begin{split}
||S_{\theta_i}\Delta_3||_{s,\ast,T} \le C\theta_i^{s-6}||\Delta_3||_{6,\ast,T} \le C\theta_i^{s-6}\cdot\delta\theta_i^{8-\alpha}=C\delta\theta_i^{s+2-\alpha}, \quad\mbox{for}\,\,\,s\ge 6,
	\end{split}
\end{equation}
where in particular for the estimate of $||\Delta_3||_{6,\ast,T}$ we have used $(H_{i-1})(a)$, \eqref{estimate6}, \eqref{10.41}  and
\begin{equation*}
	||{\bf v}[{\bf u}^a+\mathbf{u}_{i-1}+\delta\mathbf{u}_{i-1}]-{\bf v}[{\bf u}^a+{\bf u}_{i-1}]||_{s,\ast,T}\le C\delta\theta_i^{s-\alpha},\qquad\,\,\,s\in\{6,\dots,\tilde\alpha-2\}\,,
\end{equation*}
used in the cases $s=6$, $s=8$ and a similar estimate for ${\bf w}[{\bf u}^a+\mathbf{u}_{i-1}+\delta\mathbf{u}_{i-1}]-{\bf w}[{\bf u}^a+{\bf u}_{i-1}]$ in case $s=6$. Collecting \eqref{est-delta1}, \eqref{est-delta2}, \eqref{10.43}, \eqref{10.44} we get
\begin{equation}\label{stima-FH}
||F_H^{i+1/2}||_{s,\ast,T}\le
C\delta\theta_i^{s+4-\alpha}\,,\quad\mbox{for}\,\,\,s\in\{6,\dots,\tilde{\alpha}+4\}\,.
\end{equation}
Equation \eqref{eq Hi+1/2-Stheta} solved by $\mathbf H_{i+1/2}-S_{\theta_i}\mathbf H_i$ has the form
\begin{equation}\label{system-Y}
\partial_t Y+\sum_{j=1}^2\mathcal D_j(b)\partial_j Y +\mathcal Q(b)Y=F_H^{i+1/2},
\end{equation}
for $Y=\mathbf H_{i+1/2}-S_{\theta_i}\mathbf H_i$, $b=({\mathbf u}^a+\mathbf{u}_{i+1/2},\Psi^a+\Psi_{i+1/2})$, and where $\mathcal D_j$ and $\mathcal Q$ are some matrices. The matrices $\mathcal D_j$ are diagonal and, more important, $\mathcal D_1$ vanishes at the boundary. This yelds that system \eqref{system-Y} does not need any boundary condition. A standard energy argument applied to \eqref{system-Y} gives in view of \eqref{stima-FH}
\begin{equation}\label{stima-H-mod}
||\mathbf H_{i+1/2}-S_{\theta_i}\mathbf H_i||_{s,\ast,T} \le C||F_H^{i+1/2}||_{s,\ast,T}\le
C\delta\theta_i^{s+4-\alpha}\,,\quad\mbox{for}\,\,\,s\in\{6,\dots,\tilde{\alpha}+4\}\,.
\end{equation}
Collecting \eqref{stima_errore_u_i+1/2}, \eqref{stima-H-mod} gives \eqref{mo32}.

Using \eqref{app}, \eqref{estimate8},  \eqref{mo32} for $s=6$, and taking $\delta>0$ sufficiently small we have that $\varphi^a+\psi_{i+\frac{1}{2}}$ and $\tilde{{\mathbf U}}^a+\mathbf{V}_{i+\frac12}$ are sufficiently small. Then, $\varphi^a+\psi_{i+\frac{1}{2}}$ satisfies \eqref{phi-piccola} and, recalling Remark \ref{dato_iniziale}, ${{\mathbf U}}^a+\mathbf{V}_{i+\frac12}$ satisfies \eqref{hy}, \eqref{stability1}, \eqref{space}, \eqref{assumption}. ${{\mathbf H}}^a+\mathbf{H}_{i+\frac12}$ satifies \eqref{c} by construction, see \eqref{nonl_H}; the initial value at $t=0$ of ${{\mathbf H}}^a+\mathbf{H}_{i+\frac12}$ satisfies \eqref{c2} since ${{\mathbf H}}^a_{|t=0}={\mathbf H}_0$ satisfies \eqref{c2} by assumption, and ${{\mathbf H}}_{i+\frac12}=0$ for $t\leq 0$ by continuity. In conclusion, ${{\mathbf U}}^a+\mathbf{V}_{i+\frac12}$ satisfies all the constraints \eqref{hy}--\eqref{c}, \eqref{c2} and \eqref{assumption} for the background state.
\end{proof}

\subsubsection{Estimate of the second substitution errors}\label{second3}
In the following Lemma, we can estimate the second substitution errors $e'''_k,\tilde{e}'''_k$ of the iterative scheme. We define
\begin{equation}\label{e9}
e'''_k:=\mathcal{L}'(S_{\theta_k}{\mathbf V}_k,S_{\theta_k}\Psi_k)(\delta {\mathbf V}_k,\delta\Psi_k)-\mathcal{L}'({\mathbf V}_{k+\frac{1}{2}},\Psi_{k+\frac{1}{2}})(\delta {\mathbf V}_k,\delta\Psi_k),
\end{equation}
\begin{equation}\label{e11}
\begin{split}
\tilde{e}'''_k:=&\mathcal{B}'(S_{\theta_k}{\mathbf V}_k|_{x_1=0},S_{\theta_k}\psi_k)((\delta {\mathbf V}_k)|_{x_1=0},\delta\psi_k)\\
&-\mathcal{B}'({\mathbf V}_{k+\frac{1}{2}}|_{x_1=0},\psi_{k+\frac{1}{2}})((\delta {\mathbf V}_k)|_{x_1=0},\delta\psi_k).\\
\end{split}
\end{equation}
We can write \eqref{e9} and \eqref{e11} as follows:
\begin{equation}\nonumber
\begin{split}
e'''_k&=\int^1_0\mathbb{L}''(\mathbf{U}^a+{\mathbf V}_{k+\frac{1}{2}}+\tau(S_{\theta_k}{\mathbf V}_k-{\mathbf V}_{k+\frac{1}{2}}),\Psi^a+S_{\theta_k}\Psi_k)\\
&\quad((\delta {\mathbf V}_k,\delta\Psi_k),(S_{\theta_k}{\mathbf V}_k-{\mathbf V}_{k+\frac{1}{2}},0))d\tau,\\
&\tilde{e}'''_k=\mathbb{B}''((\delta {\mathbf V}_k|_{x_1=0},\delta\psi_k),((S_{\theta_k}{\mathbf V}_k-{\mathbf V}_{k+\frac{1}{2}})|_{x_1=0},0)).
\end{split}
\end{equation}

\begin{lemma}\label{seconde}
Let $\alpha\geq10$. There exist $\delta>0$ sufficiently small and $\theta_0\geq1$ sufficiently large such that, for all $k\in\{0,\cdots,i-1\}$ and for all integers $s\in\{6,\cdots,\tilde{\alpha}-2\},$ we have
\begin{equation}\label{eeee}
||e'''_k||_{s,\ast,T}\lesssim\delta^2\theta^{L_3(s)-1}_k\Delta_k,
\end{equation}
\begin{equation}\label{eeeee}
||\tilde{e}'''_k||_{H^s(\Gamma_T)}\lesssim\delta^2\theta^{L_3(s)-1}_k\Delta_k,
\end{equation}
where $L_3(s):=\max\{(s+2-\alpha)_++16-2\alpha;s+12-2\alpha\}.$
\end{lemma}
\begin{proof}
Using \eqref{mo32} and Lemma \ref{lb}, similar to the proof of Lemmata \ref{eestimate} and \ref{se}, we obtain \eqref{eeee} and \eqref{eeeee}.
Here, we can calculate the explicit form of $\tilde{e}'''_k$ as
\begin{equation}
\tilde{e}'''_k:=\left[
\begin{array}{c}
0\\
0\\
\delta H^+_k\cdot (S_{\theta_k}H^+_k-H^+_{k+\frac{1}{2}})-\delta H^-_k\cdot (S_{\theta_k}H^-_k-H^-_{k+\frac{1}{2}})
\end{array}\right]\,.
\end{equation}
\end{proof}

\subsubsection{Estimate of the last error term}\label{last}
We now estimate the last error term \eqref{D}:
\begin{equation}\label{D2}
D_{k+\frac{1}{2}}\delta\Psi_k=\frac{\delta\Psi_k}{\partial_1(\Phi^a+\Psi_{k+\frac{1}{2}})}R_k,
\end{equation}
where $R_k:=\partial_1[\mathbb{L}({\mathbf U}^a+{\mathbf V}_{k+\frac{1}{2}},\Psi^a+\Psi_{k+\frac{1}{2}})].$
It is noted that $$|\partial_1(\Phi^a+\Psi_{k+\frac{1}{2}})|=|\pm1+\partial_1(\Psi^a+\Psi_{k+\frac{1}{2}})|\geq\frac{1}{2},$$
for $\delta>0$ sufficiently small.
\newline
The following Lemma \ref{DD} can be proved by direct calculations.
\begin{lemma}\label{DD}
 Let $\alpha\geq10,\tilde{\alpha}\geq\alpha-4.$ There exist $\delta>0$ sufficiently small and $\theta_0\geq1$ sufficiently large,  such that,  for all $k\in\{0,\cdots,i-1\}$ and for all integers $s\in \{6,\cdots,\tilde{\alpha}-2\}$, we have
\begin{equation}\label{D_k}
||D_{k+\frac{1}{2}}\delta\Psi_k||_{s,\ast,T}\lesssim\delta^2\theta^{L_4(s)-1}_k\Delta_k,
\end{equation}
where $L_4(s):=s+14-2\alpha$.
\end{lemma}
\begin{proof}
Using \eqref{moser4} and \eqref{sobolev1}, we obtain that
\begin{equation}\label{DDD}
\begin{split}
||D_{k+\frac{1}{2}}\delta\Psi_k||_{s,\ast,T}&\lesssim||\delta\Psi_k||_{4,\ast,T}||R_k||_{s,\ast,T}+||\delta\Psi_k||_{s,\ast,T}||R_k||_{4,\ast,T}\\
&\quad +||\delta\Psi_k||_{4,\ast,T}||R_k||_{4,\ast,T}||\Psi^a+\Psi_{k+\frac{1}{2}}||_{s+2,\ast,T}.
\end{split}
\end{equation}
Using \eqref{f} and \eqref{system}, we can write
\begin{equation}\label{RK}
R_k=\partial_1(\mathcal{L}({\mathbf V}_k,\Psi_k)-\mathcal{F}^a+\mathcal{I}), \text{ for }t>0,
\end{equation}
where
\begin{equation}\nonumber
\mathcal{I}:=\int^1_0\mathbb{L}'(\mathbf{U}^a+{\mathbf V}_k+\tau({\mathbf V}_{k+\frac{1}{2}}-{\mathbf V}_k),\Psi^a+\Psi_k+\tau(\Psi_{k+\frac{1}{2}}-\Psi_k))({\mathbf V}_{k+\frac{1}{2}}-{\mathbf V}_k,\Psi_{k+\frac{1}{2}}-\Psi_k)d\tau.
\end{equation}
If $s\in\{4,\cdots,\tilde{\alpha}-4\},$ then using Hypothesis $(H_{i-1}),$ we obtain that
\begin{equation}\label{111}
||\mathcal{L}({\mathbf V}_k,\Psi_k)-\mathcal{F}^a||_{s+2,\ast,T}\leq 2\delta \theta^{s+1-\alpha}_k.
\end{equation}
Using \eqref{estimate6},\eqref{estimate7},\eqref{mo1}-\eqref{mo32}, \eqref{RK},\eqref{111}, Sobolev inequalities \eqref{sobolev2} and Moser-type calculus inequalities in Lemma \ref{moser2},
we have
\begin{equation}\label{1122}
||R_k||_{s,\ast,T}\lesssim \delta(\theta^{s+8-\alpha}_k+\theta^{(s+4-\alpha)_++10-\alpha}_k), \text{ for } s\in \{4,\cdots, \tilde{\alpha}-4\}.
\end{equation}
If $s\in\{\tilde{\alpha}-3,\tilde{\alpha}-2\}$, then, for $\tilde\alpha\ge\alpha-4$, using \eqref{estimate6}, \eqref{estimate8} and \eqref{mo1}-\eqref{mo32}, we can deduce directly from \eqref{RK} that
$$||R_k||_{s,\ast,T}\lesssim \delta \theta^{s+8-\alpha}_k.$$
Therefore, \eqref{1122} holds for all $s\in\{4,\cdots,\tilde{\alpha}-2\}.$ Using Hypothesis $(H_{i-1}),$ \eqref{1122}, \eqref{estimate8}, \eqref{estimate6} and \eqref{mo1}-\eqref{mo32} into \eqref{DDD}, we can obtain \eqref{D_k}.
\end{proof}

Using Lemmata \ref{eestimate} - \ref{DD}, we can conclude the following estimates of the error terms $e_k, \tilde{e}_k,$ defined by \eqref{e}.
\begin{lemma}\label{error5}
Let $\alpha\geq 10,\tilde{\alpha}\geq\alpha-4.$ There exist $\delta>0$ sufficiently small and $\theta_0\geq1$ sufficiently large, such that for all $k\in\{0,\cdots,i-1\}$ and for all integers $s\in \{6,\cdots,\tilde{\alpha}-2\}$, we have
\begin{equation}\label{ek}
||e_k||_{s,\ast,T}+||\tilde{e}_k||_{H^s(\Gamma_T)}\lesssim\delta^2\theta^{L_4(s)-1}_k\Delta_k,
\end{equation}
where $L_4(s)$ is defined in Lemma \ref{DD}.
\end{lemma}

From Lemma \ref{error5}, we obtain the estimate of the accumulated errors $E_i,\tilde{E}_i,\hat{E}_i,$ which are defined in \eqref{ae}.
\begin{lemma}\label{Eerror}
Let $\alpha\geq16,\tilde{\alpha}=\alpha+5.$ There exist $\delta>0$ sufficiently small and $\theta_0\geq1$ sufficiently large, such that
\begin{equation}\label{Eerrore}
||E_i||_{\alpha+3,\ast,T}+||\tilde{E}_i||_{H^{\alpha+3}(\Gamma_T)}\lesssim\delta^2\theta_i.
\end{equation}
\end{lemma}
\begin{proof}
Using $L_4(\alpha+3)\leq1$ if $\alpha\geq16$, it follows from \eqref{ek} that
$$||E_i||_{\alpha+3,\ast,T}\lesssim\sum^{i-1}_{k=0}||e_k||_{\alpha+3,\ast,T}\lesssim\sum^{i-1}_{k=0}\delta^2\Delta_k\lesssim\delta^2\theta_i\,,$$ if $\alpha+3\leq\tilde{\alpha}-2$. Similar arguments also hold for $||\tilde{E}_i||_{H^{\alpha+3}(\Gamma_T)}$. The minimal possible $\tilde{\alpha}$ is $\alpha+5$.
\end{proof}
\subsection{Convergence of the iteration scheme }\label{convergence}
We still need to estimate the source terms $f_{i},g_{i}.$
\begin{lemma}\label{serror}
Let $\alpha\geq16$ and $\tilde{\alpha}=\alpha+5$. There exist $\delta>0$ sufficiently small and $\theta_0\geq1$ sufficiently large, such that for all integers $s\in\{6,\cdots,\tilde{\alpha}+2\},$
\begin{equation}\label{ferrore}
||f_i||_{s,\ast,T}\lesssim\Delta_i\Big(\theta^{s-\alpha-3}_i(||\mathcal{F}^a||_{\alpha+2,\ast,T}+\delta^2)+\delta^2\theta^{L_4(s)-1}_i\Big),
\end{equation}
\begin{equation}\label{gerrore}
||g_i||_{H^{s}(\Gamma_T)}\lesssim\delta^2\Delta_i\Big(\theta^{s-\alpha-3}_i+\theta^{L_4(s)-1}_i\Big)\,.
\end{equation}
\end{lemma}
In the above inequalities we need the exponent $s-\alpha-3$ of $\theta_i$ to compensate the loss of 2 derivatives for the data in \eqref{tamesss1}, in order to recover the exponent $s-\alpha-1$ in the corresponding terms of \eqref{tamesss4}.
\begin{proof}
As in \cite{Coulombel2008, S-T plasma vuoto, Trakhinin2009}, using \eqref{as1}-\eqref{as3}, \eqref{fg}, \eqref{ek}, \eqref{Eerrore}, we obtain that
\begin{equation}\nonumber
\begin{split}
||f_i||_{s,\ast,T}&\leq||(S_{\theta_i}-S_{\theta_{i-1}})\mathcal{F}^a-(S_{\theta_i}-S_{\theta_{i-1}})E_{i-1}-S_{\theta_i}e_{i-1}||_{s,\ast,T}\\
&\lesssim \Delta_i\theta^{s-\alpha-3}_i(||\mathcal{F}^a||_{\alpha+2,\ast,T}+\theta^{-1}_i||E_{i-1}||_{\alpha+3,\ast,T})+||S_{\theta_i}e_{i-1}||_{s,\ast,T}\\
&\lesssim \Delta_i\{\theta^{s-\alpha-3}_i(||\mathcal{F}^a||_{\alpha+2,\ast,T}+\delta^2)+\delta^2\theta^{L_4(s)-1}_i\}.
\end{split}
\end{equation}
Using \eqref{ek} and \eqref{Eerrore} we can obtain \eqref{gerrore}.
\end{proof}

Similar to the proof of \cite{Coulombel2008, S-T plasma vuoto, Trakhinin2009}, we can obtain the estimate of $(\delta{\mathbf V}_i,\delta\Psi_i)$ by \eqref{mo32} and the tame estimate \eqref{tamesss} applied to problem \eqref{effective3}.
\begin{lemma}\label{delta}
Let $\alpha\geq16$ and $\tilde{\alpha}=\alpha+5$. If $\delta>0$ and $||\mathcal{F}^a||_{\alpha+2,\ast,T}/\delta$ are sufficiently small and $\theta_0\geq1$ is sufficiently large, then for all integers $s\in\{6,\cdots,\tilde{\alpha}\}$,
 \begin{equation}\label{d}
||(\delta {\mathbf V}_i,\delta\Psi_i)||_{s,\ast,T}+||\delta\psi_i||_{H^{s}(\Gamma_T)}\leq \delta\theta^{s-\alpha-1}_i\Delta_i.
\end{equation}
\end{lemma}
\begin{proof}
Let us consider problem \eqref{effective3} that will be solved by applying Theorem \ref{keytame}. We first notice that $({\mathbf U}^a+{\mathbf V}_{i+\frac{1}{2}}, \Psi^a+\Psi_{i+\frac{1}{2}}, \varphi^a+\psi_{i+\frac{1}{2}})$ satisfy the constraints \eqref{hy}--\eqref{c}, \eqref{c2}, \eqref{assumption}. Thus we may apply our tame estimate \eqref{tamesss} and obtain
	\begin{equation}\label{tamesss1}
		\begin{split}
			&||\delta\dot{{\mathbf V}_i}||_{s,\ast,T}+||\delta\psi_i||_{H^s(\Gamma_T)}\\
			&\leq C\Big(||{f}_i||_{s+2,\ast,T}+||{g}_i||_{H^{s+2}(\Gamma_T)}\\
			&\quad+(||{f}_i||_{8,\ast,T}+||{g}_i||_{H^{8}(\Gamma_T)})||(\tilde{\mathbf U}^a+{\mathbf V}_{i+\frac{1}{2}}, \nabla(\Psi^a+\Psi_{i+\frac{1}{2}}))||_{s+4,\ast,T}\Big).
		\end{split}
	\end{equation}
	On the other hand, from \eqref{goodunkown} it follows
	\begin{equation}\label{good2}
		\begin{split}
			&||\delta{{\mathbf V}_i}||_{s,\ast,T}
			\leq ||\delta\dot{{\mathbf V}_i}||_{s,\ast,T}+
			C||\delta{{\mathbf \Psi}_i}||_{s,\ast,T}||(\tilde{\mathbf U}^a+{\mathbf V}_{i+\frac{1}{2}}, \Psi^a+\Psi_{i+\frac{1}{2}})||_{6,\ast,T}\\
			&\quad\qquad+
		C||\delta{{\mathbf \Psi}_i}||_{4,\ast,T}||(\tilde{\mathbf U}^a+{\mathbf V}_{i+\frac{1}{2}}, \Psi^a+\Psi_{i+\frac{1}{2}})||_{s+2,\ast,T}.
		\end{split}
	\end{equation}
	From
	\begin{equation}\label{Psipsi}
		||\delta{{\mathbf \Psi}_i}||_{s,\ast,T}\le ||\delta\psi_i||_{H^s(\Gamma_T)},
	\end{equation}
\eqref{tamesss1} for $s=4$ and \eqref{assumption} we have
\begin{equation*}
	\begin{split}
		&||\delta{{\mathbf \Psi}_i}||_{4,\ast,T}\le ||\delta\psi_i||_{H^4(\Gamma_T)}
		\leq C\Big(||{f}_i||_{6,\ast,T}+||{g}_i||_{H^{6}(\Gamma_T)}\\
		&\quad+(||{f}_i||_{8,\ast,T}+||{g}_i||_{H^{8}(\Gamma_T)})||(\tilde{\mathbf U}^a+{\mathbf V}_{i+\frac{1}{2}}, \nabla(\Psi^a+\Psi_{i+\frac{1}{2}}))||_{8,\ast,T}\Big)\\
		&	\leq C
		(||{f}_i||_{8,\ast,T}+||{g}_i||_{H^{8}(\Gamma_T)}).
	\end{split}
\end{equation*}
Then, from \eqref{tamesss1} and \eqref{good2} we obtain
	\begin{equation*}
	\begin{split}
		&||\delta{{\mathbf V}_i}||_{s,\ast,T}+||\delta\psi_i||_{H^s(\Gamma_T)}\\
		&\leq C\Big(||{f}_i||_{s+2,\ast,T}+||{g}_i||_{H^{s+2}(\Gamma_T)}\\
		&\quad+(||{f}_i||_{8,\ast,T}+||{g}_i||_{H^{8}(\Gamma_T)})||(\tilde{\mathbf U}^a+{\mathbf V}_{i+\frac{1}{2}}, \nabla(\Psi^a+\Psi_{i+\frac{1}{2}}))||_{s+4,\ast,T}\Big)\\
		&\quad+||\delta \psi_i||_{H^s(\Gamma_T)}||(\tilde{\mathbf U}^a+{\mathbf V}_{i+\frac{1}{2}}, \Psi^a+\Psi_{i+\frac{1}{2}})||_{6,\ast,T}.
	\end{split}
\end{equation*}
Using \eqref{app}, \eqref{estimate8} and  \eqref{mo32} for $s=6$, taking $\delta>0$ sufficiently small, we can absorb the last term in the right-hand side above into the left-hand side to get
	\begin{equation}\label{tamesss3}
	\begin{split}
		&||\delta{{\mathbf V}_i}||_{s,\ast,T}+||\delta\psi_i||_{H^s(\Gamma_T)}\\
		&\leq  C\Big(||{f}_i||_{s+2,\ast,T}+||{g}_i||_{H^{s+2}(\Gamma_T)}\\
		&\quad+(||{f}_i||_{8,\ast,T}+||{g}_i||_{H^{8}(\Gamma_T)})||(\tilde{\mathbf U}^a+{\mathbf V}_{i+\frac{1}{2}}, \nabla(\Psi^a+\Psi_{i+\frac{1}{2}}))||_{s+4,\ast,T}\Big).
	\end{split}
\end{equation}
The remaining part of the work is to estimate the right-hand side of \eqref{tamesss3}.
\newline
Using Lemma \ref{serror}, \eqref{estimate8}, Lemma \ref{mo} and \eqref{Psipsi}, \eqref{tamesss3} becomes
	\begin{equation}\label{tamesss4}
	\begin{split}
		&||(\delta{{\mathbf V}_i},\delta{{\mathbf \Psi}_i})||_{s,\ast,T}+||\delta\psi_i||_{H^s(\Gamma_T)}\\
		&\leq  C\Delta_i\Big(\theta_i^{s-\alpha-1}(||\mathcal{F}^a||_{\alpha+2,\ast,T}+\delta^2)+\delta^2\theta_i^{L_4(s+2)-1}\Big)\\
		&\quad+C\Delta_i\Big(\theta_i^{5-\alpha}(||\mathcal{F}^a||_{\alpha+2,\ast,T}+\delta^2)+\delta^2\theta_i^{21-2\alpha}\Big)\Big(\delta\theta_i^{s+8-\alpha}+\delta\theta_i^{(s+6-\alpha)_+}\Big).
	\end{split}
\end{equation}
One checks that, for $\alpha\ge16$ and $6\le s\le\tilde{\alpha}$, the following inequalities hold true:
\begin{equation*}
	\begin{split}
	&L_4(s+2)\le s-\alpha,\qquad 5-\alpha+(s+6-\alpha)_+\le s-\alpha-1,
\\
&
s+13-2\alpha\le s-\alpha-1,\qquad
21-2\alpha+s+8-\alpha\le s-\alpha-1,\\
&21-2\alpha+(s+6-\alpha)_+\le s-\alpha-1.
	\end{split}
\end{equation*}
From \eqref{tamesss4} we thus obtain \eqref{d}, provided $\delta>0$ and $||\mathcal{F}^a||_{\alpha+2,\ast,T}/\delta$ are sufficiently small and $\theta_0\geq1$ is sufficiently large.	
\end{proof}

Finally, similar to the proof of \cite{Coulombel2008, S-T plasma vuoto, Trakhinin2009}, we can obtain the remaining inequalities in $(H_i).$
\begin{lemma}\label{Hn}
Let $\alpha\geq16$. If $\delta>0$ and $||\mathcal{F}^a||_{\alpha+2,\ast,T}/\delta$ are sufficiently small and if $\theta_0\geq1$ is sufficiently large, then for all integers $s\in\{6,\cdots,\tilde{\alpha}-2\}$
 \begin{equation}\label{Le}
||\mathcal{L}({\mathbf V}_i,\Psi_i)-\mathcal{F}^a||_{s,\ast,T}\leq 2\delta\theta^{s-\alpha-1}_i.
\end{equation}
Moreover, for all integers $s\in\{6,\cdots,\tilde{\alpha}-2\}$
 \begin{equation}\label{B22}
||\mathcal{B}({\mathbf V}_i|_{x_1=0},\Psi_i)||_{H^{s}(\Gamma_T)}\leq \delta\theta^{s-\alpha-1}_i.
\end{equation}
\end{lemma}
\begin{proof}
Recall that, by summing the relations \eqref{l}, we have
\begin{equation}
\mathcal{L}({\mathbf V}_i,\Psi_i)-\mathcal{F}^a=(S_{\theta_{i-1}}-I)\mathcal{F}^a+(I-S_{\theta_{i-1}})E_{i-1}+e_{i-1}.
\end{equation}
The proof of \eqref{Le} then follows by applying \eqref{as2}, \eqref{ek}, \eqref{Eerrore}, provided that $\delta>0$ and $||\mathcal{F}^a||_{\alpha+2,\ast,T}/\delta$ are sufficiently small  and $\theta_0\geq1$ is sufficiently large. The proof of \eqref{B22} is similar.
\end{proof}
%\smallskip

We are now in the position to prove the main theorem for the existence of the solution to the nonlinear problem \eqref{IVP}.
\medskip

\noindent
{\it Proof of Theorem \ref{maintheorem}}: \;
Let the initial data $({\mathbf U}^{\pm}_0,\varphi_0)$ satisfy all the assumptions of Theorem \ref{maintheorem}. Let $\alpha=m+1\ge 16$, $\tilde{\alpha}=\alpha+5$, $\mu=m+10$. Then the initial data ${\mathbf U}^{\pm}_0\in H^{\mu+1.5}(\mathbb R^2_+)$ and $\varphi_0\in H^{\mu+1.5}(\mathbb R)$ are compatible up to order $\mu$ and there exists an approximate solution  $(\tilde{\mathbf U}^a,\varphi^a)\in H^{\mu+2}(\Omega_T)\times H^{\mu+2}(\Gamma_T)$ to problem \eqref{IVP}. Observe that $\mu+2=\tilde\alpha+6$ as required in \eqref{small}. We choose $\delta>0$, $T>0$ sufficiently small, $\theta_0\ge 1$ sufficiently large as in the previous lemmata. We also assume $T>0$ small enough so that $||\mathcal{F}^a||_{\alpha+2,\ast,T}/\delta$ is sufficiently small. Then in view of Lemmata \ref{H0}, \ref{delta}, \ref{Hn}, property ($H_i$) holds for all integers $i$. In particular, we have
$$\sum^{\infty}_{i=0}(||(\delta {\mathbf V}_i,\delta\Psi_i)||_{s,\ast,T}+||\delta\psi_i||_{H^{s}(\Gamma_T)})\lesssim\sum^{\infty}_{i=0}\theta^{s-\alpha-2}_i<\infty,\quad\text{for}\,\,\,6\leq s\leq\alpha-1.$$
Thus, the sequence $({\mathbf V}_i,\Psi_i)$ converges to some limit $({\mathbf V},\Psi)$ in $H^{\alpha-1}_{\ast}(\Omega_T),$ and sequence $\psi_i$ converges to some limit $\psi$ in $H^{\alpha-1}(\Gamma_T).$ Passing to the limit in \eqref{Le} and \eqref{B22} for $s=m=\alpha-1$, we obtain \eqref{system}. Therefore, $({\mathbf U},\Phi)=({\mathbf U}^a+{\mathbf V},\Phi^a+\Psi)$ is a solution on $\Omega_T$ of nonlinear system \eqref{IVP}.

\bigskip

\appendix
\section{Trace theorem in anisotropic space}\label{tracetheorem}
In this Appendix let us recall the following trace theorem in the anisotropic space by Ohno, Shizuta, Yanagisawa \cite{Ohno-Shizuta1994}.
\begin{theorem}[\cite{Ohno-Shizuta1994}]\label{teoA1}
Let $s\ge 2$ be an integer. Then the mapping
$$
C_0^\infty(\overline{\mathbb{R}^2_+})\ni u\mapsto\{\partial^j_1u|_{x_1=0}\,,\,\,j=0,\dots[s/2]-1\}\in \underbrace{C^\infty_0(\mathbb R)\times\dots C^\infty_0(\mathbb R)}_{[s/2]\,\,\,\mbox{times}}
$$
extends by continuity to a continuous linear mapping of
$$
H^s_\ast(\mathbb R^2_+)\rightarrow\prod\limits_{j=0}^{[s/2]-1}H^{s-2j-1}(\mathbb R)\,.
$$
This mapping is surjective and there exists a continuous linear right-inverse
$$
(v_0,\dots,v_{[s/2]-1})\mapsto\mathcal R(v_0,\dots,v_{[s/2]-1})
$$
of
$$
\prod\limits_{j=0}^{[s/2]-1}H^{s-2j-1}(\mathbb R)\rightarrow H^s_\ast(\mathbb R^2_+)\,
$$
such that
$$
\partial^j_1(\mathcal R(v_0,\dots,v_{[s/2]-1}))|_{x_1=0}=v_j\,,\quad j=0,\dots,[s/2]-1\,.
$$
\end{theorem}

\section{Proof of Theorem \ref{th-wp-hom}}\label{proof-Thm-5-1}
Here we focus only on the proof of the energy estimate \eqref{est-wp-hom}; the existence and uniqueness of the solution can be shown by standard methods.

Let ${\mathbf V}={\mathbf V}(t,{\bf x})=({\mathbf V}^+,{\mathbf V}^-)$, $\varphi=\varphi(t,x_2)$, where ${\mathbf V}^\pm:=(\dot q^\pm, \dot u^\pm_n, \dot u_2^\pm, \dot H^\pm_n, \dot H^\pm_2, \dot S^\pm)$, be sufficiently smooth vector fields respectively on $\Omega_T$ and $\Gamma_T$, satisfying the linear system \eqref{A}, or its equivalent form \eqref{BBB}, together with the boundary and the ``initial" conditions from \eqref{ho2}, that is
\begin{equation}\label{BC}
\partial_t\varphi+\hat{u}^\pm_2\partial_2\varphi-\dot{u}^\pm_N\mp\varphi\partial_1\hat{u}^\pm_N=0\,,\quad \dot{q}^+-\dot{q}^-+\varphi[\partial_1\hat{q}]=0 \quad\mbox{on}\,\,\,\Gamma_T\,,
\end{equation}
\begin{equation}\label{IC}
({\mathbf V}, \varphi)=(0,0)\quad\mbox{for}\,\,\,t<0\,.
\end{equation}
Recall from Section \ref{linearized} that $\hat\Psi^\pm=\hat\Psi^\pm(t,{\bf x})$ are defined through the basic front function $\hat\varphi=\hat\varphi(t,x_2)$ by
\[
\hat\Psi^\pm(t,{\bf x}):=\chi(\pm x_1)\hat\varphi(t,x_2)\,,\quad\forall\,{\bf x}\in\mathbb R^2_+\,\,\,t\in(-\infty,T]\,,
\]
where $\chi\in C^{\infty}_0(\R)$ satisfies $\chi\equiv 1$ on $[-1,1]$, so that $\partial_1\hat\Psi^\pm\vert_{x_1=0}=0$. We set
\[
\dot u^\pm_n:=\dot u_1^\pm-\partial_2\hat\Psi^\pm\dot u^\pm_2\,,\quad\dot H^\pm_n:=\dot H_1^\pm-\partial_2\hat\Psi^\pm\dot H^\pm_2\,;
\]
notice in particular
\[
\dot u^\pm_n\equiv \dot u^\pm_N:=\dot u_1^\pm-\partial_2\hat\varphi\dot u^\pm_2\,,\quad\dot H^\pm_n\equiv \dot H^\pm_N:=\dot H_1^\pm-\partial_2\hat\varphi\dot H^\pm_2\,,\quad\mbox{on}\,\,\,\Gamma_T\,.
\]
Recall that $({\mathbf V},\varphi)$ as above must also satisfy the interior and boundary constraints
\begin{equation}\label{I&Bconst}
\mathrm{div}\dot{\mathbf h}^{\pm}=0 \text{ in }\Omega_T\quad\mbox{and}\quad\hat{H}^{\pm}_2\partial_2\varphi-\dot{H}^{\pm}_{N}\mp\varphi\partial_1\hat{H}^{\pm}_{N}=0 \text{ on }\Gamma_T\,,
\end{equation}
where
$$
\dot {\mathbf h}^\pm:=(\dot H^\pm_n,\partial_1\hat\Phi^\pm\dot H^\pm_2)\quad\mbox{and}\quad\hat\Phi^\pm(t,{\bf x}):=\pm x_1+\hat\Psi^\pm(t,{\bf x})\,.
$$
In the forthcoming calculations we will make use of the following shortcut notation
\begin{eqnarray}
I(t):=\Vert{\mathbf V}^+(t)\Vert_{L^2(\mathbb R^2_+)}^2+\Vert{\mathbf V}^-(t)\Vert_{L^2(\mathbb R^2_+)}^2\,,\label{I0}\\
I_{0}(t):=\Vert\partial_t{\mathbf V}^+(t)\Vert_{L^2(\mathbb R^2_+)}^2+\Vert\partial_t{\mathbf V}^-(t)\Vert_{L^2(\mathbb R^2_+)}^2\,,\label{It}
\\
\begin{split}
I_{1,n}(t):=\Vert\partial_1{\mathbf V}^+_n(t)\Vert_{L^2(\mathbb R^2_+)}^2+\Vert\partial_1{\mathbf V}^-_n(t)\Vert_{L^2(\mathbb R^2_+)}^2,
\end{split}\label{I1n}\\
I_{\sigma}(t):=\Vert\sigma\partial_1{\mathbf V}^+(t)\Vert_{L^2(\mathbb R^2_+)}^2+\Vert\sigma\partial_1{\mathbf V}^-(t)\Vert_{L^2(\mathbb R^2_+)}^2\,,\label{Isigma}\\
I_{2}(t):=\Vert\partial_2{\mathbf V}^+(t)\Vert_{L^2(\mathbb R^2_+)}^2+\Vert\partial_2{\mathbf V}^-(t)\Vert_{L^2(\mathbb R^2_+)}^2\,,\label{I2}
\end{eqnarray}
and we set
\begin{equation}\label{I}
I_{1,\ast}(t):=I(t)+I_0(t)+I_{\sigma}(t)+I_{2}(t)\,.
\end{equation}
In \eqref{I1n}, ${\mathbf V}^\pm_n$ denote the noncharacteristic part of the solution ${\mathbf V}^\pm$, that is
\[
{\mathbf V}^\pm_n=(\dot q^\pm,\dot u^\pm_n,\dot H^\pm_n)\,.
\]
For all integers $m\ge 0$, we will also write $C_m$ to denote a generic positive constant depending (nonlinearly) on $\Vert\hat{\mathbf U}^\pm\Vert_{W^{m,\infty}(\Omega_T)}$ and $\Vert\nabla_{t,x_2}\hat\varphi\Vert_{W^{m,\infty}(\Gamma_T)}$  and the positive number $k$ from \eqref{hy} and \eqref{stability1}, that is
\[
C_m=C_m(\Vert\hat{\mathbf U}^\pm\Vert_{W^{m,\infty}(\Omega_T)},\Vert\nabla_{t,x_2}\hat\varphi\Vert_{W^{m,\infty}(\Gamma_T)},k)\,,
\]
that might possibly be different from an occurrence to another even within the same sequence of inequalities.

\smallskip
\noindent
From the notation above and in view of \eqref{IC}, it straightforwardly turns out that
\begin{equation}\label{equiv}
\begin{split}
& \frac1{C_1}\Vert\dot{\mathbf U}(t)\Vert_{H^1_\ast(\mathbb R^2_+)}^2\le I(t)+I_\sigma(t)+I_2(t)\le C_1\Vert\dot{\mathbf U}(t)\Vert_{H^1_\ast(\mathbb R^2_+)}^2\,,\\
& \frac1{C_1}|||\dot{\mathbf U}(t)|||_{1,\ast}^2\le I_{1,\ast}(t)\le C_1|||\dot{\mathbf U}(t)|||_{1,\ast}^2\,,\\
& \frac1{C_1}\Vert\dot{\mathbf U}\Vert_{1,\ast,t}^2\le\int_0^t I_{1,\ast}(s)ds\le C_1\Vert\dot{\mathbf U}\Vert_{1,\ast,t}^2\,,
\end{split}
\end{equation}
for all $t\in(0,T]$, see \eqref{norma-equiv}, \eqref{norma-int} and \eqref{V}.

\smallskip
\noindent
The estimate of Theorem \ref{th-wp-hom} will be obtained by applying classical arguments from the energy method in order to get a control of the $L^2$-norm of the solution ${\mathbf V}^\pm$ and the front $\varphi$, as well as its tangential space time derivatives, corresponding to the different expressions listed in \eqref{I0}--\eqref{I2}.
%\begin{remark}\label{entropia}
%The linearized equation for the entropy $\dot S^\pm$ is an evolution-like equation because the coefficient of the normal derivative of the entropy vanishes on the boundary; this yields that no boundary conditions are needed to be coupled to the equation in order to derive an apriori energy estimate for it. So to estimate $\dot S^\pm$, we just handle the equation of $\dot S^\pm$ alone by the standard energy method tools.
%\end{remark}
%%%%%%%%%%%%%%%%%%%%%%%%%%%%%%%%%%%%%%%%%%%%%%%%%%%%%%%%%%%%%%%%%%%%%%%
\subsection{Estimate of $I(t)$}\label{I0_sct}
We scalarly multiply both sides of system \eqref{BBB} and integrate by parts in $\Omega_t$ to get the energy identity
\begin{equation}\label{en-id-I0}
\begin{split}
\int_{\mathbb R^2_+}(\mathcal B_0{\mathbf V}\cdot{\mathbf V})(t)d{\bf x}&-\int_{\Gamma_t}(\mathcal B_1{\mathbf V}\cdot{\mathbf V})\vert_{x_1=0}\,dx_2ds=2\int_{\Omega_t}\tilde{\mathcal F}\cdot{\mathbf V}d{\bf x}ds\\
&+\int_{\Omega_t}(\partial_t\mathcal B_0+\partial_1\mathcal B_1+\partial_2\mathcal B_2-\mathcal B_3){\mathbf V}\cdot{\mathbf V}d{\bf x}ds\,,
\end{split}
\end{equation}
where, to shorcut notation, we write hereafter $\mathcal{B}_k$ instead of $\mathcal{B}_k(\hat{{\mathbf U}},\hat{\Psi})$ for $k=0,1,2,3$.
\newline
From Cauchy-Schwarz and Young inequality, the right-hand side above is estimated by
\begin{equation}\label{source_0}
\begin{split}
&2\int_{\Omega_t}\tilde{\mathcal F}\cdot{\mathbf V}d{\bf x}ds
+\int_{\Omega_t}(\partial_t\mathcal B_0+\partial_1\mathcal B_1+\partial_2\mathcal B_2-\mathcal B_3){\mathbf V}\cdot{\mathbf V}d{\bf x}ds\\
&\le 2\Vert\tilde{\mathcal F}\Vert_{L^2(\Omega_t)}\Vert{\mathbf V}\Vert_{L^2(\Omega_t)}+\Vert\partial_t\mathcal B_0+\partial_1\mathcal B_1+\partial_2\mathcal B_2-\mathcal B_3\Vert_{L^\infty(\Omega_t)}\Vert{\mathbf V}\Vert_{L^2(\Omega_t)}^2\\
&\le\Vert\tilde{\mathcal F}\Vert_{L^2(\Omega_t)}^2+\Vert{\mathbf V}\Vert_{L^2(\Omega_t)}^2+\Vert\partial_t\mathcal B_0+\partial_1\mathcal B_1+\partial_2\mathcal B_2-\mathcal B_3\Vert_{L^\infty(\Omega_t)}\Vert{\mathbf V}\Vert_{L^2(\Omega_t)}^2\\
&\le\Vert\tilde{\mathcal F}\Vert_{L^2(\Omega_t)}^2+C_1\Vert{\mathbf V}\Vert_{L^2(\Omega_t)}^2\le C_0\Vert\tilde{\mathbf F}\Vert_{L^2(\Omega_t)}^2+C_1\int_0^t I(s)ds\\
&\le C_0\Vert{\mathbf F}\Vert_{L^2(\Omega_t)}^2+C_1\int_0^t I(s)ds\,,
\end{split}
\end{equation}
where
\begin{equation}\label{source}
\tilde{\mathcal F}=J^T\tilde{\mathbf F}\quad\mbox{and}\quad\tilde{\mathbf F}=\mathcal S(\hat{\mathbf U}){\mathbf F}
\end{equation}
are used, together with \eqref{I0}, to make the last inequality above.
\newline
Now we need to get an estimate of the quadratic form $(\mathcal B_1{\mathbf V}\cdot{\mathbf V})$ under the second boundary integral in the left-hand side of \eqref{en-id-I0}.

From \eqref{boundary-q-form-lambda}, a direct calculation gives \eqref{boundary-q-form_lambda_1}, \eqref{lot_1},
%\begin{equation}\label{boundary-q-form_lambda_1}
%(\mathcal B_1{\mathbf V}\cdot{\mathbf V})=2[\hat u_2-\hat\lambda\hat H_2]\dot q^+\partial_2\varphi+\mbox{l.o.t}\,,\quad\mbox{on}\,\,\,\{x_1=0\}\,,
%\end{equation}
%with
%\begin{equation}\label{lot_1}
%\begin{split}
%\mbox{l.o.t.}:=&-2[\partial_1\hat u_N-\hat\lambda\partial_1\hat H_N]
%\dot q^+\varphi-2[\partial_1\hat q]\varphi\partial_t\varphi-2[\partial_1\hat q](\hat u^-_2-\hat\lambda\hat H^-_2)\varphi\partial_2\varphi\\
%&-2[\partial_1\hat q](\partial_1\hat u^-_N-\hat\lambda^-\partial_1\hat H^-_N)\varphi^2\,,
%\end{split}
%\end{equation}
where we have made use of the boundary conditions in \eqref{BC} and \eqref{I&Bconst}. In \eqref{boundary-q-form_lambda_1} the initials ``l.o.t'' are used to mean ``lower order terms'' with respect to the {\em leading part} $[\hat u_2-\hat\lambda\hat H_2]\dot q^+\partial_2\varphi$ of the boundary quadratic form \eqref{boundary-q-form_lambda_1}. All terms appearing in \eqref{lot_1} are products of the form $\varphi\,\partial_i\varphi$, with $i=0,2$, or $v_{j}^{\pm}\vert_{x_1=0}\,\varphi$, with $j=1,\dots,6$, up to some coefficients. To make uniform notation, here we have set $\partial_0\equiv\partial_t$ and ${\mathbf V}^\pm:=(v_1^\pm,v^\pm_2,v^\pm_3,v^\pm_4,v^\pm_5,v^\pm_6)$.
\newline
As already shown in Sect. \ref{wellposedness}, the major advantage to settle the functions $\hat\lambda^\pm=\lambda(\hat{\mathbf U}^\pm)$ as prescribed in Lemma \ref{lemma-Paolo-Ale-Paola} is making the leading term in the boundary quadratic form $(\mathcal B_1{\mathbf V}\cdot{\mathbf V})\vert_{x_1=0}$  to be identically zero. Under this choice, the latter reduces indeed to
\begin{equation}\label{boundary-q-form_lambda_3}
\begin{split}
\frac{1}{2}(\mathcal B_1{\mathbf V}\cdot{\mathbf V})\vert_{x_1=0}&=\mbox{l.o.t.}=-[\partial_1\hat u_N-\hat\lambda\partial_1\hat H_N]
\dot q^+\varphi-[\partial_1\hat q]\varphi\partial_t\varphi\\
&\quad-[\partial_1\hat q](\hat u^-_2-\hat\lambda\hat H^-_2)\varphi\partial_2\varphi-[\partial_1\hat q](\partial_1\hat u^-_N-\hat\lambda^-\partial_1\hat H^-_N)\varphi^2\,.
\end{split}
\end{equation}
We now focus on the estimate of the boundary integral of the different quadratic terms in \eqref{boundary-q-form_lambda_3}. Because the explicit form of the coefficients involved in the different quadratic terms appearing in \eqref{boundary-q-form_lambda_3} is useless, hereafter we adopt the custom to denote as $\hat{c}=\hat{c}(t,x_2)$ a generic function on $\Gamma_T$ represented by some nonlinear smooth function of $\hat{\mathbf U}^\pm\vert_{x_1=0}$ and $\nabla_{t,x_2}\hat\varphi$ only, that may be possibly different from a line to another within the same formula. %We follow here the method of \cite{Trakhinin2005}
From \eqref{boundary-q-form_lambda_3}
\begin{equation*}
\begin{split}
\int_{\Gamma_t}(\mathcal B_1{\mathbf V}\cdot{\mathbf V})\vert_{x_1=0}\,dx_2ds&=\int_{0}^t\!\!\!\int_{\mathbb R}\!\!\hat c
\dot q^+\varphi\,dx_2ds+\int_{0}^t\!\!\!\int_{\mathbb R}\!\!\hat c\varphi\partial_t\varphi
\,dx_2ds\\
&\quad+\int_{0}^t\!\!\!\int_{\mathbb R}\!\!\hat c\varphi\partial_2\varphi\,dx_2ds+\int_{0}^t\!\!\!\int_{\mathbb R}\!\!\hat c\varphi^2\,dx_2ds
\end{split}
\end{equation*}
follows at once.
\newline
Hence Cauchy--Schwarz and Young's inequalities together with Leibniz's rule and integration by parts yields
\[
\begin{split}
&\left\vert\int_{0}^t\!\!\!\int_{\mathbb R}\!\!\hat c
\dot q^+\varphi\,dx_2ds\right\vert\le\Vert\hat c\Vert_{L^\infty(\Gamma_t)}\left\{\Vert\dot q^+\vert_{x_1=0}\Vert_{L^2(\Gamma_t)}^2+\Vert\varphi\Vert_{L^2(\Gamma_t)}^2\right\}\,;\\
\\
&\left\vert\int_{0}^t\!\!\!\int_{\mathbb R}\!\!\hat c\varphi\partial_t\varphi
\,dx_2ds\right\vert=\left\vert\frac12\int_{0}^t\!\!\!\int_{\mathbb R}\!\!\hat c\partial_t(\varphi^2)\,dx_2ds\right\vert=\left\vert\frac12\int_{0}^t\!\!\!\int_{\mathbb R}\!\!\partial_t(\hat c\varphi^2)\,dx_2ds-\frac12\int_{0}^t\!\!\!\int_{\mathbb R}\!\!\partial_t\hat c\varphi^2\,dx_2ds\right\vert\\
&\qquad=\left\vert\frac12\int_{\mathbb R}\!\!\hat c(t)\varphi^2(t)\,dx_2-\frac12\int_{0}^t\!\!\!\int_{\mathbb R}\!\!\partial_t\hat c\varphi^2\,dx_2ds\right\vert\\
&\qquad\le\frac12\Vert\hat c(t)\Vert_{L^\infty(\mathbb R)}\Vert\varphi(t)\Vert^2_{L^2(\mathbb R)}+\frac12\Vert\partial_t\hat c\Vert_{L^\infty(\Gamma_t)}\Vert\varphi\Vert^2_{L^2(\Gamma_t)}\,;\\
\\
&\left\vert\int_{0}^t\!\!\!\int_{\mathbb R}\!\!\hat c\varphi\partial_2\varphi\,dx_2ds\right\vert=\left\vert\frac12\int_{0}^t\!\!\!\int_{\mathbb R}\!\!\hat c\partial_2(\varphi^2)\,dx_2ds\right\vert=\left\vert-\frac12\int_{0}^t\!\!\!\int_{\mathbb R}\!\!\partial_2\hat c\varphi^2\,dx_2ds\right\vert\\
&\qquad\quad\le\frac12\Vert\partial_2\hat c\Vert_{L^\infty(\Gamma_t)}\Vert\varphi\Vert^2_{L^2(\Gamma_t)}\,;\\
\\
&\left\vert\int_{0}^t\!\!\!\int_{\mathbb R}\!\!\hat c\varphi^2\,dx_2ds\right\vert\le\Vert\hat c\Vert_{L^\infty(\Gamma_t)}\Vert\varphi\Vert^2_{L^2(\Gamma_t)}\,.
\end{split}
\]
Then we sum up the preceding estimates to get
\begin{equation}\label{stima_bdry_0}
\begin{split}
\left\vert\int_{\Gamma_t}(\mathcal B_1{\mathbf V}\cdot{\mathbf V})\vert_{x_1=0}\,dx_2ds\right\vert\le C_1\left\{\Vert \dot q^+\vert_{x_1=0}\Vert_{L^2(\Gamma_t)}^2+\Vert\varphi(t)\Vert_{L^2(\mathbb R)}^2\right\}+C_2\Vert\varphi\Vert^2_{L^2(\Gamma_t)}\,.
\end{split}
\end{equation}
Because functions $\hat\lambda^\pm$ chosen in Lemma \ref{lemma-Paolo-Ale-Paola} make $\mathcal{B}_0$ to be positive definite, we also have
\begin{equation}\label{est_f_b}
\int_{\mathbb R^2_+}(\mathcal B_0{\mathbf V}\cdot{\mathbf V})(t)d{\bf x}\ge c_0\Vert{\mathbf V}(t)\Vert^2_{L^2(\mathbb R^2_+)}=c_0I(t)\,,
\end{equation}
with some constant $c_0>0$ (depending on the number $k$ in \eqref{hy} and \eqref{stability1}).
\newline
Combining the latter with \eqref{en-id-I0}, \eqref{source_0} and \eqref{stima_bdry_0} then leads to
\begin{equation}\label{en-ineq-I0-prel}
\begin{split}
I(t)\le\frac1{c_0}\left\vert\int_{\Gamma_t}(\mathcal B_1{\mathbf V}\cdot{\mathbf V})\vert_{x_1=0}\,dx_2ds\right\vert+C_0\Vert{\mathbf F}\Vert_{L^2(\Omega_t)}^2+C_1\int_0^t I(s)ds\\
\le C_0\Vert{\mathbf F}\Vert_{L^2(\Omega_t)}^2+C_1\left\{\Vert \dot q^+\vert_{x_1=0}\Vert_{L^2(\Gamma_t)}^2+\Vert\varphi(t)\Vert_{L^2(\mathbb R)}^2+\int_0^t I(s)ds\right\}+C_2\Vert\varphi\Vert^2_{L^2(\Gamma_t)}\,.
\end{split}
\end{equation}
Notice that in the right-hand side of \eqref{en-ineq-I0-prel} the $L^2$-norm of the trace of $\dot q^+$ on the boundary $\Gamma_t$ needs for a control of the normal derivative of $\dot q^+$ in the interior of the domain $\Omega_t$: indeed, we apply to $\dot q^+$ the following {\em trace type inequality}
\begin{equation}\label{trace-in}
\Vert f\vert_{x_1=0}\Vert^2_{L^2(\Gamma_t)}\le\Vert f\Vert^2_{L^2(\Omega_t)}+\Vert\partial_1 f\Vert^2_{L^2(\Omega_t)}
\end{equation}
which holds true for an arbitrary sufficiently smooth scalar/vector-valued function $f=f(t,{\bf x})$ over $\Omega_t$. The above estimate \eqref{trace-in} follows by passing from a boundary integral on $\mathbb R$ to a spatial volume one on $\mathbb R^2_+$, at fixed $t$, as follows:
\begin{equation}\label{trace_t}
\begin{split}
&\int_{\mathbb R}\vert f(t)\vert_{x_1=0}\vert^2dx_2=-\int_{\mathbb R}\int_0^{+\infty}\partial_1\left(\vert f(t)\vert_{x_1=0}\vert^2\right)dx_1dx_2\\
&=-\int_{\mathbb R^2_+}2 f\partial_1 f(t)\,dx\le\Vert f(t)\Vert^2_{L^2(\mathbb R^2_+)}+\Vert\partial_1 f(t)\Vert^2_{L^2(\mathbb R^2_+)}\,,
\end{split}
\end{equation}
from which \eqref{trace-in} follows at once from integrating over $(-\infty,t)$.

This clearly shows that inequality \eqref{en-ineq-I0-prel} cannot provide a closed $L^2$-estimate of the vector unknown ${\mathbf V}$, since a control of the normal derivative of $\dot q^+$ on $\Omega_t$ is needed in the right-hand side. Fortunately, the component $\dot q^+$ of ${\mathbf V}$ belongs to the {\em noncharacteristic components} ${\bf V}^\pm_n=(\dot q^\pm,\dot u_n^\pm, \dot H^\pm_n)$, whose normal derivative can be expressed directly from linear system \eqref{A} as a function of tangential derivatives of ${\bf V}$ only and the source term $\mathcal F$. Thus estimating the $L^2-$norm of $\partial_1{\mathbf V}^\pm_n$ reduces to estimate the $L^2-$norm of tangential derivatives of ${\bf V}$ and the $L^2-$norm of $\mathcal F$, as it will be detailed in the next section.
\newline
Applying \eqref{trace-in} to $f=\dot q^+$, from \eqref{en-ineq-I0-prel}
\begin{equation}\label{en-ineq-I0}
\begin{split}
I(t)
\le C_0\Vert{\mathbf F}\Vert_{L^2(\Omega_t)}^2+C_1\left\{\Vert\varphi(t)\Vert_{L^2(\mathbb R)}^2+\int_0^t (I+I_{1,n})(s)ds\right\}+C_2\int_0^t\Vert\varphi(s)\Vert^2_{L^2(\mathbb R)}ds
\end{split}
\end{equation}
follows at once, see \eqref{I1n}.
%(and by $\hat{\mathbf c}=\hat{\mathbf c}(t,{\bf x})$ a suitable lifting of $\hat c$ over $\Omega_T$)

\subsection{Estimate of $I_{1,n}(t)$}\label{I1n_sct}
In this section we manage to find an ``explicit" expression of the normal derivative of the noncharacteristic component ${\mathbf V}^\pm_n$ of the unknowns directly from linear system \eqref{A}. The key step is taking advantage of the form of the normal derivative coefficient $\mathcal A_1$ in \eqref{AA} and noticing that from $\mathcal A_{(0)}\vert_{x_1=0}=0$ the identity
\begin{equation}\label{conorVSnor}
\mathcal A_{(0)}\partial_1{\mathbf V}=\mathcal H_{(0)}\sigma\partial_1{\mathbf V}
\end{equation}
easily follows, where the matrix coefficient $\mathcal H_{(0)}=\mathcal H_{(0)}(t,{\bf x}):=\displaystyle\frac{\mathcal A_{(0)}(t,{\bf x})}{\sigma(x_1)}\in L^\infty(\Omega_T)$ provided that $\mathcal A_{(0)}, \partial_1\mathcal A_{(0)}\in L^\infty(\Omega_T)$ and obeys the estimate
\[
\Vert\mathcal H_{(0)}\Vert_{L^\infty(\Omega_t)}\le \Vert\partial_1\mathcal A_{(0)}\Vert_{L^\infty(\Omega_t)}\,,
\]
see \cite[Lemma B.9]{MTP2009} for details.
\newline
After \eqref{conorVSnor}, from \eqref{A} we derive
\begin{equation}\label{Asigmad1}
\mathcal{A}\partial_1{\mathbf V}=\mathcal{F}-\mathcal{A}_0\partial_t{\mathbf V}-\mathcal H_{(0)}\sigma\partial_1{\mathbf V}-\mathcal{A}_2\partial_2{\mathbf V}-\mathcal{A}_3{\mathbf V}\quad\mbox{in}\,\,\,\Omega_T\,,
\end{equation}
where the matrix $\mathcal A$ involved in the left-hand side above only applies to components $\dot{q}^\pm$, $\dot{u}^\pm_n$ of the noncharacteristic part ${\mathbf V}_n$ of ${\mathbf V}$, in view of \eqref{AA}; indeed an explicit calculation gives
\begin{equation}\label{Ad1}
\begin{split}
&\mathcal A\partial_1{\mathbf V}=\left(\frac{1}{\partial_1\hat{\Phi}^+}\mathcal{E}_{12}\partial_1{\mathbf V}^+, \frac{1}{\partial_1\hat{\Phi}^-}\mathcal{E}_{12}\partial_1{\mathbf V}^-\right)\\
&\frac{1}{\partial_1\hat{\Phi}^\pm}\mathcal{E}_{12}\partial_1{\mathbf V}^\pm=\left(\frac{1}{\partial_1\hat{\Phi}^\pm}\partial_1\dot u^\pm_n, \frac{1}{\partial_1\hat{\Phi}^\pm}\partial_1\dot q^\pm, 0, 0,0,0\right)\,.
\end{split}
\end{equation}
By gathering \eqref{Asigmad1}, \eqref{Ad1} we derive the announced explicit form of normal derivatives of the noncharacteristic components $(\dot q^\pm, \dot u_n^\pm)$ of ${\mathbf V}^\pm$ as a function of space time tangential derivatives of ${\mathbf V}^\pm$ and $\mathcal F^\pm$ alone, namely
\begin{equation}\label{d1qun}
\begin{split}
&\partial_1\dot q^\pm=\partial_1\hat\Phi^\pm\left(\mathcal{F}^\pm-\mathcal{A}^\pm_0\partial_t{\mathbf V}^\pm-\mathcal H_{(0)}^\pm\sigma\partial_1{\mathbf V}^\pm-\mathcal{A}_2^\pm\partial_2{\mathbf V}^\pm-\mathcal{A}_3^\pm{\mathbf V}^\pm\right)_1\\
&\partial_1\dot u^\pm_n=\partial_1\hat\Phi^\pm\left(\mathcal{F}^\pm-\mathcal{A}^\pm_0\partial_t{\mathbf V}^\pm-\mathcal H_{(0)}^\pm\sigma\partial_1{\mathbf V}^\pm-\mathcal{A}_2^\pm\partial_2{\mathbf V}^\pm-\mathcal{A}_3^\pm{\mathbf V}^\pm\right)_2\quad\mbox{in}\,\,\,\Omega_T\,,
\end{split}
\end{equation}
where the subindices $1,2$ appearing above are referred to the first and the second components of the vectors.
\newline
As regards to the normal derivative of the noncharacteristic components $\dot H^\pm_n$, they can be still derived from the first condition in \eqref{I&Bconst} as a known function of tangential derivatives of ${\mathbf V}^\pm$
\begin{equation}\label{d1Hn}
\partial_1\dot H^\pm_n=-\partial_2(\dot H^\pm_2\partial_1\hat\Phi^\pm)=-\partial_2\dot H^\pm_2\partial_1\hat\Phi^\pm-\dot H^\pm_2\partial_2\partial_1\hat\Phi^\pm\quad\mbox{in}\,\,\,\Omega_T\,.
\end{equation}
The estimate of $I_{1,n}(t)$ then follows at once from \eqref{d1qun} and \eqref{d1Hn}; we get
\begin{equation}\label{en-ineq-I1n}
\begin{split}
I_{1,n}(t)&\le C_0\Vert{\mathbf F}(t)\Vert_{L^2(\mathbb R^2_+)}^2+C_1\left\{I(t)+I_0(t)+I_{\sigma}(t)+I_2(t)\right\}\\
&=C_0\Vert{\mathbf F}(t)\Vert_{L^2(\mathbb R^2_+)}^2+C_1 I_{1,\ast}(t)\,.
\end{split}
\end{equation}
After estimate \eqref{en-ineq-I1n}, it appears that the $L^2-$estimate of the normal derivative of the noncharacteristic part ${\mathbf V}_n^\pm$ of the solution is reducted to control the $L^2-$norm of the tangential space time derivatives of ${\mathbf V}$, that is $I_\sigma(t)$, $I_2(t)$ and $I_0(t)$, which naturally leads to establish an $H^1_\ast-$estimate for ${\mathbf V}$. The subsequent sections will be devoted to obtain the $H^1_\ast-$estimate.
%%%%%%%%%%%%%%%%%%%%%%%%%%%%%%%%%%%%%%%%%%%%%%%%%%%%%%%%%%%%%%%%%%%%%%%%%
\subsection{Estimate of $I_\sigma(t)$}\label{Isigma_sct}
We set for shortness ${\mathbf V}_\sigma:=\sigma\partial_1{\mathbf V}$. Applying the conormal derivative $\sigma\partial_1$ to both sides of system \eqref{BBB} we get a similar linear system satisfied by ${\mathbf V}_\sigma$. We compute
\begin{equation}\label{syst_Vs_prel}
\begin{split}
\mathcal{B}_0\partial_t{\mathbf V}_\sigma+\mathcal{B}_1\sigma\partial_1\partial_1{\mathbf V}+\sigma\partial_1\mathcal B_1\partial_1{\mathbf V}+\mathcal{B}_2\partial_2{\mathbf V}_\sigma+\mathcal{B}_3{\mathbf V}_\sigma\\
=\sigma\partial_1\tilde{\mathcal{F}}-\sigma\partial_1\mathcal B_0\partial_t{\mathbf V}-\sigma\partial_1\mathcal B_2\partial_2{\mathbf V}-\sigma\partial_1\mathcal B_3{\mathbf V}.
\end{split}
\end{equation}
From Leibniz's rule
\[
\partial_1{\mathbf V}_\sigma=\partial_1(\sigma\partial_1{\mathbf V})=\sigma\partial_1\partial_1{\mathbf V}+\sigma^\prime\partial_1{\mathbf V}\,;
\]
plugging the latter into \eqref{syst_Vs_prel} and rewriting $\sigma\partial_1\mathcal B_1\partial_1{\mathbf V}=\partial_1\mathcal B_1{\mathbf V}_\sigma$
gives
\begin{equation}\label{syst_Vs}
\mathcal{B}_0\partial_t{\mathbf V}_\sigma+\mathcal{B}_1\partial_1{\mathbf V}_\sigma+\mathcal{B}_2\partial_2{\mathbf V}_\sigma+(\mathcal{B}_3+\partial_1\mathcal B_1){\mathbf V}_\sigma=\tilde{\mathcal F}_{\sigma}\,,
\end{equation}
where
\begin{equation}\label{Fs}
\tilde{\mathcal F}_\sigma:=\sigma\partial_1\tilde{\mathcal{F}}-\sigma\partial_1\mathcal B_0\partial_t{\mathbf V}+\sigma^\prime\mathcal B_1\partial_1{\mathbf V}-\sigma\partial_1\mathcal B_2\partial_2{\mathbf V}-\sigma\partial_1\mathcal B_3{\mathbf V}\,.
\end{equation}
Performing on system \eqref{syst_Vs} the same standard energy arguments as done for system \eqref{BBB} in Sect. \ref{I0_sct}, leads to the following energy identity for ${\mathbf V}_\sigma$
\begin{equation*}
\begin{split}
\int_{\mathbb R^2_+}(\mathcal B_0{\mathbf V}_\sigma\cdot{\mathbf V}_\sigma)(t)d{\bf x}&-\int_{\Gamma_t}(\mathcal B_1{\mathbf V}_\sigma\cdot{\mathbf V}_\sigma)\vert_{x_1=0}\,dx_2ds=2\int_{\Omega_t}\tilde{\mathcal F}_\sigma\cdot{\mathbf V}_\sigma d{\bf x}ds\\
&+\int_{\Omega_t}(\partial_t\mathcal B_0+\partial_2\mathcal B_2-\partial_1\mathcal B_1-\mathcal B_3){\mathbf V_{\sigma}}\cdot{\mathbf V_{\sigma}}d{\bf x}ds\,;
\end{split}
\end{equation*}
however because $\sigma\vert_{x_1=0}=0$, the quadratic form $(\mathcal B_1{\mathbf V}_\sigma\cdot{\mathbf V}_\sigma)$ under the boundary integral in the left-hand side above vanishes and the energy identity reduces to
\begin{equation}\label{en_id_Is}
\int_{\mathbb R^2_+}(\mathcal B_0{\mathbf V}_\sigma\cdot{\mathbf V}_\sigma)(t)d{\bf x}=2\int_{\Omega_t}\tilde{\mathcal F}_\sigma\cdot{\mathbf V}_\sigma d{\bf x}ds+\int_{\Omega_t}(\partial_t\mathcal B_0+\partial_2\mathcal B_2-\partial_1\mathcal B_1-\mathcal B_3){\mathbf V_{\sigma}}\cdot{\mathbf V_{\sigma}}d{\bf x}ds\,.
\end{equation}
The second integral in the right-hand side of the identity above is trivially estimated as in \eqref{source_0}
\begin{equation*}
\int_{\Omega_t}(\partial_t\mathcal B_0+\partial_2\mathcal B_2-\partial_1\mathcal B_1-\mathcal B_3){\mathbf V_{\sigma}}\cdot{\mathbf V_{\sigma}}d{\bf x}ds\le C_1\Vert{\mathbf V_{\sigma}}\Vert_{L^2(\Omega_t)}^2\,.
\end{equation*}
Now we focus on the estimate of the first integral
\begin{equation}\label{int_Fs}
2\int_{\Omega_t}\tilde{\mathcal F}_\sigma\cdot{\mathbf V}_\sigma d{\bf x}ds
\end{equation}
in the right-hand side of \eqref{en_id_Is}. Substituting the explicit form \eqref{Fs} of $\tilde{F}_\sigma$, using that, similarly to $A_1$ (see \eqref{AA})
\begin{equation}\label{matrice_B1}
\mathcal B_1=\mathcal B^{(0)}_1+\mathcal B\,,\quad\mbox{where}\,\,\,\mathcal B^{(0)}_1\vert_{x_1=0}=0\,,
\end{equation}
while $\mathcal B$ only ``applies" to the noncharacteristic components ${\bf V}^\pm_n$ of ${\mathbf V}^\pm$, and denoting $\mathcal K^{(0)}_1=\frac{\mathcal B^{(0)}_1}{\sigma},$ we have from Cauchy--Schwarz and Young's inequalities
\[
\begin{split}
&\left\vert2\int_{\Omega_t}\tilde{\mathcal F}_\sigma\cdot{\mathbf V}_\sigma d{\bf x}ds\right\vert\\
&=\left\vert2\int_{\Omega_t}\left(\sigma\partial_1\tilde{\mathcal{F}}-\sigma\partial_1\mathcal B_0\partial_t{\mathbf V}+\sigma^\prime\mathcal K^{(0)}_1\sigma\partial_1{\mathbf V}+\sigma^\prime\mathcal B\partial_1{\mathbf V}-\sigma\partial_1\mathcal B_2\partial_2{\mathbf V}-\sigma\partial_1\mathcal B_3{\mathbf V}\right)\cdot{\mathbf V}_\sigma d{\bf x}ds\right\vert\\
&\le 2\left\{\Vert\sigma\partial_1\tilde{\mathcal F}\Vert_{L^2(\Omega_t)}+C_1\left(\Vert\partial_t{\mathbf V}\Vert_{L^2(\Omega_t)}+\Vert{\mathbf V}_\sigma\Vert_{L^2(\Omega_t)}+\Vert\partial_1{\mathbf V}_n\Vert_{L^2(\Omega_t)}\right.\right.\\
&\left.\left.\qquad\qquad+\Vert\partial_2{\mathbf V}\Vert_{L^2(\Omega_t)}+\Vert{\mathbf V}\Vert_{L^2(\Omega_t)}\right)\right\}\Vert{\mathbf V}_\sigma\Vert_{L^2(\Omega_t)}\\
&\le \Vert\sigma\partial_1\tilde{\mathcal F}\Vert_{L^2(\Omega_t)}^2+C_1\int_0^t(I+I_0+I_{\sigma}+I_{1,n}+I_2)(s)ds\\
&=C_1\left\{\Vert{\mathbf F}\Vert^2_{L^2(\Omega_t)}+\Vert\sigma\partial_1{\mathbf F}\Vert_{L^2(\Omega_t)}^2+\int_0^t(I_{1,\ast}+I_{1,n})(s)ds\right\}\,.
\end{split}
\]
At last, estimating from below the space integral in the left-hand side of \eqref{en_id_Is} as in \eqref{est_f_b} (with ${\mathbf V}_\sigma$ instead of ${\mathbf V}$) we end up with
\begin{equation}\label{en-ineq-Is}
I_\sigma(t)\le C_1\left\{\Vert{\mathbf F}\Vert^2_{L^2(\Omega_t)}+\Vert\sigma\partial_1{\mathbf F}\Vert_{L^2(\Omega_t)}^2+\int_0^t(I_{1,\ast}+I_{1,n})(s)ds\right\}\,.
\end{equation}
%%%%%%%%%%%%%%%%%%%%%%%%%%%%%%%%%%%%%%%%%%%%%%%%%%%%%%%%%%%%%%%%%%%%%%%%%%
\subsection{Estimate of $I_2(t)$}\label{I2_sct}
We set for shortness ${\mathbf V}_{x_2}:=\partial_2{\mathbf V}$. Applying $\partial_2$ to both sides of system \eqref{BBB} we get a similar linear system satisfied by ${\mathbf V}_{x_2}$, that is
\begin{equation}\label{syst_Vx2_prel}
\mathcal{B}_0\partial_t{\mathbf V}_{x_2}+\mathcal{B}_1\partial_1{\mathbf V}_{x_2}+\mathcal{B}_2\partial_2{\mathbf V}_{x_2}+(\partial_2\mathcal B_2+\mathcal{B}_3){\mathbf V}_{x_2}=\tilde{\mathcal F}_2\,,
\end{equation}
where
\[
\tilde{\mathcal F}_2:=\partial_2\tilde{\mathcal{F}}-\partial_2\mathcal B_0\partial_t{\mathbf V}-\partial_2\mathcal B_1\partial_1{\mathbf V}-\partial_2\mathcal B_3{\mathbf V}\,.
\]
As usual, from the above linear system we derive, by scalar multiplication by ${\mathbf V}_{x_2}$ and integration by parts in $\Omega_t$, the energy identity
\begin{equation}\label{en-id-I2}
\begin{split}
\int_{\mathbb R^2_+}(\mathcal B_0{\mathbf V}_{x_2}\cdot{\mathbf V}_{x_2})(t)d{\bf x}&-\int_{\Gamma_t}(\mathcal B_1{\mathbf V}_{x_2}\cdot{\mathbf V}_{x_2})\vert_{x_1=0}\,dx_2ds=2\int_{\Omega_t}\tilde{\mathcal F}_2\cdot{\mathbf V}_{x_2}d{\bf x}ds\\
&+\int_{\Omega_t}(\partial_t\mathcal B_0+\partial_1\mathcal B_1-2\partial_2\mathcal B_2-\mathcal B_3){\mathbf V}_{x_2}\cdot{\mathbf V}_{x_2}d{\bf x}ds\,.
\end{split}
\end{equation}
The second integral in the right-hand side of the above identity is estimated as usual as
\begin{equation}\label{est_0_term}
\int_{\Omega_t}(\partial_t\mathcal B_0+\partial_1\mathcal B_1-2\partial_2\mathcal B_2-\mathcal B_3){\mathbf V}_{x_2}\cdot{\mathbf V}_{x_2}d{\bf x}ds\le C_1\Vert{\mathbf V}_{x_2}\Vert_{L^2(\Omega_t)}^2=C_1\int_0^t I_2(s)ds\,,
\end{equation}
whereas to estimate the first integral in the right-hand side we still use the decomposition of $\mathcal B_1$ in \eqref{matrice_B1} and repeat the same arguments used in the estimate of \eqref{int_Fs} and use \eqref{source}, to get
\begin{equation}\label{est_F2}
\begin{split}
&2\int_{\Omega_t}\tilde{\mathcal F}_2\cdot{\mathbf V}_{x_2} d{\bf x}ds\\
&\le 2\left\{\Vert\partial_2\tilde{\mathcal F}\Vert_{L^2(\Omega_t)}+C_1\left(\Vert\partial_t{\mathbf V}\Vert_{L^2(\Omega_t)}+\Vert\sigma\partial_1{\mathbf V}\Vert_{L^2(\Omega_t)}+\Vert\partial_1{\mathbf V}_n\Vert_{L^2(\Omega_t)}\right.\right.\\
&\left.\left.\qquad\qquad+\Vert{\mathbf V}\Vert_{L^2(\Omega_t)}\right)\right\}\Vert{\mathbf V}_{x_2}\Vert_{L^2(\Omega_t)}\\
&\le C_1\left\{\Vert{\mathbf F}\Vert^2_{L^2(\Omega_t)}+\Vert\partial_2{\mathbf F}\Vert_{L^2(\Omega_t)}^2+\int_0^t(I_{1,\ast}+I_{1,n})(s)ds\right\}\,.
\end{split}
\end{equation}
Now we need to get an estimate of the quadratic form $(\mathcal B_1{\mathbf V}_{x_2}\cdot{\mathbf V}_{x_2})$ under the second boundary integral in the left-hand side of \eqref{en-id-I2}. The explicit expression of this quadratic form is in principle the same as the one for ${\mathbf V}$ in \eqref{boundary-q-form-lambda}, that is
\begin{equation}\label{boundary-q-form-d2}
(\mathcal{B}_1{\mathbf V}_{x_2},{\mathbf V}_{x_2})|_{x_1=0}=2[\partial_2\dot{q}(\partial_2\dot{u}_{N}-\hat{\lambda}\partial_2\dot{H}_{N})],
\end{equation}
with $\hat\lambda:= \lambda(\hat{\mathbf U})$.
\newline
As done to treat the quadratic form \eqref{boundary-q-form-lambda} for ${\mathbf V}$, now we make use of the boundary conditions in \eqref{BC} and the boundary constraint in \eqref{I&Bconst}, differentiated with respect to $x_2$, to rewrite \eqref{boundary-q-form-d2} as the sum of the same leading part as in \eqref{boundary-q-form_lambda_1}, vanishing as a consequence of the choice of $\hat\lambda^\pm$, and lower order terms. We compute
\begin{equation}\label{boundary-q-form-d2_1}
\begin{split}
(\mathcal{B}_1{\mathbf V}_{x_2},{\mathbf V}_{x_2})|_{x_1=0}=2[\hat{u}_2-\hat\lambda\hat{H}_2]\partial_2\dot{q}^+\partial_{2}\partial_2\varphi + \mbox{l.o.t}\,,
\end{split}
\end{equation}
where
\begin{equation}\label{lot_d2}
\begin{split}
\mbox{l.o.t}:=&-2[\partial_{1}\hat{u}_N-\hat{\lambda}\partial_{1}\hat{H}_N-\partial_2\hat u_2+\hat\lambda\partial_2\hat H_2]\partial_2\dot{q}^+\partial_2\varphi-2[\partial_2\partial_1\hat u_N-\hat\lambda\partial_2\partial_1\hat H_N]\partial_2\dot{q}^+\varphi\\
&-2[\partial_{1}\hat{q}]\partial_2\varphi(\partial_2\dot u^-_{N}-\hat\lambda^-\partial_2\dot H^-_{N})-2\partial_2([\partial_{1}\hat{q}])\varphi(\partial_2\dot u^-_{N}-\hat\lambda^-\partial_2\dot H^-_{N})\,.
\end{split}
\end{equation}
As already announced the leading quadratic term $2[\hat{u}_2-\hat\lambda\hat{H}_2]\dot{q}^+_{x_2}\partial_{2}\varphi$ in \eqref{boundary-q-form-d2_1} vanishes because of the chosen $\hat\lambda^\pm$. We now focus on the estimate of the boundary integral of the different lower order quadratic terms in \eqref{lot_d2}. Here below we denote again by $\hat{c}=\hat{c}(t,x_2)$ the different coefficients of those boundary lower order terms, which are all smooth functions of $\hat{\mathbf U}^\pm\vert_{x_1=0}$ and $\nabla_{t,x_2}\hat\varphi$ and their derivatives and whose explicit form is useless.
\newline
From \eqref{lot_d2}, we get
\begin{equation}\label{boundary-q-form-d2_2}
\begin{split}
\int_{\Gamma_t}(\mathcal B_1{\mathbf V}_{x_2}\cdot{\mathbf V}_{x_2})\vert_{x_1=0}\,dx_2ds&=\int_{0}^t\!\!\!\int_{\mathbb R}\!\!\hat c\,\partial_2\dot{q}^+\partial_2\varphi dx_2ds+\int_{0}^t\!\!\!\int_{\mathbb R}\!\!\hat c\,\partial_2\dot{q}^+\varphi\,dx_2ds\\
&\quad+\int_{0}^t\!\!\!\int_{\mathbb R}\!\!\hat c\,\partial_2\dot u^-_{N}\partial_2\varphi\,dx_2ds+\int_{0}^t\!\!\!\int_{\mathbb R}\!\!\hat c\,\partial_2\dot H^-_{N}\partial_2\varphi\,dx_2ds\\
&\quad+\int_{0}^t\!\!\!\int_{\mathbb R}\!\!\hat c\,\partial_2\dot u^-_{N}\varphi\,dx_2ds+\int_{0}^t\!\!\!\int_{\mathbb R}\!\!\hat c\,\partial_2\dot H^-_{N}\varphi\,dx_2ds\,.
\end{split}
\end{equation}
To estimate the boundary integrals above, we follow similar arguments to those of \cite{Trakhinin2005}.
\newline
The first step is to write $\partial_2\varphi$ as a linear combination of $\dot H^\pm_{N}\vert_{x_1=0}$ and $\varphi$; this can be done by making use of boundary constraints \eqref{I&Bconst} and exploiting that $\hat H_2^\pm$ are never simultaneously zero on the boundary as a consequence of the stability condition \eqref{stability1}, see Remark \ref{H_2nonzero}. From the boundary conditions \eqref{I&Bconst} we have
\[
\hat{H}^{+}_2\partial_2\varphi=\dot{H}^{+}_{N}+\varphi\partial_1\hat{H}^{+}_{N}\quad\mbox{and}\quad \hat{H}^{-}_2\partial_2\varphi=\dot{H}^{-}_{N}-\varphi\partial_1\hat{H}^{-}_{N}\qquad\mbox{on}\,\,\,\Gamma_T\,.
\]
Then multiplying the first one by $\hat{H}^{+}_2\vert_{x_1=0}$ and the second one by  $\hat{H}^{-}_2\vert_{x_1=0}$, then adding the results we get
\begin{equation}\label{d2phi}
\begin{split}
\partial_2\varphi&=\frac{\hat{H}^{+}_2\dot{H}^{+}_{N}+\hat{H}^{-}_2\dot{H}^{-}_{N}+(\hat{H}^{+}_2\partial_1\hat{H}^{+}_{N}-\hat{H}^{-}_2\partial_1\hat{H}^{-}_{N})\varphi}{(\hat{H}^{+}_2)^2+(\hat{H}^{-}_2)^2}\\
&=\hat d_1\dot H^+_N+\hat d_2\dot H^-_N+\hat d_3\varphi\,,\quad\mbox{on}\,\,\,\Gamma_T\,,
\end{split}
\end{equation}
where $\hat d_i=\hat d_i(t,x_2)$ are suitable functions depending only on the boundary values of $\hat H^\pm$, $\partial_1\hat H^\pm$ and second order derivatives of $\hat\varphi$, whose esplicit form could be easily deduced from above.

\smallskip
Let us start to estimate the first term in the right-hand side of \eqref{boundary-q-form-d2_2}. Inserting the expression of $\partial_2\varphi$ provided by \eqref{d2phi}, we find
\begin{equation*}
\begin{split}
\int_{0}^t\!\!\!\int_{\mathbb R}\!\!\hat c\,\partial_2\dot{q}^+\partial_2\varphi dx_2ds&=\int_{0}^t\!\!\!\int_{\mathbb R}\!\!\hat c\,\partial_2\dot{q}^+\dot{H}^+_N dx_2ds+\int_{0}^t\!\!\!\int_{\mathbb R}\!\!\hat c\,\partial_2\dot{q}^+\dot{H}^-_N dx_2ds+\int_{0}^t\!\!\!\int_{\mathbb R}\!\!\hat c\,\partial_2\dot{q}^+\varphi\,dx_2ds\\
&=\mathcal I_1+\mathcal I_2+\mathcal I_3\,.
\end{split}
\end{equation*}
Let us estimate $\mathcal I_1$. Here the trick of passing from a boundary integral over $\mathbb R$ to a volume integral over $\mathbb R^2_+$ is used as already done to get the trace inequality \eqref{trace-in}.
In the following we will adopt the notation $\hat{\mathbf c}=\hat{\mathbf c}(t,{\bf x})$ to mean a suitable lifting from $\Gamma_t$ to $\Omega_t$ of a boundary coefficient $\hat c=\hat c(t,x_2)$. Then we have
\[
\begin{split}
\mathcal I_1&=\int_{0}^t\!\!\!\int_{\mathbb R}\!\!\hat c\,\partial_2\dot{q}^+\dot{H}^+_N dx_2ds=-\int_{0}^t\!\!\!\int_{\mathbb R^2_+}\!\!\partial_1(\hat {\mathbf c}\,\partial_2\dot{q}^+\dot{H}^+_n) d{\bf x}ds\\
&=-\int_{0}^t\!\!\!\int_{\mathbb R^2_+}\!\!\partial_1\hat {\mathbf c}\,\partial_2\dot{q}^+\dot{H}^+_n d{\bf x}ds-\int_{0}^t\!\!\!\int_{\mathbb R^2_+}\!\!\hat {\mathbf c}\,\partial_1\partial_2\dot{q}^+\dot{H}^+_n d{\bf x}ds
-\int_{0}^t\!\!\!\int_{\mathbb R^2_+}\!\!\hat {\mathbf c}\,\partial_2\dot{q}^+\partial_1\dot{H}^+_n d{\bf x}ds\\
&=\mathcal I_{1,1}+\mathcal I_{1,2}+\mathcal I_{1,3}\,.
\end{split}
\]
$\mathcal I_{1,1}$ and $\mathcal I_{1,3}$ can be easily estimated by Cauchy-Schwarz and Young's inequalities by
\begin{equation}\label{I11+13}
\begin{split}
\vert\mathcal I_{1,1}+\mathcal I_{1,3}\vert&\le C_2\Vert\partial_2\dot{q}^+\Vert_{L^2(\Omega_t)}\Vert\dot{H}^+_n\Vert_{L^2(\Omega_t)}+C_1\Vert\partial_2\dot{q}^+\Vert_{L^2(\Omega_t)}\Vert\partial_1\dot{H}^+_n\Vert_{L^2(\Omega_t)}\\
&\le C_2\left\{\int_0^t(I_{1,\ast}+I_{1,n})(s)ds\right\}
\end{split}
\end{equation}
recall that ${\mathbf V}_n^\pm=(\dot q^\pm, \dot u^\pm_n, \dot H^\pm_n)$ are the noncharacteristic components of the solution ${\mathbf V}^\pm$, see also \eqref{I1n}.
\newline
As regards to $\mathcal I_{1,2}$, differently from above, we cannot immediately end up by Cauchy-Schwarz and Young's inequality, because this should require a control of the $L^2-$norm of the second order derivative of $\dot q^+$, preventing to close the $H^1_\ast-$estimate. Instead, here integration by parts with respect to the tangential space variable $x_2$ and Leibiniz's rule are used to further rewrite $\mathcal I_{1,2}$ as
\begin{equation}\label{I12}
\begin{split}
\mathcal I_{1,2}&=-\int_{0}^t\!\!\!\int_{\mathbb R^2_+}\!\!\hat {\mathbf c}\,\partial_1\partial_2\dot{q}^+\dot{H}^+_n d{\bf x}ds=\int_{0}^t\!\!\!\int_{\mathbb R^2_+}\!\!\partial_2(\hat {\mathbf c}\dot{H}^+_n)\partial_1\dot{q}^+ d{\bf x}ds\\
&=\int_{0}^t\!\!\!\int_{\mathbb R^2_+}\!\!\partial_2\hat {\mathbf c}\dot{H}^+_n\partial_1\dot{q}^+ d{\bf x}ds+\int_{0}^t\!\!\!\int_{\mathbb R^2_+}\!\!\hat {\mathbf c}\partial_2\dot{H}^+_n\partial_1\dot{q}^+ d{\bf x}ds\,;
\end{split}
\end{equation}
then we observe that the last two integral above are similar to $\mathcal I_{1,1}$ and $\mathcal I_{1,3}$ and therefore can be estimated in the same way by
\[
\begin{split}
\vert\mathcal I_{1,2}\vert&\le C_2\Vert\partial_1\dot{q}^+\Vert_{L^2(\Omega_t)} \left(\Vert\dot{H}^+_n\Vert_{L^2(\Omega_t)}+\Vert\partial_2\dot{H}^+_n\Vert_{L^2(\Omega_t)}\right)\\
&\le C_2\left\{\int_0^t(I_{1,\ast}+I_{1,n})(s)ds\right\}\,.
\end{split}
\]
Adding \eqref{I11+13} and \eqref{I12} gives the estimate of $\mathcal I_1$
\[
\vert\mathcal I_1\vert\le C_2\left\{\int_0^t(I_{1,\ast}+I_{1,n})(s)ds\right\}\,.
\]
It is clear that $\mathcal I_2:=\displaystyle\int_{0}^t\!\!\!\int_{\mathbb R}\!\!\hat c\,\partial_2\dot{q}^+\dot{H}^-_N dx_2ds$ can estimated by repeating exactly the same arguments applied to $\mathcal I_1$.
\newline
Concerning $\mathcal I_3=\displaystyle\int_{0}^t\!\!\!\int_{\mathbb R}\!\!\hat c\,\partial_2\dot{q}^+\varphi\,dx_2ds$, we are reduced to apply the same arguments as above by first integrating by parts and then using Leibniz's rule and replacing once again $\partial_2\varphi$ by the expression in the right-hand side of \eqref{d2phi}:
\[
\begin{split}
\mathcal I_3&=\int_{0}^t\!\!\!\int_{\mathbb R}\!\!\hat c\,\partial_2\dot{q}^+\varphi\,dx_2ds=-\int_{0}^t\!\!\!\int_{\mathbb R}\!\!\partial_2(\hat c\varphi)\dot{q}^+dx_2ds\\
&=-\int_{0}^t\!\!\!\int_{\mathbb R}\!\!\partial_2\hat c\,\varphi\dot{q}^+dx_2ds-\int_{0}^t\!\!\!\int_{\mathbb R}\!\!\hat c\,\partial_2\varphi\dot{q}^+dx_2ds\\
&=\int_{0}^t\!\!\!\int_{\mathbb R}\!\!\hat c\,\varphi\dot{q}^+dx_2ds+\int_{0}^t\!\!\!\int_{\mathbb R}\!\!\hat c\,\dot{H}^+_N\dot{q}^+dx_2ds+\int_{0}^t\!\!\!\int_{\mathbb R}\!\!\hat c\,\dot{H}^-_N\dot{q}^+dx_2ds\,.
\end{split}
\]
Since all the boundary traces under the integral are of noncharacteristic components of ${\mathbf V}$, we end up by Cauchy-Schwarz, Young's inequalities and trace type inequality \eqref{trace-in} to get
\begin{equation}
\vert\mathcal I_3\vert\le C_2\left\{\int_0^t(I+I_{1,n})(s)ds+\int_0^t\Vert\varphi(s)\Vert_{L^2(\mathbb R)}^2 ds\right\}\,.
\end{equation}
To complete the estimate of the remaining boundary integrals involved in the right-hand side of \eqref{boundary-q-form-d2_2}, it is then sufficient to notice that the second boundary integral is exactly the same as $\mathcal I_3$, while the other boundary integrals are the same as the first and second ones, where $\partial_2\dot q^+\vert_{x_1=0}$ is replaced by $\partial_2\dot u^-_N\vert_{x_1=0}$ or by $\partial_2\dot H^-_N\vert_{x_1=0}$ (but $\dot u^-_n$ and $\dot H^-_n$ are still noncharacteristic components of the vector solution ${\mathbf V}$, so that they are treated along the same arguments as to $\dot q^+$). Therefore we end up with
\begin{equation}\label{stima_bdry_2}
\left\vert\int_{\Gamma_t}(\mathcal B_1{\mathbf V}_{x_2}\cdot{\mathbf V}_{x_2})\vert_{x_1=0}\,dx_2ds\right\vert\le C_2\left\{\int_0^t(I_{1,\ast}+I_{1,n})(s)ds+\int_0^t\Vert\varphi(s)\Vert_{L^2(\mathbb R)}^2 ds\right\}\,.
\end{equation}
Using \eqref{est_0_term}, \eqref{est_F2} and \eqref{stima_bdry_2} together with the counterpart of \eqref{est_f_b} with ${\mathbf V}_{x_2}$ instead of ${\mathbf V}$, from \eqref{en-id-I2} we obtain
\begin{equation}\label{en-ineq-I2}
I_2(t)\le C_1\left\{\Vert{\mathbf F}\Vert^2_{L^2(\Omega_t)}+\Vert\partial_2{\mathbf F}\Vert_{L^2(\Omega_t)}^2\right\}+ C_2\left\{\int_0^t(I_{1,\ast}+I_{1,n})(s)ds+\int_0^t\Vert\varphi(s)\Vert_{L^2(\mathbb R)}^2 ds\right\}\,.
\end{equation}
%%%%%%%%%%%%%%%%%%%%%%%%%%%%%%%%%%%%%%%%%%%%%%%%%%%%%%%%%%%%%%%%%%%%%%%
\subsection{Estimate of $I_0(t)$}\label{It_sct}
We set for shortness ${\mathbf V}_{t}:=\partial_t{\mathbf V}$. Applying $\partial_t$ to both sides of system \eqref{BBB} we get a similar linear system satisfied by ${\mathbf V}_{t}$, that is
\begin{equation}\label{syst_Vt_prel}
\mathcal{B}_0\partial_t{\mathbf V}_{t}+\mathcal{B}_1\partial_1{\mathbf V}_{t}+\mathcal{B}_2\partial_2{\mathbf V}_{t}+(\partial_t\mathcal B_0+\mathcal{B}_3){\mathbf V}_{t}=\tilde{\mathcal F}_t\,,
\end{equation}
where
\[
\tilde{\mathcal F}_t:=\partial_t\tilde{\mathcal{F}}-\partial_t\mathcal B_1\partial_1{\mathbf V}-\partial_t\mathcal B_2\partial_2{\mathbf V}-\partial_t\mathcal B_3{\mathbf V}\,.
\]
As usual, from the above linear system we derive, by scalar multiplication by ${\mathbf V}_{t}$ and integration by parts in $\Omega_t$, the energy identity
\begin{equation}\label{en-id-It}
\begin{split}
\int_{\mathbb R^2_+}(\mathcal B_0{\mathbf V}_{t}\cdot{\mathbf V}_{t})(t)d{\bf x}&-\int_{\Gamma_t}(\mathcal B_1{\mathbf V}_{s}\cdot{\mathbf V}_{s})\vert_{x_1=0}\,dx_2ds=2\int_{\Omega_t}\tilde{\mathcal F}_s\cdot{\mathbf V}_{s}d{\bf x}ds\\
&+\int_{\Omega_t}(\partial_1\mathcal B_1+\partial_2\mathcal B_2-\partial_t\mathcal B_0-\mathcal B_3){\mathbf V}_{s}\cdot{\mathbf V}_{s}d{\bf x}ds\,.
\end{split}
\end{equation}
The second integral in the right-hand side of the above identity is estimated as usual as
\begin{equation}\label{est_0_term_t}
\int_{\Omega_t}(\partial_1\mathcal B_1+\partial_2\mathcal B_2-\partial_t\mathcal B_0-\mathcal B_3){\mathbf V}_{s}\cdot{\mathbf V}_{s}d{\bf x}ds\le C_1\Vert{\mathbf V}_{s}\Vert_{L^2(\Omega_t)}^2=C_1\int_0^t I_0(s)ds\,,
\end{equation}
whereas to estimate the first integral in the right-hand side we still use the decomposition of $\mathcal B_1$ in \eqref{matrice_B1} and repeat the same arguments used in the estimates \eqref{int_Fs}, \eqref{est_F2} and use \eqref{source}, to get
\begin{equation}\label{est_Ft}
\begin{split}
&2\int_{\Omega_t}\tilde{\mathcal F}_s\cdot{\mathbf V}_{s} d{\bf x}ds\\
&\le 2\left\{\Vert\partial_s\tilde{\mathcal F}\Vert_{L^2(\Omega_t)}+C_1\left(\Vert\sigma\partial_1{\mathbf V}\Vert_{L^2(\Omega_t)}+\Vert\partial_1{\mathbf V}_n\Vert_{L^2(\Omega_t)}\right.\right.\\
&\left.\left.\qquad\qquad+\Vert\partial_2{\mathbf V}\Vert_{L^2(\Omega_t)}+\Vert{\mathbf V}\Vert_{L^2(\Omega_t)}\right)\right\}\Vert{\mathbf V}_{t}\Vert_{L^2(\Omega_t)}\\
&\le C_1\left\{\Vert{\mathbf F}\Vert^2_{L^2(\Omega_t)}+\Vert\partial_s{\mathbf F}\Vert_{L^2(\Omega_t)}^2+\int_0^t(I_{1,\ast}+I_{1,n})(s)ds\right\}\,.
\end{split}
\end{equation}
Now we need to get an estimate of the quadratic form $(\mathcal B_1{\mathbf V}_{t}\cdot{\mathbf V}_{t})$ under the second boundary integral in the left-hand side of \eqref{en-id-It}, whose explicit expression is
\begin{equation}\label{boundary-q-form-dt}
(\mathcal{B}_1{\mathbf V}_{t},{\mathbf V}_{t})|_{x_1=0}=2[\partial_t\dot{q}(\partial_t\dot{u}_{N}-\hat{\lambda}\partial_t\dot{H}_{N})],
\end{equation}
with $\hat\lambda:= \lambda(\hat{\mathbf U})$.
\newline
As in the case of the quadratic form \eqref{boundary-q-form-d2}, we make use of the boundary conditions in \eqref{BC}, differentiated with respect to $t$, to rewrite \eqref{boundary-q-form-dt} as the sum of the same leading part as in \eqref{boundary-q-form-d2_1}, vanishing as a consequence of the choice of $\hat\lambda^\pm$, and lower order terms. We compute
\begin{equation}\label{boundary-q-form-dt_1}
\begin{split}
(\mathcal{B}_1{\mathbf V}_{t},{\mathbf V}_{t})|_{x_1=0}=2[\hat{u}_2-\hat\lambda\hat{H}_2]\partial_t\dot{q}^+\partial_{t}\partial_2\varphi + \mbox{l.o.t}\,,
\end{split}
\end{equation}
where
\begin{equation}\label{lot_dt}
\begin{split}
\mbox{l.o.t}:=&-2[\partial_{1}\hat{u}_N-\hat{\lambda}\partial_{1}\hat{H}_N]\partial_t\dot{q}^+\partial_t\varphi+2[\partial_t\hat u_2-\hat\lambda\partial_t\hat H_2]\partial_2\varphi\partial_t\dot q^+\\
&-2[\partial_t\partial_1\hat u_N-\hat\lambda\partial_t\partial_1\hat H_N]\partial_t\dot{q}^+\varphi-2[\partial_{1}\hat{q}]\partial_t\varphi(\partial_t\dot u^-_{N}-\hat\lambda^-\partial_t\dot H^-_{N})\\
&-2\partial_t([\partial_{1}\hat{q}])\varphi(\partial_t\dot u^-_{N}-\hat\lambda^-\partial_t\dot H^-_{N})\\
&=\hat c\,\partial_t\dot{q}^+\partial_t\varphi+\hat c\,\partial_2\varphi\partial_t\dot q^++\hat c\,\partial_t\dot{q}^+\varphi+\hat c\,\partial_t\varphi\partial_t\dot u^-_{N}+\hat c\,\partial_t\varphi\partial_t\dot H^-_{N}\\
&+\hat c\,\varphi\partial_t\dot u^-_{N}+\hat c\,\varphi\partial_t\dot H^-_{N}\,.
\end{split}
\end{equation}
We solve the first boundary condition in \eqref{BC} (we choose the $+$ side) with respect to $\partial_t\varphi$ and replace $\partial_2\varphi$ by \eqref{d2phi} to get
\begin{equation}\label{dtphi}
\partial_t\varphi=\dot u^+_N+\hat D_1\dot H^+_{N}+\hat D_2\dot H^-_N+\hat D_3\varphi\,,
\end{equation}
with suitable coefficients $\hat D_i=\hat D_i(t,x_2)$, $i=1,2,3$, smoothly depending
on the boundary traces of $\hat u^+$, $\hat H^\pm$, $\partial_1\hat H^\pm$, $\partial_1\hat u^+$ and second order derivatives of $\hat\varphi$, whose explicit form is useless for the subsequent calculations.
\newline
Now we insert \eqref{dtphi} and \eqref{d2phi} in the expression \eqref{lot_dt} to rewrite the latter as
\begin{equation}\label{lot_dt_1}
\begin{split}
\mbox{l.o.t}:=&\hat c\,\partial_t\dot{q}^+(\dot u^+_N+\hat D_1\dot H^+_{N}+\hat D_2\dot H^-_N+\hat D_3\varphi)+\hat c\,(\hat d_1\dot H^+_N+\hat d_2\dot H^-_N+\hat d_3\varphi)\partial_t\dot q^+\\
&+\hat c\,\partial_t\dot{q}^+\varphi+\hat c\,(\dot u^+_N+\hat D_1\dot H^+_{N}+\hat D_2\dot H^-_N+\hat D_3\varphi)\partial_t\dot u^-_{N}\\
&+\hat c\,(\dot u^+_N+\hat D_1\dot H^+_{N}+\hat D_2\dot H^-_N+\hat D_3\varphi)\partial_t\dot H^-_{N}+\hat c\,\varphi\partial_t\dot u^-_{N}+\hat c\,\varphi\partial_t\dot H^-_{N}\,.
\end{split}
\end{equation}
From the resulting expression above, it appears that the lower order terms above are reduced to a sum of two types of quadratic terms, namely
\[
\mbox{l.o.t}\cong\hat c\,\partial_t\alpha\,\beta+\hat c\,\partial_t\alpha\,\varphi\,,\quad\mbox{on}\,\,\,\Gamma_t\,,
\]
where $\alpha=\alpha(t,{\bf x})$ and $\beta=\beta(t,{\bf x})$ are used to mean any component of the noncharacteristic part ${\mathbf V}^\pm_n=(\dot q^\pm,\dot u^\pm_n, \dot H^\pm_n)$ of the solution, while $\hat c=\hat c(t,x_2)$ denotes, as usual, suitable functions on $\Gamma_t$, smoothly depending on space-time derivatives of $\hat{\mathbf U}^\pm$ and $\nabla_{t,x_2}\hat\varphi$ up to second order.

\smallskip
\noindent
In view of the preceding manipulations, to get an estimate of the boundary integral of the quadratic form \eqref{boundary-q-form-dt} we only need to estimate the following types of boundary integrals:
\[
\mathcal J_1:=\int_0^t\!\!\!\int_{\mathbb R}\hat c\,\partial_s\alpha\,\beta\, dx_2 ds\quad\mbox{and}\quad \mathcal J_2:=\int_0^t\!\!\!\int_{\mathbb R}\hat c\,\partial_s\alpha\,\varphi\, dx_2 ds\,.
\]
As it was done in Sect. \ref{I2_sct}, in $\mathcal J_1$ we pass to a volume integral over $\mathbb R^2_+$ and use Leibniz's rule to get
\[
\begin{split}
\mathcal J_1=&-\int_0^t\!\!\!\int_{\mathbb R^2_+}\partial_1(\hat {\mathbf c}\,\partial_s\alpha\,\beta) d{\bf x} ds=-\int_0^t\!\!\!\int_{\mathbb R^2_+}\partial_1\hat {\mathbf c}\,\partial_s\alpha\,\beta\,d{\bf x} ds-\int_0^t\!\!\!\int_{\mathbb R^2_+}\hat {\mathbf c}\,\partial_1\partial_s\alpha\,\beta\,d{\bf x} ds\\
&-\int_0^t\!\!\!\int_{\mathbb R^2_+}\hat {\mathbf c}\,\partial_s\alpha\,\partial_1\beta\,d{\bf x} ds=\mathcal J_{1,1}+\mathcal J_{1,2}+\mathcal J_{1,3}\,.
\end{split}
\]
By Cauchy--Schwarz and Young's inequalities, $\mathcal J_{1,1}$ and $\mathcal J_{1,3}$ are estimated by
\begin{equation}\label{est_J11-13}
\begin{split}
\vert \mathcal J_{1,1}+\mathcal J_{1,3}\vert &\le C_3\left\{\Vert{\mathbf V}\Vert^2_{L^2(\Omega_t)}+\Vert\partial_s{\mathbf V}\Vert^2_{L^2(\Omega_t)}+\Vert\partial_1{\mathbf V}_n\Vert^2_{L^2(\Omega_t)}\right\}\\
&\le C_3\int_0^t(I_{1,\ast}+I_{1,n})(s)ds\,.
\end{split}
\end{equation}
The middle integral $\mathcal J_{1,2}$ is the most involved one, because the second order derivative $\partial_1\partial_t\alpha$ prevents from directly estimating $\mathcal J_{1,2}$ similarly to \eqref{est_J11-13}.
\newline
From Leibniz's rule with respect to time (notice that $\alpha|_{t=0}\equiv 0$ and $\beta|_{t=0}\equiv 0$), we rewrite $\mathcal J_{1,2}$ as
\begin{equation}\label{J12}
\begin{split}
\mathcal J_{1,2}&=-\int_0^t\!\!\!\int_{\mathbb R^2_+}\partial_s(\hat {\mathbf c}\,\partial_1\alpha\,\beta)\,d{\bf x} ds+\int_0^t\!\!\!\int_{\mathbb R^2_+}\partial_s\hat {\mathbf c}\,\partial_1\alpha\,\beta\,d{\bf x} ds+\int_0^t\!\!\!\int_{\mathbb R^2_+}\hat {\mathbf c}\,\partial_1\alpha\,\partial_s\beta\,d{\bf x} ds\\
&=-\int_{\mathbb R^2_+}\hat {\mathbf c}(t)\,\partial_1\alpha(t)\,\beta(t)\,d{\bf x}+\int_0^t\!\!\!\int_{\mathbb R^2_+}\partial_s\hat {\mathbf c}\,\partial_1\alpha\,\beta\,d{\bf x} ds+\int_0^t\!\!\!\int_{\mathbb R^2_+}\hat {\mathbf c}\,\partial_1\alpha\,\partial_s\beta\,d{\bf x} ds\,.
\end{split}
\end{equation}
The last two integrals above are estimated exactly by the same right-hand side of \eqref{est_J11-13}.
\newline
Concerning, instead, the first spatial integral
$$
-\displaystyle\int_{\mathbb R^2_+}\hat {\mathbf c}(t)\,\partial_1\alpha(t)\,\beta(t)\,d{\bf x}\,,
$$
in view of estimates \eqref{en-ineq-I0}, \eqref{en-ineq-I1n} (recall that $\alpha$ is a noncharacteristic component of ${\mathbf V}$), the use of Cauchy--Schwarz and weighted Young's inequalities gives
\begin{equation}\label{est_alpha_beta(t)}
\begin{split}
&\left\vert-\int_{\mathbb R^2_+}\hat {\mathbf c}(t)\,\partial_1\alpha(t)\,\beta(t)\,d{\bf x}\right\vert\le C_2\Vert\partial_1\alpha(t)\Vert_{L^2(\mathbb R^2_+)}\Vert\beta(t)\Vert_{L^2(\mathbb R^2_+)}\\
&\le\varepsilon\Vert\partial_1\alpha(t)\Vert_{L^2(\mathbb R^2_+)}^2+\frac{C_2}{\varepsilon}\Vert\beta(t)\Vert_{L^2(\mathbb R^2_+)}^2\le\varepsilon I_{1,n}(t)+\frac{C_2}{\varepsilon}\left\{\Vert{\mathbf F}\Vert_{L^2(\Omega_t)}^2+\Vert\varphi(t)\Vert_{L^2(\mathbb R)}^2\right.\\
&\quad\left.+\int_0^t (I+I_{1,n})(s)ds+\int_0^t\Vert\varphi(s)\Vert^2_{L^2(\mathbb R)}ds\right\}\,,
\end{split}
\end{equation}
where $\varepsilon>0$ will be chosen to be small enough.
\newline
Adding \eqref{est_alpha_beta(t)} and the analogous of \eqref{est_J11-13} for the second and third integral in the right-hand side of \eqref{J12} finally yields
\begin{equation}\label{est_J12}
\begin{split}
\vert\mathcal J_{1,2}\vert\le &\varepsilon I_{1,n}(t)+\frac{C_2}{\varepsilon}\left\{\Vert{\mathbf F}\Vert_{L^2(\Omega_t)}^2+\Vert\varphi(t)\Vert_{L^2(\mathbb R)}^2+\int_0^t\Vert\varphi(s)\Vert^2_{L^2(\mathbb R)}ds\right\}\\
&+\frac{C_3}{\varepsilon}\int_0^t (I_{1,\ast}+I_{1,n})(s)ds
\end{split}
\end{equation}
and adding the latter and \eqref{est_J11-13} we get
\begin{equation}\label{est_J1}
\begin{split}
\vert\mathcal J_{1}\vert\le &\varepsilon I_{1,n}(t)+\frac{C_2}{\varepsilon}\left\{\Vert{\mathbf F}\Vert_{L^2(\Omega_t)}^2+\Vert\varphi(t)\Vert_{L^2(\mathbb R)}^2+\int_0^t\Vert\varphi(s)\Vert^2_{L^2(\mathbb R)}ds\right\}\\
&+\frac{C_3}{\varepsilon}\int_0^t (I_{1,\ast}+I_{1,n})(s)ds\,.
\end{split}
\end{equation}
The boundary integral $\mathcal J_2$ is treated along the same lines as $\mathcal J_{1,2}$. After Leibniz's rule with respect to time, we first rewrite
\[
\begin{split}
\mathcal J_2&=\int_0^t\!\!\!\int_{\mathbb R}\partial_s(\hat c\,\alpha\,\varphi\,) dx_2 ds-\int_0^t\!\!\!\int_{\mathbb R}\partial_s\hat c\,\alpha\,\varphi\, dx_2 ds-\int_0^t\!\!\!\int_{\mathbb R}\hat c\,\alpha\,\partial_s\varphi\, dx_2 ds\\
&=\int_{\mathbb R}\hat c(t)\,\alpha(t)\,\varphi(t) dx_2-\int_0^t\!\!\!\int_{\mathbb R}\partial_s\hat c\,\alpha\,\varphi\, dx_2 ds-\int_0^t\!\!\!\int_{\mathbb R}\hat c\,\alpha\,\partial_s\varphi\, dx_2 ds\\
&=\mathcal J_{2,1}+\mathcal J_{2,2}+\mathcal J_{2,3}\,.
\end{split}
\]
The middle integral $\mathcal J_{2,2}$ is estimated at once by Cauchy--Schwarz, Young's inqualities and the trace type inequality \eqref{trace-in}
\begin{equation}\label{est_J22}
\begin{split}
\vert\mathcal J_{2,2}\vert=\left\vert-\int_0^t\!\!\!\int_{\mathbb R}\partial_s\hat c\,\alpha\,\varphi\, dx_2 ds\right\vert&\le C_3\left\{\Vert{\mathbf V}\Vert_{L^2(\Omega_t)}^2+\Vert\partial_1{\mathbf V}_n\Vert_{L^2(\Omega_t)}^2+\Vert\varphi\Vert^2_{L^2(\Gamma_t)}\right\}\\
&= C_3\left\{\int_0^t(I(s)+I_{1,n}(s))ds+\int_0^t\Vert\varphi(s)\Vert^2_{L^2(\mathbb R)}ds\right\}\,.
\end{split}
\end{equation}
Concerning $\mathcal J_{2,3}$, we substitute $\partial_t\varphi$ with the right-hand side of \eqref{dtphi} to rewrite it as
\[
\mathcal J_{2,3}=-\int_0^t\!\!\!\int_{\mathbb R}\hat c\,\alpha\,(\dot u^+_N+\hat D_1\dot H^+_{N}+\hat D_2\dot H^-_N+\hat D_3\varphi)\, dx_2 ds\cong \int_0^t\!\!\!\int_{\mathbb R}\hat c\,\alpha\,\beta\,dx_2\,ds+\int_0^t\!\!\!\int_{\mathbb R}\hat c\,\alpha\,\varphi\,dx_2\,ds
\]
where the second type of boundary integral above is exactly the same as $\mathcal J_{2,2}$ and the first one is trivially estimated by using once again the trace type inequality, hence we get the same estimate \eqref{est_J22}
\begin{equation}\label{est_J23}
\begin{split}
\vert\mathcal J_{2,3}\vert\le C_3\left\{\int_0^t(I(s)+I_{1,n}(s))ds+\int_0^t\Vert\varphi(s)\Vert^2_{L^2(\mathbb R)}ds\right\}\,.
\end{split}
\end{equation}
Concerning, instead, the first spatial integral $\mathcal J_{2,1}$, the use of Cauchy--Schwarz and weighted Young's inequalities and the trace type inequality \eqref{trace_t} gives
\begin{equation}\label{est_J21}
\begin{split}
\vert\mathcal J_{2,1}\vert&\le \varepsilon\Vert\alpha(t)|_{x_1=0}\Vert_{L^2(\mathbb R)}^2+\frac{C_2}{\varepsilon}\Vert\varphi(t)\Vert^2_{L^2(\mathbb R)}\\
&\le\varepsilon\left\{\Vert\alpha(t)\Vert^2_{L^2(\mathbb R^2_+)}+\Vert\partial_1\alpha(t)\Vert^2_{L^2(\mathbb R^2_+)}\right\}+\frac{C_2}{\varepsilon}\Vert\varphi(t)\Vert^2_{L^2(\mathbb R)}\\
&\le\varepsilon\left\{I(t)+I_{1,n}(t)\right\}+\frac{C_2}{\varepsilon}\Vert\varphi(t)\Vert^2_{L^2(\mathbb R)}\,.
\end{split}
\end{equation}
Gathering \eqref{est_J21}, \eqref{est_J22} and \eqref{est_J23} we get
\begin{equation}\label{est_J2}
\begin{split}
\vert\mathcal J_{2}\vert\le &\varepsilon\left\{I(t)+I_{1,n}(t)\right\}+\frac{C_2}{\varepsilon}\Vert\varphi(t)\Vert^2_{L^2(\mathbb R)}\\
&+C_3\left\{\int_0^t(I(s)+I_{1,n}(s))ds+\int_0^t\Vert\varphi(s)\Vert^2_{L^2(\mathbb R)}ds\right\}\,.
\end{split}
\end{equation}
Summing the estimate of $\mathcal J_1$ and $\mathcal J_2$ above, we end up with the following estimate of the boundary integral of the quadratic form \eqref{boundary-q-form-dt}
\begin{equation}\label{stima_bdry_t}
\begin{split}
&\Big|\int_0^t\int_{\mathbb R}(\mathcal{B}_1{\mathbf V}_{s},{\mathbf V}_{s})|_{x_1=0}dx_2\,ds\Big|\cong\vert\mathcal J_1+\mathcal J_2\vert\\
&\quad\le\varepsilon\left\{I(t)+I_{1,n}(t)\right\}+\frac{C_2}{\varepsilon}\left\{\Vert{\mathbf F}\Vert_{L^2(\Omega_t)}^2+\Vert\varphi(t)\Vert_{L^2(\mathbb R)}^2\right\}\\
&\qquad+\frac{C_3}{\varepsilon}\left\{\int_0^t (I_{1,\ast}+I_{1,n})(s)ds+\int_0^t\Vert\varphi(s)\Vert^2_{L^2(\mathbb R)}ds\right\}\,.
\end{split}
\end{equation}
Using \eqref{est_0_term_t}, \eqref{est_Ft} and \eqref{stima_bdry_t} together with the counterpart of \eqref{est_f_b} with ${\mathbf V}_{t}$ instead of ${\mathbf V}$, from \eqref{en-id-It} we obtain
\begin{equation}\label{en-ineq-It}
\begin{split}
I_0(t)\le &\frac{C_3}{\varepsilon}\left\{\Vert{\mathbf F}\Vert^2_{L^2(\Omega_t)}+\Vert\partial_s{\mathbf F}\Vert_{L^2(\Omega_t)}^2+\int_0^t(I_{1,\ast}+I_{1,n})(s)ds+\int_0^t\Vert\varphi(s)\Vert_{L^2(\mathbb R)}^2 ds\right\}\\
&+\varepsilon\left\{I(t)+I_{1,n}(t)\right\}+\frac{C_2}{\varepsilon}\Vert\varphi(t)\Vert_{L^2(\mathbb R)}^2\,.
\end{split}
\end{equation}
%%%%%%%%%%%%%%%%%%%%%%%%%%%%%%%%%%%%%%%%%%%%%%%%%%%%%%%%%%%%%%%%%%%%%%%%%%
\subsection{Estimate of the front $\varphi$}\label{ front_sct} To estimate the $L^2-$norm of the front we multiply by $\varphi$ the first boundary condition in \eqref{BC} (we choose the $+$ side) and integrate by parts over $\Gamma_t$ to get
\begin{equation}\label{en-id-phi}
\int_{\mathbb R}\vert\varphi(t)\vert^2dx_2-\int_0^t\!\!\!\int_{\mathbb R}(\partial_2\hat u^+_2)\varphi^2\,dx_2\,ds-2\int_0^t\!\!\!\int_{\mathbb R}\dot u^+_N\varphi\,dx_2\,ds-2\int_0^t\!\!\!\int_{\mathbb R}(\partial_1\hat u^+_N)\varphi^2\,dx_2\,ds=0\,,
\end{equation}
from which, by Cauchy--Schwarz and Young's inequalities, and the trace type inequality \eqref{trace-in}
\begin{equation}\label{est-phi}
\begin{split}
\Vert\varphi(t)\Vert^2_{L^2(\mathbb R)}&\le C_1\int_0^t\Vert\varphi(s)\Vert^2_{L^2(\mathbb R)}\,ds+\Vert\dot u^+_N\vert_{x_1=0}\Vert^2_{L^2(\Gamma_t)}\\
&\le C_1\int_0^t\Vert\varphi(s)\Vert^2_{L^2(\mathbb R)}\,ds+\Vert\partial_1{\mathbf V}_n^+\Vert^2_{L^2(\Omega_t)}+\Vert{\mathbf V}^+\Vert^2_{L^2(\Omega_t)}\\
&\le C_1\int_0^t\Vert\varphi(s)\Vert^2_{L^2(\mathbb R)}\,ds+\int_0^t(I_{1,n}+I)(s)ds\,.
\end{split}
\end{equation}
Using the expressions of $\partial_2\varphi$ and $\partial_t\varphi$ in \eqref{d2phi} and \eqref{dtphi} we can recover an estimate of the $L^2-$norms of those derivatives from the estimate \eqref{est-phi} and \eqref{trace-in}. Precisely, integrating over $\mathbb R$  \eqref{d2phi} and \eqref{dtphi} and using \eqref{trace-in} we get
\begin{equation}\label{est_d2phi}
\begin{split}
\Vert\partial_2\varphi(t)\Vert^2_{L^2(\mathbb R)}&\le C_2\{\Vert{\mathbf V}_n(t)\vert_{x_1=0}\Vert^2_{L^2(\mathbb R)}+\Vert\varphi(t)\Vert^2_{L^2(\mathbb R)}\}\\
&\le C_2\{I(t)+I_{1,n}(t)+\Vert\varphi(t)\Vert^2_{L^2(\mathbb R)}\}\,,
\end{split}
\end{equation}
see \eqref{I0}, \eqref{I1n}, and analogously
\begin{equation}\label{est_dtphi}
\begin{split}
\Vert\partial_t\varphi(t)\Vert^2_{L^2(\mathbb R)}\le C_2\{I(t)+I_{1,n}(t)+\Vert\varphi(t)\Vert^2_{L^2(\mathbb R)}\}\,.
\end{split}
\end{equation}
%%%%%%%%%%%%%%%%%%%%%%%%%%%%%%%%%%%%%%%%%%%%%%%%%%%%%%%%%%%%%%%%%%%%%%%%
\subsection{$H^1_\ast-$estimate}\label{H1ast_sct}
We add estimates \eqref{en-ineq-I0}, \eqref{en-ineq-Is}, \eqref{en-ineq-I2}, \eqref{en-ineq-It} and \eqref{est-phi} to get
\[
\begin{split}
I_{1,\ast}(t)+\Vert\varphi(t)\Vert^2_{L^2(\mathbb R)}&=I(t)+I_0(t)+I_{\sigma}(t)+I_{2}(t)+\Vert\varphi(t)\Vert^2_{L^2(\mathbb R)}\\
&\le \frac{C_3}{\varepsilon}\int_0^t(I_{1,\ast}(s)+I_{1,n}(s)+\Vert\varphi(s)\Vert^2_{L^2(\mathbb R)})ds\\
&\quad+\varepsilon\left\{I(t)+I_{1,n}(t)\right\}+\frac{C_2}{\varepsilon}\Vert\varphi(t)\Vert_{L^2(\mathbb R)}^2\\
&\quad+\frac{C_3}{\varepsilon}\left\{\Vert{\mathbf F}\Vert^2_{L^2(\Omega_t)}+\Vert\sigma\partial_1{\mathbf F}\Vert_{L^2(\Omega_t)}^2+\Vert\partial_2{\mathbf F}\Vert_{L^2(\Omega_t)}^2+\Vert\partial_s{\mathbf F}\Vert_{L^2(\Omega_t)}^2\right\}\\
&\le \frac{C_3}{\varepsilon}\int_0^t(I_{1,\ast}(s)+I_{1,n}(s)+\Vert\varphi(s)\Vert^2_{L^2(\mathbb R)})ds\\
&\quad+\varepsilon\left\{I(t)+I_{1,n}(t)\right\}+\frac{C_2}{\varepsilon}\Vert\varphi(t)\Vert_{L^2(\mathbb R)}^2+\frac{C_3}{\varepsilon}\Vert{\mathbf F}\Vert^2_{H^1_\ast(\Omega_t)}\,.
\end{split}
\]
Then using \eqref{est-phi} to estimate $\Vert\varphi(t)\Vert_{L^2(\mathbb R)}^2$ in the right-hand side above, the previous inequality reduces to
\[
\begin{split}
I_{1,\ast}(t)+\Vert\varphi(t)\Vert^2_{L^2(\mathbb R)}\le &\frac{C_3}{\varepsilon}\int_0^t(I_{1,\ast}(s)+I_{1,n}(s)+\Vert\varphi(s)\Vert^2_{L^2(\mathbb R)})ds\\
&+\varepsilon\left\{I(t)+I_{1,n}(t)\right\}+\frac{C_3}{\varepsilon}\Vert{\mathbf F}\Vert^2_{H^1_\ast(\Omega_t)}\,.
\end{split}
\]
Then we use \eqref{en-ineq-I1n} to get a control of the normal derivatives of ${\mathbf V}^\pm_n$ involved in the term $I_{1,n}$ in the right-hand side above to get
\[
\begin{split}
I_{1,\ast}(t)+\Vert\varphi(t)\Vert^2_{L^2(\mathbb R)}\le &\frac{C_3}{\varepsilon}\int_0^t(I_{1,\ast}(s)+\Vert\varphi(s)\Vert^2_{L^2(\mathbb R)})ds\\
&+C_1\varepsilon\left\{I_{1,\ast}(t)+\Vert{\mathbf F}(t)\Vert^2_{L^2(\mathbb R^2_+)}\right\}+\frac{C_3}{\varepsilon}\Vert{\mathbf F}\Vert^2_{H^1_\ast(\Omega_t)}\,.
\end{split}
\]
In order to estimate the spatial $L^2-$norm of the source ${\mathbf F}(t)$ in the right-hand side above we use the following argument
\[
\begin{split}
\Vert{\mathbf F}(t)\Vert^2_{L^2(\mathbb R^2_+)}&=\int_{\mathbb R^2_+}\vert {\mathbf F}(t)\vert^2 d{\bf x}=\int_{\mathbb R^2_+}\int_0^t\partial_s(\vert{\mathbf F}(s)\vert^2)ds\,d{\bf x}=\int_{\mathbb R^2_+}\int_0^t2{\mathbf F}(s)\partial_s{\mathbf F}(s)\,ds\,d{\bf x}\\
&\le\int_{\mathbb R^2_+}\int_0^t\vert{\mathbf F}(s)\vert^2\,ds\,d{\bf x}+\int_{\mathbb R^2_+}\int_0^t\vert\partial_s{\mathbf F}(s)\vert^2\,ds\,d{\bf x}=\Vert{\mathbf F}\Vert^2_{L^2(\Omega_t)}+\Vert\partial_s{\mathbf F}\Vert^2_{L^2(\Omega_t)}\,,
\end{split}
\]
so that
\[
\begin{split}
I_{1,\ast}(t)+\Vert\varphi(t)\Vert^2_{L^2(\mathbb R)}\le &\frac{C_3}{\varepsilon}\int_0^t(I_{1,\ast}(s)+\Vert\varphi(s)\Vert^2_{L^2(\mathbb R)})ds+C_1\varepsilon I_{1,\ast}(t)+\frac{C_3}{\varepsilon}\Vert{\mathbf F}\Vert^2_{H^1_\ast(\Omega_t)}\,,
\end{split}
\]
hence taking $\varepsilon>0$ small enough in order to absorb $C_1\varepsilon I_{1,\ast}(t)$ in the left-hand side above, we end up with
\[
\begin{split}
I_{1,\ast}(t)+\Vert\varphi(t)\Vert^2_{L^2(\mathbb R)}\le & C_3\int_0^t(I_{1,\ast}(s)+\Vert\varphi(s)\Vert^2_{L^2(\mathbb R)})ds+C_3\Vert{\mathbf F}\Vert^2_{H^1_\ast(\Omega_t)}\,.
\end{split}
\]
Applying Gr\"onwall's lemma, we obtain
\[
I_{1,\ast}(t)+\Vert\varphi(t)\Vert^2_{L^2(\mathbb R)}\le C_3\Vert{\mathbf F}\Vert^2_{H^1_\ast(\Omega_t)}e^{C_3t}
\]
and integrating on $t$ over $(0,T)$
\begin{equation}\label{stima_H1ast+L2phi}
\Vert{\mathbf V}\Vert^2_{1,\ast,T}+\Vert\varphi\Vert^2_{L^2(\Gamma_T)}\le C\Vert{\mathbf F}\Vert^2_{H^1_\ast(\Omega_T)}\,,
\end{equation}
where $C$ is a positive constant depending only on $T$ and $K$ from \eqref{space}.
\newline
From the use of estimates \eqref{est_d2phi}, \eqref{est_dtphi} we can manage to include the $L^2(\Gamma_T)-$norms of $\nabla_{t,x_2}\varphi$ in the left-hand side of \eqref{stima_H1ast+L2phi}, so as to get a control of the $H^1(\Gamma_T)-$norm of the front $\varphi$.
\newline
From integration of \eqref{est_d2phi}, \eqref{est_dtphi} with respect to $t$ over $(0,T)$ and the use of \eqref{en-ineq-I1n} we get
\[
\begin{split}
\Vert\nabla_{t,x_2}\varphi\Vert^2_{L^2(\Gamma_T)}&\le C_2\left\{\int_0^T(I(t)+I_{1,n}(t))dt+\int_0^T\Vert\varphi(t)\Vert^2_{L^2(\mathbb R)}dt\right\}\\
&\le C_2\left\{\int_0^T I_{1,\ast}(t)\,dt+\int_0^T\Vert{\mathbf F}(t)\Vert^2_{L^2(\mathbb R^2_+)}dt+\int_0^T\Vert\varphi(t)\Vert^2_{L^2(\mathbb R)}dt\right\}\\
&\le C_2\left\{\Vert{\mathbf V}\Vert^2_{1,\ast,T}+\Vert{\mathbf F}\Vert^2_{L^2(\Omega_T)}+\Vert\varphi\Vert^2_{L^2(\Gamma_T)}\right\}\,,
\end{split}
\]
where, in the last inequality, it is used that
$$
\int_0^T I_{1,\ast}(t)\,dt\cong\Vert{\mathbf V}\Vert^2_{1,\ast,T}\,.
$$
Using \eqref{stima_H1ast+L2phi} to estimate the right-hand side above, we end up with
\begin{equation}\label{stima_L2nablaphi}
\begin{split}
\Vert\nabla_{t,x_2}\varphi\Vert^2_{L^2(\Gamma_T)}\le C\Vert{\mathbf F}\Vert^2_{H^1_\ast(\Omega_T)}\,.
\end{split}
\end{equation}
Then adding the latter to \eqref{stima_H1ast+L2phi} we get
\begin{equation}\label{stima_H1ast+H1phi}
\Vert{\mathbf V}\Vert^2_{1,\ast,T}+\Vert\varphi\Vert^2_{H^1(\Gamma_T)}\le C\Vert{\mathbf F}\Vert^2_{H^1_\ast(\Omega_T)}\,,
\end{equation}
which is just the estimate \eqref{est-wp-hom}, in view of \eqref{V} (recalling that to shortcut notation we have set $\dot{\mathbf U}=\dot{\mathbf U}^\natural$ and $\dot{\mathbf U}=J{\mathbf V}$, see \eqref{V}).

%%%%%%%%%%%%%%%%%%%%%%%%%%%%%%%%%%%%%%%%%%%%%%%%%%%%%%%%%%%5da cancellare%%%%%%%%%%%%%%%%%%%%%%%%
%%%%%%%%%%%%%%%%%%%%%%%%%%%%%%%%%%%%%%%%%%%%%%%%%%%%%%%%%%%%%%%%%%%%%%%%%%%%%%%%%%%%%%%%%%%%%

%%%%%%%%%%
%

\bigskip

\section*{Acknowledgments}

The authors are grateful to the anonymous referee for her/his comments and suggestions, that have contributed to improve this paper.

The research of A. Morando, P. Secchi, P. Trebeschi was supported in part by the Italian MUR Project PRIN prot. 20204NT8W4.
 D. Yuan was supported by NSFC Grant No.12001045 and China Postdoctoral Science Foundation No.2020M680428, No.2021T140063. D. Yuan thanks the Department of Mathematics of the University of Brescia for its kind hospitality.

%%%%%%%%%%
%
\section*{Statements and Declarations}
This research does not have any associated data.

%%%%%%%%%%%%%%%%%%%%%%%%%%%%%%%%%%%%%%%%%%%%%%

\end{document}